\documentclass[11pt]{article}

  \usepackage{latexsym}
  \usepackage{comment}
  \usepackage{amsmath,amsthm,amsfonts,amssymb,graphicx,epsfig,latexsym,color,amssymb}

  \usepackage{epsf,enumerate}
  \usepackage{tikz-cd}
  \usepackage{mathtools,float}
  \usepackage[normalem]{ulem}
  \usepackage{adjustbox}
  \usepackage{hyperref,appendix}
  
  \usepackage[margin=1.25in]{geometry}

  \newtheorem{thm}{Theorem}[section]
  \newtheorem{lemma}[thm]{Lemma}
  
  \newtheorem{cor}[thm]{Corollary}
  \newtheorem{prop}[thm]{Proposition}
  \newtheorem*{main}{Main Theorem}{}

  \theoremstyle{remark}

  \newtheorem{definition}[thm]{Definition}
  \newtheorem{remark}[thm]{Remark}
  
  \newtheorem*{definition*}{Definition}
  \newtheorem*{remark*}{Remark}

  \def\R{{\mathbb R}}

  \def\Z{{\mathbb Z}}
  \def\N{{\mathbb N}}

  \def\S{{\mathcal S}}
  
  \def\t{{\tau}}
  
  \def\F{{\mathcal F}}

  \def\diam{\operatorname{diam}}

  \def\A{{\mathcal A}}
  
  \def\ML{{\mathcal{ML}}}
    \def\PML{{\mathcal{PML}}}
  \def\C{{\mathcal C}}
  \def\Mod{\operatorname{Mod}}
  \def\MCG{\operatorname{MCG}}
  \def\asdim{\operatorname{asdim}}
  \def\Out{\operatorname{Out}}
  \def\ind{\operatorname{ind}}
  \def\ss{{\overset{ss}\longrightarrow}}
    \def\T{{\mathcal T}}

\newcommand{\bS}{{\bf S}}
\newcommand{\ATL}{\operatorname{ATL}}
\newcommand{\cF}{{\cal F}}
\newcommand{\cI}{{\mathcal I}}
\newcommand{\cP}{{\mathcal P}}

\date \today
\title{Disintegrating the curve complex}
\author{Mladen Bestvina, Kenneth Bromberg, and Alexander J. Rasmussen}
\begin{document}
\maketitle

\begin{abstract}
We study a finite sequence of graphs, beginning with the curve graph and ending
with a graph quasi-isometric to a tree. There is a Lipschitz map from one
graph in the sequence to the next. This sequence was first introduced by
Hamenst\"{a}dt. We prove (as conjectured by Hamenst\"{a}dt) that the graphs in this
sequence are hyperbolic and that the coarse fibers of the maps in the sequence
are quasi-trees. This gives an upper bound on the asymptotic dimension of each
graph in the sequence and as a result, an upper bound on the asymptotic dimension
of the curve graph. Additionally, we show that the action of the mapping class
group on each graph in the sequence is acylindrical, and classify the boundary
and actions of individual mapping classes for each graph in the sequence.
\end{abstract}

\section{Introduction}

The curve graph $\C(\Sigma)$ of a finite type orientable surface $\Sigma$
is a crucial tool for studying the coarse geometry of the mapping class group
$\Mod(\Sigma)$. The graphs $\C=\C(\Sigma)$ were introduced by Harvey in \cite{Harvey}
as an analog of Tits buildings for mapping class groups. From the perspective
of coarse geometry, the most important property of $\C(\Sigma)$ is that it is
an infinite diameter \emph{hyperbolic} graph, as proved by Masur-Minsky in \cite{MM}.
Our purpose in this paper is to study $\C(\Sigma)$ via
``disintegration.'' Our main result is the following. Some of the terms in the Main
Theorem will be explained below.
See the numbered references in parentheses for precise statements of results.

\begin{main}
\label{bigthm}
  There is a tower of graphs:
  $$\C_{m(\Sigma)}(\Sigma)\to \C_{m(\Sigma)-1}(\Sigma)\to\cdots\to\C_0(\Sigma)$$
  with the following properties:
  \begin{itemize}
    \item $\C_{m(\Sigma)}(\Sigma)$ is quasi-isometric to $\C(\Sigma)$
          (Section \ref{sec:2}).
    \item The mapping class group $\Mod(\Sigma)$ acts on $\C_k(\Sigma)$
          by isometries (Section \ref{sec:2}).
          Moreover, $\C_k(\Sigma)$ is hyperbolic and has infinite diameter
          (Proposition \ref{hyperbolic}).
    \item $m(\Sigma)$ is a linear function of the complexity of
          $\Sigma$
          (Section \ref{sec:2}).
    \item The graph $\C_0(\Sigma)$ at the bottom of the tower is a
          quasi-tree (Theorem \ref{mainthm}).
    \item Each map $\C_{k+1}(\Sigma)\to\C_k(\Sigma)$ in the tower is 1-Lipschitz,
          coarsely surjective, alignment-preserving, and $\Mod(\Sigma)$-equivariant
          (Section \ref{sec:2}).
    \item The coarse fibers of each $\C_{k+1}(\Sigma)\to\C_k(\Sigma)$ are
          quasi-trees and are quasi-convex in $\C_{k+1}(\Sigma)$ (Theorem \ref{mainthm}).
    \item The action of $\Mod(\Sigma)$ on each $\C_k(\Sigma)$ is
          acylindrical (and in particular there is a positive
          lower bound on the asymptotic translation lengths of loxodromic
          elements of $\Mod(\Sigma)$ on $\C_k(\Sigma)$) (Theorem \ref{allacyl}).
    \item The Gromov boundary of each $\C_k(\Sigma)$ is a certain
          explicit subset of the space of ending laminations on
          $\Sigma$ (which is the Gromov boundary of $\C(\Sigma)$,
          \cite{klarreich}) (Theorem \ref{boundary}).
    \item There is an explicit classification of which elements of
          $\Mod(\Sigma)$ act loxodromically on $\C_k(\Sigma)$. All other
          elements act elliptically (Theorem \ref{loxclassification}).
  \end{itemize}
\end{main}

\noindent See Section \ref{sec:2} for the definition of the graphs $\C_k$
and the definition of coarse fibers.

An important part of the Main Theorem (specifically,
Theorem \ref{mainthm}) was first announced by Hamenst\"{a}dt in a
very inspiring lecture \cite{ursula-video} at MSRI in Fall
2016.
However, a proof
has not yet been circulated.

The graph $\C_0(\Sigma)$ has been studied before, by Maher-Masai-Schleimer
\cite{MMS2} and by Hamenst\"adt \cite{HamenstadtPrincipal,ursula_random}
(who called it the \emph{principal
  curve graph}). The definitions of $\C_0(\Sigma)$ differ slightly in the two
papers, but yield quasi-isometric spaces. We will adopt the analog of the
Maher-Masai-Schleimer definition as our working definition, but we
discuss the relationship with Hamenst\"adt's definition; see Section
\ref{sec:3}. The two papers together prove hyperbolicity and infinite diameter
of $\C_0(\Sigma)$, identify the Gromov boundary and loxodromic elements,
show that $\C(\Sigma)\to\C_0(\Sigma)$ is
Lipschitz and alignment-preserving, and that the action
is weakly properly discontinuous (\emph{WPD}) (see \cite{BF} for the definition
of WPD; the definition of \emph{alignment-preserving} is given below). Our
contribution to the study of $\C_0(\Sigma)$ is that it is a quasi-tree,
and that the action is acylindrical
(see \cite{bowditch:tight} for the definition of acylindricity).
We also prove the corresponding
results for all $\C_k(\Sigma)$. The most difficult part of the Main Theorem is the
proof that the coarse fibers of each map are quasi-trees; see Section \ref{sec:7}.
The workhorse in
this proof is an analog of the Masur-Minsky Nesting Lemma \cite[Lemma 4.7]{MM} giving a
criterion for when a splitting sequence of train tracks makes progress in
$\C_k(\Sigma)$. See Section \ref{sec:progress}.

We point out that this tower immediately implies
a bound on asymptotic dimensions:
\begin{cor}
  For each $k$, $\asdim \C_k(\Sigma) \leq k+1$. In particular,
  $\asdim \C(\Sigma) \leq m(\Sigma) + 1$.
\end{cor}
\noindent This follows from the Bell-Dranishnikov theorem
\cite{Hurewicz} about asymptotic dimension and induction on $k$.
Thus the tower can be viewed as a strong
witness of the fact that $\asdim\C(\Sigma)<\infty$.

To put this work in context, let us briefly describe some
disintegration results, for other hyperbolic complexes. In \cite{FS},
Fujiwara-Schleimer constructed a tower interpolating between the arc
and the curve graph of (bordered or punctured) $\Sigma$ with coarse
fibers which are quasi-trees of curve graphs of subsurfaces.
A \emph{witness}
$W$ is an essential subsurface that intersects every arc.
There is a notion
of complexity $\xi(W)$, taking on positive integer values, so that $V\subset W$
implies $\xi(V)\leq \xi(W)$ and equality holds exactly if $V$ and $W$ are
isotopic. Now let $X_0=\mathcal A$ be the arc graph of the surface
$\Sigma$. For $c>0$ define $X_c$ by adding a cone point over
the set of arcs contained
in a witness $W$ of complexity $\leq c$, for each such witness $W$.
This yields a tower of 1-Lipschitz maps
$$\mathcal A=X_0\to X_1\to\cdots\to X_{\xi(\Sigma)-1}.$$ It turns out
that the last graph $X_{\xi(S)-1}$ is quasi-isometric to the curve
graph $\C(\Sigma)$, so the tower yields an interpolation between the arc
and the curve graphs. Moreover, each $X_c$ is hyperbolic and each
$X_c\to X_{c+1}$ is obtained by coning off a collection
of pairwise disjoint quasi-convex subsets with bounded closest-point
projections to
each other, each of which is quasi-isometric to the curve graph of a witness
of complexity $c+1$. Fujiwara-Schleimer observe that this structure
implies that the arc graph has finite asymptotic dimension, by using
\cite{Hurewicz} and the fact that curve graphs have finite asymptotic
dimension. In \cite{bf:boundary} the first author and Feighn use this structure in
a similar inductive fashion to give an alternative description of the
Gromov boundary of the arc graph, originally due to Schleimer and
Pho-On \cite{PhoOn}.

There is an analogous construction for $\Out(F_n)$-complexes. In
\cite{bf:boundary}, the first author and Feighn construct a tower interpolating
between the free splitting and cyclic splitting complexes of the free
group $F_n$. All complexes in the tower are hyperbolic and the maps are
Lipschitz and $\Out(F_n)$-equivariant. The coarse fibers are
quasi-trees of certain relative cyclic splitting
complexes. The existence of interpolating towers between cyclic
splitting and free factor complexes, and between free factor complexes
and quasi-trees is still open.

In both the arc-to-curve-graph and free-to-cyclic-splitting-complex towers,
each map is coning off of a collection of
quasi-convex sets with uniformly bounded projections to each
other. The coarse fibers can then be understood using the projection
complex \cite{BBF} constructed from this collection. That is \emph{not} the
case in our Main Theorem. The sets being coned off at each stage,
while quasi-convex, typically have infinite diameter
intersections, and one such set may in fact be properly
contained in another.

We now say a bit more about some terms in
the Main Theorem that may be less familiar to the reader.

Given a train track $\tau\subset\Sigma$, the set of projective
measured laminations it carries is a cell in the space $\PML(\Sigma)$
of all projective measured laminations, which is a sphere. The
codimension of this cell is the {\it index} $\ind(\tau)$ of $\tau$. We
define $\C_k(\Sigma)$ by starting with $\C(\Sigma)$ and adding edges
between curves $\alpha,\beta$ provided they are both carried by a
train track of index $>k$. If $m(\Sigma)$ is the smallest number such
that train tracks of index $>m(\Sigma)$ are necessarily non-filling,
then $\C_{m(\Sigma)}(\Sigma)$ is quasi-isometric to $\C(\Sigma)$. See Section
\ref{sec:2}.

A triple $x,y,z$ in a metric space $X$
is called $K$-aligned if $d(x,y)+d(y,z) \leq d(x,z) + K$.
We say that a Lipschitz map $f:X\to Y$ is \emph{alignment-preserving} if
there exists $K>0$ such that whenever $x,y,z$ is $0$-aligned, $f(x),f(y),f(z)$
is $K$-aligned. If $X$ and $Y$ are hyperbolic, then an
alignment-preserving map $f$ sends
re-parametrized quasi-geodesics
to re-parametrized quasi-geodesics.

The statement that the coarse fibers of $\C_{k+1}(\Sigma) \to\C_k(\Sigma)$ are
quasi-convex quasi-trees means that for $B\subset \C_k(\Sigma)$ bounded,
its preimage in $\C_{k+1}(\Sigma)$ is contained in a quasi-convex set
which is quasi-isometric to a tree. The quasi-isometry constants
here necessarily deteriorate with the diameter of $B$.

Recall that an \emph{ending lamination} is a filling geodesic
lamination $\Lambda$ which is the support of a transverse measure. The space of
ending laminations is identified with the Gromov boundary $\partial\C(\Sigma)$
\cite{klarreich}.
The \emph{index} $\ind(\Lambda)$ of $\Lambda$ is the
largest index of a train track that carries $\Lambda$. The Gromov
boundary $\partial\C_k(\Sigma)$ of $\C_k(\Sigma)$ can then be identified with the
space of ending laminations of index $\leq k$ with the
coarse Hausdorff topology, which is the quotient of the topology on $\PML(\Sigma)$
(Theorem \ref{boundary}).

An element $\phi\in \Mod(\Sigma)$ is loxodromic in the action on $\C_k(\Sigma)$ if and only
if it is pseudo-Anosov and its (un)stable lamination has index $\leq
  k$ (Theorem \ref{loxclassification}).
Thus going down the tower, the set of loxodromic elements decreases, as does
the Gromov boundary.

We note that if $X$ is a hyperbolic metric space with finite asymptotic dimension, then
there is an analogous filtration
\[
  X=X_n \to X_{n-1} \to \cdots \to X_0
\]
by hyperbolic graphs, where the coarse fibers are quasi-trees and $n$ is the
asymptotic dimension of $X$. This follows from the quasi-median quasi-isometric
embedding of $X$ in a product of quasi-trees
proven by Petyt recently in \cite{Petyt}. Furthermore, Petyt-Spriano-Zalloum
recently proved in \cite[Theorem 1.3]{stablecylinders} that $\C(\Sigma)$ equivariantly
quasi-isometrically embeds in a product of quasi-trees. An immediate corollary
is that there is a tower of graphs for $\C(\Sigma)$ for which coarse fibers
are quasi-trees. This recovers certain parts of the main theorem.
We compare \cite[Theorem 1.3]{stablecylinders} and the main theorem.
On one hand, our main theorem does not prove equivariant quasi-isometric
embedding of $\C(\Sigma)$ into a finite product of quasi-trees.
On the other hand \cite[Theorem 1.3]{stablecylinders} does not give a
bound on the length of a tower of hyperbolic graphs as in our main theorem
and thus cannot be used to bound $\asdim \C(\Sigma)$.
The main advantage of our main theorem is that it produces a tower of
natural hyperbolic graphs with very explicit topological definitions.
These topological definitions yield detailed information on the boundaries
of the associated hyperbolic graphs, the dynamics of their mapping class group
actions, and the structure of their coarse fibers.

Finally, we suggest some ideas for further study. The filtration of $\C(\Sigma)$
and the techniques of this paper may be useful for the following purposes. Firstly,
they may be useful for precisely computing the asymptotic dimensions of curve graphs
and mapping class groups. If one could prove that the asymptotic dimension
of $\C_{k+1}(\Sigma)$ is greater than that of $\C_k(\Sigma)$ for $k<m(\Sigma)$,
this would show that
the asymptotic dimension of $\C(\Sigma)$ is exactly $m(\Sigma)+1$.
One could hope to use this in turn to compute $\asdim \Mod(\Sigma)$ exactly.
It is an interesting question whether $\asdim \Mod(\Sigma)$ is equal to the
virtual cohomological dimension of $\Mod(\Sigma)$, or indeed whether there exists any
finitely generated group with finite asymptotic dimension and finite virtual cohomological
dimension which are unequal. The current best known bounds on $\asdim \Mod(\Sigma)$
are quadratic in the complexity of $\Sigma$ \cite{BBF, BHS} while
$\operatorname{vcd} \Mod(\Sigma)$ is linear in the complexity \cite{vcd},
so there is currently a mysterious discrepancy between the two bounds.

The filtration of $\C(\Sigma)$ also gives a route for
proving geometric properties of curve graphs and subgroups of mapping class
groups by \emph{induction} on the filtration. For instance, one could attempt to use
induction to produce quasi-isometric embeddings of hyperbolic planes
in $\C(\Sigma)$ (see \cite{LS} for previous work on this problem)
or to study the structure of certain subgroups of mapping class groups.
As an example, if a subgroup quasi-isometrically embeds in $\C_0(\Sigma)$
under the orbit map then it must be virtually free. It would be interesting
to know e.g. if Veech groups of quadratic differentials with vertical and
horizontal laminations in $\partial \C_k(\Sigma)$ quasi-isometrically embed
in $\C_k(\Sigma)$. 

\subsection{Acknowledgments}

The authors thank Ursula Hamenst\"{a}dt for an inspiring lecture that
gave rise to this project. The first author was partially supported by
NSF grant DMS-2304774 and by a grant from the Simons
Foundation (SFI-MPS-SFM-00011601). The second author was partially
supported by NSF grants DMS-1906095 and DMS-2405104 and a grant from the Simons foundation (SFI-MPS-SFM-00006400). The third author was partially
supported by NSF grants DMS-2202986 and DMS-1840190.

\section{The tower}\label{sec:2}

Let $\Sigma$ be a connected orientable surface of finite type, without
boundary.  We will always assume that $\Sigma$ is not sporadic,
i.e. it is not a sphere with $\leq 4$ punctures or a torus with $\leq
1$ puncture. The goal of this section is to define a tower of graphs
with the curve graph $\C(\Sigma)$ at the top and the \emph{principal
curve graph} at the bottom.

We will consider train tracks on $\Sigma$.
We recall a few basic definitions about train tracks, and assume that the reader
is familiar with the basic theory of train tracks. We refer the reader to
\cite{Penner} for further background.

The train tracks we consider may not be connected,
but unless stated otherwise they will be trivalent (that is, {\it generic}
in the terminology of \cite{Penner}).
A train track will always be assumed not to have
(smooth) disks, (smooth) annuli, (smooth) punctured disks,
monogons, or bigons as complementary components.

If $\tau$ is a train track, then the edges of $\tau$ are called \emph{branches}
and the vertices are called \emph{switches}. A \emph{carrying map} from a
geodesic lamination or train track $\sigma$ to $\tau$ is a map
$f:\Sigma\to\Sigma$ such that $f(\sigma)\subset \tau$, $f$ is homotopic to the
identity, and the derivative of $f$ on each tangent line to $\sigma$ is injective.
We will denote a carrying map from $\sigma$ to $\tau$ by $\sigma\to \tau$,
leaving the notation for the map itself implicit. A \emph{train path} is a smooth
immersion of an interval in $\tau$. If $f:\sigma \to \tau$ is a carrying map
between train tracks then the image of any branch of $\sigma$ is a train path
on $\tau$. Recall that $\tau$ is \emph{recurrent} if, for every branch $b$ of
$\tau$, there is a closed curve carried by $\tau$ whose image passes through
$b$.
We will generally assume that train tracks are recurrent,
but we will make this hypothesis explicit
wherever necessary.
Additionally, we will sometimes require train tracks to be \emph{transversely recurrent}.
The definition of transverse recurrence is given in \cite{Penner}. A track
which is both recurrent and transversely recurrent is called \emph{birecurrent}.

The set $M(\tau)$ of measured laminations carried by $\tau$ is a polyhedral cone
in Thurston's measured lamination space $\ML(\Sigma)$.
Namely, it is a cone over $P(\tau)$, the projectivization of $M(\tau)$, which has the structure of a compact polyhedron.
We denote by $\S(\tau)$ the set of simple closed curves carried by $\tau$.
The vertices of $P(\tau)$ are represented by simple closed curves lying in
$\S(\tau)$, which we call \emph{vertex cycles}. We denote the set of vertex
cycles of $\tau$ by $V(\tau) \subset \S(\tau)$.
According to
\cite[Proposition 3.11.3]{Mosher}, a vertex cycle is either an embedded simple
closed curve in $\tau$ or a barbell. In particular, a vertex cycle of
$\tau$ traverses each branch of $\tau$ at most twice.
By \cite[Corollary 2.3]{HamenstadtTT}, the diameter of $V(\tau)$
in $\C(\Sigma)$ is bounded in terms of the track $\tau$ (and hence the diameter is
also bounded only in terms of $\Sigma$, since there are finitely many train tracks
on $\Sigma$ up to the mapping class group action).
Note that if
$\sigma$ and $\tau$ are train tracks and there is a carrying map $\sigma\to \tau$,
then there is an inclusion $\S(\sigma) \subset \S(\tau)$.

We define the \emph{index} of $\tau$ as
$$\ind(\tau)\vcentcolon=\dim \ML(\Sigma)-\dim M(\tau)$$ i.e. as the codimension of
$M(\tau)$ in $\ML$. If $\tau$ is
recurrent and $\sigma\subset\tau$ is a proper subtrack, then
$M(\sigma)$ is a proper face of $M(\tau)$,
and in particular $\ind(\sigma)>\ind(\tau)$.

We next give a combinatorial description of $\ind(\tau)$.
Recall that a train track $\tau$ is \emph{orientable}
if every branch can be oriented so that every carried
curve has an orientation that agrees with the orientation of every branch it traverses.

\begin{prop}[{\cite[Lemma 3.1]{BB}, \cite[Lemma 2.1.1]{Penner}}]\label{dimensions}
  Suppose that $\tau$ is recurrent with $v$ switches and $b$ branches.
  If $\tau$ is non-orientable then $\dim M(\tau)=b-v=-\chi(\tau)$. If $\tau$
  is orientable then $\dim M(\tau)=b-(v-1)=1-\chi(\tau)$.
\end{prop}

The point of this result is that in the non-orientable case the switch equations are
all independent, and in the orientable case one is redundant.
For the next result, a \emph{maximal} track is a track
which is maximal with respect to inclusions of tracks: its complementary components
are all either triangles or once-punctured monogons. Recall here that we
are excluding sporadic surfaces.
A track is \emph{filling} if it
carries a minimal filling lamination: that is, it carries a minimal lamination
whose complementary regions are ideal polygons with at most
one puncture. If a track is filling, it is necessary that all
complementary regions are polygons with at most one puncture. However,
this condition is not sufficient.
Each such polygon is bounded by finitely many (smooth) train paths (called \emph{sides})
meeting at switches pointing outward from the polygon (called \emph{corners}).

\begin{prop}\label{numerology}
Let $\tau$ be a filling recurrent train track and let $s_1, \dots, s_k$ be the numbers of sides of the complementary polygons of $\tau$.
Let $p$ be the number of punctures of $\Sigma$.
\begin{enumerate}[(i)]
\item If $\tau$ is non-orientable then
$$\ind(\tau) = 2p + \sum(s_i - 3).$$

\item If $\tau$ is orientable then
$$\ind(\tau) =-1+ 2p + \sum(s_i -3).$$
\end{enumerate}
\end{prop}

\begin{proof}
Note that $\dim \ML(\Sigma) = -3\chi(\Sigma) - p$ and that
$\chi(\Sigma)=\chi(\tau)+(k-p)$ since the number of unpunctured
polygons of $\Sigma \smallsetminus \tau$ is $k-p$.
Therefore, if $\tau$ is non-orientable,
$$\ind(\tau) = -3\chi(\Sigma) - p + \chi(\tau) = -2\chi(\tau) -3k + 2p$$
by Proposition \ref{dimensions}.
There is one switch of $\tau$ for each corner, so the total number of switches is $\sum s_i$. Since the train track is generic, each vertex is valence three and the number of edges of $\tau$ is $\frac32\sum s_i$. Therefore
$$\chi(\tau) = -\frac12 \sum s_i$$
and
$$\ind(\tau) =  \sum s_i - 3k + 2p = 2p +\sum(s_i -3).$$

If $\tau$ is orientable, by Proposition \ref{dimensions}, the dimension of $M(\tau)$ increases by one so
$$\ind(\tau) =-1+ 2p + \sum(s_i -3).$$
\end{proof}

When the index of a train track is sufficiently large, the train track is necessarily
not filling. The next proposition quantifies the exact upper bound on the index of a filling
track.

\begin{prop}\label{ms}
  The largest index $m(\Sigma)$ of a recurrent filling track on $\Sigma$ is
  \begin{itemize}
  \item $4g-5$ if $\Sigma$ is closed of genus $g>2$,
  \item 2 if $\Sigma$ is closed of genus $g=2$,
    \item $4g+p-4$ if $\Sigma$ is nonsporadic of genus $g$ with $p$
      punctures and $(g,p)\neq (1,2)$,
      \item 1 if $\Sigma$ is the twice-punctured torus.
  \end{itemize}
\end{prop}

\begin{proof}
  From the formulas above, it follows that the index of a filling track
  is maximized when the
  number of complementary regions is minimized and the track is non-orientable. For
  closed surfaces of genus $g>2$ there are non-orientable recurrent
  tracks with one complementary component (with $4g-2$ sides).
  When $g=2$ such a train track is necessarily
  orientable. When $p>0$ the smallest possible number of complementary
  components is $p$ and there are such non-orientable tracks as long as
  $(g,p)\neq (1,2)$, in which case such tracks are orientable. The
  (non)-existence of tracks as claimed follows from
  \cite{masur-smillie}. Note that the results there are phrased in terms of
  singularities of quadratic differentials.
  \end{proof}

Set $m(\Sigma)$ to be the number given by Proposition \ref{ms}.

\begin{definition}
  \label{ckdefn}
  Let $0\leq k \leq m(\Sigma)$. Define $\C_k(\Sigma)$ to be the graph whose
  vertices are the isotopy classes of essential simple closed curves on $\Sigma$
  and whose edges join pairs $\{\alpha,\beta\}$ such that there is a track $\tau$
  with $\ind(\tau) >k$ that carries both $\alpha$ and $\beta$. Distance in $\C_k(\Sigma)$ is denoted by $d_k$.
\end{definition}

We do not require the track $\tau$ in Definition \ref{ckdefn} to be connected.
For any $\alpha,\beta$ having disjoint representatives, we may consider
their disjoint union to be a train track $\tau$ and we have $\ind(\tau) > m(\Sigma)$.
Thus, $\alpha,\beta$ are joined by an edge in $\C_k(\Sigma)$ for each $k$ and the
natural map $\C(\Sigma) \to \C_k(\Sigma)$ is 1-Lipschitz. Moreover,
$\C_{k+1}(\Sigma)\to \C_k(\Sigma)$ is also 1-Lipschitz.
Although we have restricted $k$ to lie between 0 and $m(\Sigma)$,
$\C_{-1}(\Sigma)$ would naturally be the complete graph, since any two
curves are carried by a common train track. The graph $\C_0(\Sigma)$ is the
\emph{principal curve graph}, so named by Hamenst\"adt in
\cite{HamenstadtPrincipal}.
Hamenst\"adt's definition is slightly
different; we are using the one given by Maher-Masai-Schleimer
in \cite{MMS2}.

Throughout the rest of the paper, $d=d_{\C(\Sigma)}$ denotes distance in the
curve graph.
Note that if $d_{m(\Sigma)}(\alpha, \beta) = 1$, then $\alpha$ and $\beta$ are carried by a track $\tau$ that doesn't fill, so $\alpha$ and $\beta$ don't fill. In particular, $d(\alpha, \beta) \le 2$ and therefore
$\C_{m(\Sigma)}(\Sigma)$ is quasi-isometric to
$\C(\Sigma)$. This yields a tower of graphs with natural 1-Lipschitz maps
that are the identity on the vertices:
$$\C\overset{qi}{\simeq}\C_{m(\Sigma)}\to \mathcal
  C_{m(\Sigma)-1}\to\cdots\to \C_{1}\to \mathcal
  C_{0}$$

Our main theorem is:

\begin{thm}
  \label{mainthm}
  Every $\C_{k}(\Sigma)$ is hyperbolic. The principal curve graph
  $\C_{0}(\Sigma)$ is a quasi-tree. Each map $\mathcal
    C_{k+1}(\Sigma)\to\C_k(\Sigma)$ has the property that every bounded set in
  $\C_k(\Sigma)$ is contained in a bounded set whose preimage in $\C_{k+1}(\Sigma)$ is
  a quasi-convex quasi-tree.
\end{thm}

\section{Quasi-trees}
Let $X$ be a graph. A subgraph $A \subset X$ is
 \emph{$K$-coarsely connected} if for every $x,y\in A$, there is a
sequence of vertices $x=x_0,x_1,\ldots,x_n=y$ in $A$ with $d(x_i,x_{i+1})\leq K$ for each $i$.

For $x,y \in X$ we let $[x,y]$ denote a geodesic between $x$ and $y$. The set $A$ is {\em $Q$-quasi-convex} if, for all $x, y \in A$, we have $[x,y] \subset B_Q(S)$ where $B_Q(A)$ is the $Q$-neighborhood of $A$.
If $x,y \in A$ and $x = x_0, \dots, x_n =y$ is geodesic in $X$, then for each $x_i$ there is a $y_i \in A$ with $d(x_i, y_i) \le Q$. Therefore $x= y_0, \dots, y_n = y$ is a $(1, 2Q)$-quasi-geodesic between $x$ and $y$ that lies in $A$, and since $d(y_i, y_{i+1}) \le 2Q+1$, $A$ is $2Q+1$-connected. We denote the path $\{y_i\}_{i=0}^n$ by $q_{x,y}$.

For sets $A_1, \dots, A_m$, define their {\em coarse intersection} to be
$$I_C(A_1, \dots, A_m) = B_C(A_1) \cap \cdots \cap B_C(A_m).$$

We collect some facts about quasi-convex subsets of a $\delta$-hyperbolic graph that we will use repeatedly.

\begin{lemma}\label{fellow travel}
Fix $0<\delta\le Q \le C$ and let $X$ be a $\delta$-hyperbolic geodesic metric space.
\begin{enumerate}
\item Let $A$ be a $Q$-quasi-convex subset of $X$. If $x,y \in B_C(A)$ then 
\begin{eqnarray*}
[x,y] &\subset & B_{Q+2\delta}(A) \cup B_{C+2\delta}(\{x,y\})\\  & \subset & B_{C+2\delta}(B_C(A)).
\end{eqnarray*}

\item\label{intersection quasiconvex} If $A_1, \dots, A_n$ are $Q$-quasi-convex subsets of $X$  then $I_C(A_1, \dots, A_n)$ is $(C+2\delta)$-quasi-convex.

\item\label{union quasiconvex} If $A_0, A_1, \dots$ are $Q$-quasi-convex subsets of $X$ with $I_C(A_0, A_i) \neq \emptyset$ for all $i\ge 1$ and $C\ge Q$, then $\bigcup_{i\ge 0} A_i$ is $(2C+4\delta)$-quasi-convex.

\item\label{triple} If $A_1, A_2$, and $A_3$ are $Q$-quasi-convex subsets and $I_C(A_1, A_2)$, $I_C(A_1, A_3)$, and $I_C(A_2, A_3)$ are all non-empty, then $I_{2C+3\delta}(A_1, A_2, A_3) \neq \emptyset$.

\item\label{larger}  If $A_1, \dots, A_n$ are $Q$-quasi-convex subsets of $X$ and the diameter of $I_C(A_1, \dots, A_n)$ is $> 4C + 8\delta$, then
$$I_C(A_1, \dots, A_n) \subset B_{C+2\delta}(I_{Q+2\delta}(A_1, \dots, A_n)).$$
\end{enumerate}
\end{lemma}

\begin{proof}
\begin{enumerate}
\item Let $x', y' \in A$ with $d(x,x'), d(y,y') \le C$. By $\delta$-hyperbolicity, $[x,y]$ is contained in the $2\delta$-neighborhood of $[x,x'] \cup [x', y'] \cup [y', y]$, and $B_{2\delta}([x,x']\cup [y,y']) \cap [x,y] \subset B_{C+2\delta}(\{x,y\})$. The inclusion follows.

\item This follows directly from (1).

\item For $i \ge 1$, fix $t_i \in I_C(A_0, A_i)$ and fix $t\in A_0$. By $\delta$-hyperbolicity, for any $x_i \in A_i$ we have $[t_i, x_i] \subset B_\delta([t,t_i] \cup [t, x_i])$ and therefore $[x_i, t_i] \subset B_{2C+3\delta}(A_0\cup A_i)$. For $x_j \in A_j$, we also have $[x_i, x_j] \subset B_{\delta}([t, x_i]\cup[t, x_j])$. Combining with the previous inclusion gives the result.

\item By assumption, there are points $x_{ij} \in I_C(A_i, A_j)$. Applying $\delta$-hyperbolicity, there is a point $x' \in [x_{12}, x_{23}]$ that is contained in $I_\delta([x_{12}, x_{13}], [x_{23}, x_{13}])$. Then $x' \in I_{2C+3\delta}(A_1, A_2, A_3)$.

\item If there is an $x \in I_C(A_1, \dots, A_n)$ such that $d(x,y) \le 2C+4\delta$ for all $y \in I_C(A_1, \dots, A_n)$, then the diameter of $I_C(A_1, \dots, A_n)$ is $\le 4C+8\delta$. If not, for all $x \in I_C(A_1, \dots, A_n)$ there is a $y \in I_C(A_1, \dots, A_n)$ with $d(x,y) > 2C+4\delta$ so, by (1), there is an $x' \in I_{Q+2\delta}(A_1, \dots, A_n)$ with $d(x,x') \le C+ 2\delta$.
\end{enumerate}
\end{proof}

A metric space is a {\em quasi-tree} if it is quasi-isometric to a simplicial tree. Manning's {\em bottleneck condition} gives a criterion for a geodesic metric space to be a quasi-tree. We use the version described on
\cite[Page 13]{BBF}. 

\begin{thm}[{\cite[Theorem 4.6]{M}}]
  \label{bottleneckcrit}
  Suppose that $X$ is a geodesic metric space and $\Delta>0$ is a constant such that the
  following property is satisfied. For any pair of points $x,y\in X$ there is a path $p_{x,y}$
  from $x$ to $y$ such that any other path $q$ between $x$ and $y$ contains $p_{x,y}$
  in its $\Delta$-neighborhood. Then $X$ is quasi-isometric to a tree, with quasi-isometry
  constants depending only on $\Delta$.
\end{thm}

For our applications, there is an extra subtlety in that the metric spaces we will be interested in will not be geodesic and, in fact, will not even be connected. More concretely, we need a version of the bottleneck condition that we can apply to a quasi-convex subgraph of a graph.

Let $X$ be a graph and $A$ a $Q$-quasi-convex subgraph. For $K \ge 2Q+1$, we construct a new graph $X_K$ by adding an edge between any pair of vertices $x, y \in X$ with $2\le d(x,y) \le D$. Let $A_K \subset X_K$ be the subgraph spanned by $A$, i.e. the largest subgraph of $X_K$ with the same vertex set as $A$. Let $d$, $d_A$, and $d_K$ be the metrics for $X,A$, and $X_K$, respectively, where $d_A$ is the subspace metric. Let $\tilde d_A$ be the path metric on $A_K$. Clearly, $d(x,y) \le K\tilde d_A(x,y)$ for vertices $x,y \in A$. We also observe that the paths $q_{x,y}$ are $K$-connected and have length $d(x,y)$, so $\tilde d_A(x,y) \le d(x,y)$.
As $d_A(x,y) = d(x,y)$, we have
$$\tilde d_A(x,y) \le  d_A(x,y) \le K\tilde d_A(x,y)$$
and therefore $(A, d_A)$ and $(A_K, \tilde d_A)$ are quasi-isometric. Moreover, a $K$-connected path in $A$ determines an actual path in $A_K$ and conversely the vertices of any path in $A_K$ are a $K$-connected path in $A$. Combining this with Theorem \ref{bottleneckcrit}, we have:

\begin{cor}\label{subspacebottleneck}
Let $X$ be a graph and $A\subset X$ a $Q$-quasi-convex subgraph. Assume there is $K>2Q+1$ and $\Delta>0$ such that for any pair of points $x,y \in A$ there is a $K$-connected path $p_{x,y}$ in $A$ from $x$ to $y$ such that the $\Delta$-neighborhood of any $K$-connected path in $A$ between $x$ and $y$ contains $p_{x,y}$. Then $A$ is quasi-isometric to a simplicial tree with quasi-isometry constants depending only on $K$ and $\Delta$.
\end{cor}

If $A$ is a $Q$-quasi-convex subgraph of a graph $X$ and $K \ge 2Q+1$, we say that $A$ satisfies a \emph{$(\Delta,K)$-bottleneck condition} if for any $x,y \in A$ and any $K$-connected path $p$ in $A$ between $x$ and $y$, we have $[x,y] \subset B_\Delta(p)$. Observe that $q_{x,y}\subset B_Q([x,y])$, so if $A$ satisfies a $(\Delta,K)$-bottleneck condition then $q_{x,y} \subset B_{\Delta +Q}(p)$. Therefore Corollary \ref{subspacebottleneck} yields:
\begin{cor}\label{geodesic_bottleneck}
Let $A$ be a $Q$-quasi-convex subgraph of a graph $X$ and let $K \ge 2Q+1$ and $\Delta>0$. If $A$ satisfies a
$(\Delta, K)$-bottleneck condition then $A$ is quasi-isometric to a simplicial tree with quasi-isometry constants depending only on $K$ and $\Delta$.
\end{cor}

Using the quasi-isometry of the quasi-tree to a simplicial tree we also have a converse:
\begin{prop}\label{bottleneck_converse}
Let $A$ be a $Q$-quasi-convex subgraph of a $\delta$-hyperbolic graph $X$. If $A$
(with the subspace metric) is quasi-isometric to a simplicial tree then,
for any $K\ge 2Q+1$, there is $\Delta>0$, depending on $K$ and the quasi-isometry constants, such that $A$ satisfies a $(\Delta, K)$-bottleneck condition.

Furthermore, for $D>0$ any $D$-connected subset of $A$ is a quasi-tree.
\end{prop}

Next we show that a union of finitely many quasi-trees is a quasi-tree.
\begin{lemma}\label{finite quasitree}
Let $X$ be a $\delta$-hyperbolic graph and  let $T_0,\dots, T_n$ be $Q$-quasi-convex subsets. If each $T_i$ is a quasi-tree and there exists $C\ge Q$ such that $I_C(T_0, T_i) \neq \emptyset$ for each $i$, then $T = T_0 \cup \cdots \cup T_n$ is a quasi-tree with constants only depending on $\delta$, $C$, $n$ and the quasi-tree constants for the sets $T_i$.
\end{lemma}

\begin{proof}
It is enough to prove this when $n=1$, as the general case follows by induction.  By \eqref{union quasiconvex} of Lemma \ref{fellow travel},  $T_0 \cup T_1$ is $2C+ 4\delta$-quasi-convex and $D$-connected with $D = 4C+ 8\delta +1>2C+4\delta$. Let $p$ be a $D$-connected path from $x \in T_0$ to $y \in T_1$.

We break $p$ into subpaths $p_1, \dots, p_{2m}$, where the $p_{2j-1}$ are in $T_0$, the $p_{2j}$ are in $B_D(T_1)$, and the endpoints of each $p_i$, except possibly for the terminal endpoint of $p_{2m}$, are in $T_0$.

As $T_0$ and $B_D(T_1)$ are quasi-trees, they satisfy a $(\Delta, D)$-bottleneck condition for some $\Delta>0$.
For each $i$, let $g_i$ be a geodesic with the same endpoints as $p_i$.
Then $g_i \subset B_\Delta(p_i)$.

The geodesics $g_2, g_4, \dots, g_{2m-2}$ all have endpoints in the $Q$-quasi-convex set $T_0$, so there are $(2Q+1)$-connected paths $q_2, q_4, \dots, q_{2m-2}$ in $T_0$ with the same endpoints as $g_{2j}$ and $q_{2i} \subset B_Q(g_{2i}) \subset B_{\Delta + Q}(p_{2i})$.

Then $\tilde p = p_1 \cup q_1 \cup \cdots \cup p_{2m-3} \cup q_{2m-2} \cup p_{2m-1}$ is a $D$-connected path in $T_0$ with $\tilde p \subset B_{\Delta+Q}(p)$. If $\tilde g$ is the geodesic with the same endpoints as $\tilde p$, we have $\tilde g \subset B_\Delta(\tilde p) \subset B_{2\Delta + Q}(\tilde p)$. Then by $\delta$-hyperbolicity of $X$, we have
$$[x,y] \subset B_\delta(\tilde g \cup g_{2m}) \subset B_{2\Delta +Q}(p)$$
and $T$ satisfies the $(2\Delta + Q, D)$-bottleneck condition. 
\end{proof}

In the following lemma, we think of a quasi-tree $T_1$ ``overlapping'' a
quasi-tree $T_0$. The lemma gives a coarse decomposition of the union
into $T_0$ plus many quasi-trees that have small overlap with $T_0$ and
with each other.

\begin{lemma}\label{subtrees}
Let $T_0$ and  $T_1$ be $Q$-quasi-convex subsets of a $\delta$-hyperbolic graph that are quasi-trees with $I_C(T_0, T_1) \neq \emptyset$. Then there are quasi-trees $T'_0$ and $T^i_1$ with 
 $T_0 \cup T_1 = T'_0 \cup \left(\bigcup_j T^i_1\right)$
where
\begin{itemize}
\item $T_0 \subset T'_0\subset B_{C'}(T_0)$;

\item for each $i$, $T_1^i$ is $Q^*$-quasi-convex and $T_1^i \subset T_1$;

\item for each $i$, $I_{C'}(T_0, T^i_1) \neq \emptyset$;

\item for $0\le i<j$, the diameter of $I_{Q^*+2\delta}(T^i_1, T^j_1)$ is $< D$.
\end{itemize}
\end{lemma}

\begin{proof}
By Lemma \ref{finite quasitree} the union $T_0 \cup T_1$ is a quasi-tree, so there is a quasi-isometry $\phi\colon T_0\cup T_1 \to \Lambda$ to a simplicial tree $\Lambda$. The image of $T_0$ in $\Lambda$ is coarsely connected, so there is a sub-tree $\Lambda_0\subset \Lambda$ that contains $\phi(T_0)$ and is contained in a uniformly bounded neighborhood of $\phi(T_0)$. Let $\{\Lambda_i\}$ be the components of $\Lambda\smallsetminus \Lambda_0$. Then we may set
$T'_0 = \phi^{-1}(\Lambda_0)$ and $T^i_1 = \phi^{-1}(\Lambda_i)$.
\end{proof}

Finally, we give a criterion for an infinite union of quasi-trees to be a quasi-tree.
\begin{thm}\label{quasitreelem}
Let $X$ be a $\delta$-hyperbolic graph and let $T_0, T_1,\dots$ be subgraphs that are $Q$-quasi-convex, uniform quality quasi-trees.
Assume that
\begin{enumerate}
\item there exists $C>0$ such that $I_{C}(T_0, T_i) \neq \emptyset$ for all
$i\ge 1$;
\item there exists $N>0$ and $D>0$ such that the diameter of $I_{Q+2\delta}(T_{i_1}, \dots, T_{i_N})$ is bounded by $D$ for all $1\le i_1 < i_2 < \cdots < i_N$.
\end{enumerate}
Then $T = \displaystyle{\bigcup_{i\ge 0}} T_i$ is a quasi-tree with constants only depending on $\delta$, $Q$, and the quasi-tree constants for the sets $T_i$.
\end{thm}
\begin{proof}
By Lemma \ref{fellow travel} part \eqref{union quasiconvex}, $T$ is quasi-convex. We induct on $N$ to prove that $T$ is a quasi-tree.

The base case is $N=1$. We first prove this with the extra assumption that the diameter of $I_{Q+2\delta}(T_i, T_j)$ is also bounded above by $D$ for all
$i,j$.

By Lemma \ref{fellow travel} part \eqref{union quasiconvex}, $T$ is $(2C+4\delta)$-quasi-convex and is $K$-connected with $K = 4C + 8\delta +1$.
Furthermore, if $I_K(T_i, T_j) \neq \emptyset$ with $0<i<j$ then, by Lemma \ref{fellow travel} part \eqref{triple}, we have that $I_{K'}(T_0, T_i, T_j) \neq \emptyset$ with $K' = 2K+3\delta$ and, using our extra assumption, by \eqref{larger} the diameter of $I_{K'}(T_0, T_i, T_j)$ is $< D'$ where $D' = 2K'+4\delta + D$.

For each $T_i$ with $i\ge1$, choose $t_i \in  I_{C}(T_0, T_i)$. If $x \in T_i$ and $y \in T_j$ with $0\le i<j$ and $d(x,y) \le K$ then $d(x, t_i) \le 2D'$. Therefore if $i> 0$ we have $d(t_i, t_j) \le 2D' +K$.

Now let $x_0, \dots, x_n$ be a $K$-connected path in $T$. We assume $x_0 \in T_a$ and $x_n \in T_b$ with $0<a<b$. The other cases are similar (and simpler). Let $x_{i_a}$ be the first point for which $x_{i_a+1} \not\in  T_a$ and let $x_{i_b}$ be the last point for which $x_{i_b-1} \not\in T_b$.

Form a new sequence $y_{i_a}, \dots, y_{i_b}$ where $y_i = x_i$ if $x_i \in T_0$ and $y_i = t_j$ if $x_i \in T_j$ with $j \neq 0$. It is possible that $x_i$ will be contained in more than one $T_j$. In this case we choose $T_j$ arbitrarily. We claim that $y_0 = t_a, y_1, \dots, y_n = t_b$ is a $(2D'+K)$-connected path. There are four cases:
\begin{itemize}
\item If $x_i, x_{i+1} \in T_j$ with $j \neq 0$ then $d(y_i, y_{i+1}) = d(t_j, t_j) = 0$. 
\item If $x_i, x_{i+1} \in T_0$ then $d(y_i,y_{i+1})= d(x_i, x_{i+1}) \le K$.
\item If $x_i \in T_j$ and $x_{i+1} \in T_k$ with $j,k\neq 0$ and $j\neq k$ then $d(y_i, y_{i+1}) = d(t_j, t_k) \le 2D'+K$.
\item If $x_i \in T_0$ and $x_{i+1} \in T_j$ with $j\neq 0$ then $d(y_i, y_{i+1}) = d(x_i,t_j) \le 2D'$. This similarly holds if $x_i \in T'_j$ and $x_{i+1} \in T_0$.
\end{itemize}

The vertices in the path $\{y_{i_a},\ldots,y_{i_b}\}$ are either in the path $\{x_0,\ldots,x_n\}$ or in $\{t_j\}_j$. If $t_j$ is in $\{y_{i_a},\ldots,y_{i_b}\}$ then there is $x_i \in T_j$ with $x_{i+1}$ or $x_{i-1}$ not in $T_j$. Therefore $d(x_i, t_j) \le 2D'$ and $\{y_{i_a},\ldots,y_{i_b}\}$ is therefore contained in the $2D'$-neighborhood of $\{x_0,\ldots,x_n\}$.

The sets $T_i$ are $Q$-quasi-convex, uniform quality quasi-trees in $X$. Therefore, by Lemma \ref{bottleneck_converse}, there is $\Delta>0$ such that they satisfy the $(\Delta,2D'+K)$-bottleneck condition and
$[x_0, x_{i_a}] \subset B_\Delta(\{x_0, \dots, x_{i_a}\})$,
$ [y_{i_a}, y_{i_b}] \subset  B_\Delta(\{y_{i_a}, \dots, y_{i_b}\})$, and 
$ [x_{i_b}, x_n] \subset  B_\Delta(\{x_{i_b}, \dots, x_n\})$.

 Thus $B_{\Delta +2D'}(\{x_0, \dots, x_n\})$ contains the union of geodesics $[x_0, x_{i_a}] \cup [x_{i_a}, y_{i_a}] \cup [y_{i_a}, y_{i_b}] \cup [y_{i_b}, x_{i_b}] \cup [x_{i_b}, x_n]$. By $\delta$-hyperbolicity, a $4\delta$-neighborhood of this last union contains $[x_0, x_n]$ so
 $$[x_0, x_n] \subset B_{\Delta + 2D' + 4\delta}(\{x_0, \dots, x_n\})$$ and $T$ satisfies the $(\Delta + 2D'+4\delta, K)$-bottleneck condition and is quasi-tree by Corollary \ref{geodesic_bottleneck}.
 
In general, we do not have a uniform bound on the diameter of $I_{Q+2\delta}(T_0, T_i)$. However, we can reduce to this case by applying Lemma \ref{subtrees} to each pair $T_0$ and $T_i$. Each time we do this $T_0$ is enlarged to a subset $T_{0,i}$ (this is $T'_0$ in the lemma) with $T_0 \subset T_{0,i} \subset B_{C'}(T_0)$. Hence if we replace $T_0$ with the union $T'_0 =  \cup T_{0,i}$ we get a $T_0 \subset T'_{0} \subset B_{C'}(T_0)$ so $T'_0$ is a quasi-tree with uniform constants. At each application of the lemma the $T_i$ become quasi-trees $T^j_i$. Then $T'_0$ and the $T^j_i$ have the property that $I_{Q^*+2\delta}(T'_0, T^j_i)$ has uniformly bounded diameter and we have reduced to this case.

Now we complete the induction step. Assume the theorem holds for $N-1$. By assumption, for $i_1, \ldots, i_N$ distinct, the diameter of $I_{Q+2\delta}(T_{i_1}, \dots, T_{i_N})$ is $<D$.
Furthermore, if $I_{Q+2\delta}(T_{i_1}, \dots, T_{i_{N-1}})$ is non-empty then, by Lemma \ref{fellow travel} part \eqref{intersection quasiconvex}, it is a $(Q+4\delta)$-quasi-convex subset and hence a quasi-tree. Denote the set $I_{Q+2\delta}(T_{i_1},\ldots,T_{i_{N-1}})$ by $J_{{i_1}, \dots, {i_{N-1}}}$. Let $J = J_{{i_1}, \dots, {i_{N-1}}}$ and $J'=J_{{i'_1}, \dots, {i'_{N-1}}}$ be two distinct sets of this type. We can assume that $i'_1 \not\in \{i_1, \dots, i_{N-1}\}$. Then $I_{Q+6\delta}(J, J') \subset I_{2Q+ 8\delta}(T_{i_1}, \dots, T_{i_{N-1}}, T_{i'_1})$ and therefore, by Lemma \ref{fellow travel} part \eqref{larger}, the diameter of $I_{Q+2\delta}(J, J')$ is $< D'$ where $D' = D+4Q+20\delta$. By the base case, the union of $T_0$ and the trees $J_{{i_1}, \dots, {i_{N-1}}}$ is a quasi-tree. Label this union $J_0$.

Next apply Lemma \ref{subtrees} to $J_0$ and each $T_i$ to obtain quasi-trees $T'_0$ and $T^j_i$. 
We'll show that $T'_0$ and the sets $T^j_i$ satisfy the theorem conditions for $N-1$. First, we observe that $T'_0$ and the sets $T^j_i$ are $Q^*$-quasi-convex, by Lemma \ref{subtrees}, where $Q^*$ is bounded in terms of our earlier constants.
 
 We claim that the diameter of $I_{Q^*+2\delta}(T^{j_1}_{i_1}, \dots, T^{j_{N-1}}_{i_{N-1}})$ is uniformly bounded. If $i_k = i_\ell$ for some $k \neq \ell$ then $I_{Q^*+2\delta}(T_{i_k}^{j_k}, T_{i_\ell}^{j_\ell})$ is bounded in terms of the constant from Lemma \ref{subtrees}, so that $I_{Q^*+2\delta}(T^{j_1}_{i_1}, \dots, T^{j_{N-1}}_{i_{N-1}})$ is also bounded as a subset. Therefore we can assume that all the $i_k$ are distinct. If the
 diameter of $I_{Q^*+2\delta}(T^{j_1}_{i_1}, \dots, T^{j_{N-1}}_{i_{N-1}})$ is at most $4Q^* + 16\delta$ then there is nothing to show. If not, then the diameter of $I_{Q^*+2\delta}(T_{i_1}, \dots, T_{i_{N-1}})$ is greater than $4Q^* + 16\delta$ and
\begin{eqnarray*}
I_{Q^*+2\delta}\left(T^{j_1}_{i_1}, \dots, T^{j_{N-1}}_{i_{N-1}}\right) &\subset & I_{Q^*+2\delta}(T_{i_1}, \dots, T_{i_{N-1}}) \\
& \subset & B_{Q^* + 4\delta}(J_{i_1, \dots, i_{N-1}}) \\ 
& \subset & B_{Q^*+4\delta}(T'_0)
\end{eqnarray*}
where the middle inclusion follows from Lemma \ref{fellow travel} part \eqref{larger}.
It follows that
\begin{eqnarray*}
I_{Q^*+2\delta}\left(T^{j_1}_{i_1}, \dots, T^{j_{N-1}}_{i_{N-1}}\right)& = &I_{Q^*+2\delta}\left(T^{j_1}_{i_1}, \dots, T^{j_{N-1}}_{i_{N-1}}\right) \cap B_{Q^*+4\delta}(T'_0)\\ & = & I_{Q^*+4\delta}\left(T^{j_1}_{i_1}, \dots, T^{j_{N-1}}_{i_{N-1}}, T'_0\right) \\
&\subset& I_{Q^*+4\delta}\left(T^{j_1}_{i_1}, T'_0\right) \cup \cdots \cup I_{Q^*+4\delta}\left(T^{j_{N-1}}_{i_{N-1}}, T'_0\right).
\end{eqnarray*}
As the final sets on the right are uniformly bounded and $I_{Q^*+2\delta}(T^{j_1}_{i_1}, \dots, T^{j_{N-1}}_{i_{N-1}})$ is coarsely connected (with controlled bounds) this gives a uniform bound on $I_{Q^*+2\delta}(T^{j_1}_{i_1}, \dots, T^{j_{N-1}}_{i_{N-1}})$. This completes the induction step of the proof.
\end{proof}

\section{Hamenst\"{a}dt's graphs}\label{sec:3}
Hamenst\"{a}dt originally constructed a graph related to $\C_k$ using a definition
which does not involve train tracks. We will show that her graphs are in fact
quasi-isometric to the graphs $\C_k$.

We define an index $\ind(\alpha,\beta)$ associated to a pair of curves $\alpha,\beta$
in minimal position.
First suppose that $\alpha$ and $\beta$ fill $\Sigma$. Then the complementary components of $\alpha\cup \beta$ are even-sided polygons (since the curves fill) each with at most one puncture.
Any bigon in the complement of $\alpha\cup\beta$ has a single puncture.
We say that a component of $\Sigma \smallsetminus (\alpha\cup\beta)$ is {\em large} if it is not a non-punctured quadrilateral.
 We label the large components $P_1, \dots, P_n$ and let $p$ be the number of punctures.
 Let $2s_i$ be the number of sides of $P_i$. If $P_i$ is non-punctured,
 then $s_i \ge 3$. If $P_i$ is punctured, it is possible that $s_i = 1$ or $2$.
If, after orienting $\alpha$ and $\beta$, all intersection are of the same sign,
we say that $\alpha$ and $\beta$ are an {\em orientable pair} and define
$\ind(\alpha, \beta) =-1+2p + \sum (s_i -3)$.
 Otherwise we define $\ind(\alpha, \beta) = 2p+\sum(s_i -3)$.
 Finally, if $\alpha,\beta$ don't fill $\Sigma$,
 we define $\ind(\alpha, \beta) = \infty$.

Next, we define $\C^H_k(\Sigma)$ to be the graph with vertex set
$\mathcal{V}(\C^H_k(\Sigma))=\mathcal{V}(\C(\Sigma))$ and such that $\alpha,\beta\in \mathcal{V}(\C^H_k(\Sigma))$
are joined by an edge if $\ind(\alpha, \beta) > k$.
This is Hamenst\"{a}dt's original definition. Note that if all the complementary
components of $\alpha\cup \beta$ are quadrilaterals, hexagons, or punctured bigons,
then there is at least one hexagon or punctured bigon.
Thus, the algebraic intersection number of $\alpha$ and $\beta$
must be strictly less than the geometric
intersection number, and we have $\ind(\alpha, \beta) = 0$.

Denote by $d_k^H$ the metric defined on $\mathcal V(\C(\Sigma))$ by the graph
$\C_k^H(\Sigma)$. The remainder of this section will be dedicated to proving the
following:
\begin{thm}\label{ursuladef}
If $d_k(\alpha, \beta) = 1$ then $d^H_k(\alpha,\beta) \le 2$.
If $d^H_k(\alpha, \beta) =1$ then $d_k(\alpha,\beta) =1$.
In particular, $\C_k(\Sigma)$ and $\C^H_k(\Sigma)$ are quasi-isometric.
\end{thm}

For the rest of the paper, we will typically fix the surface $\Sigma$
and write $\C$ for $\C(\Sigma)$, $\PML$ for $\PML(\Sigma)$, etc.

We note that it may in fact be true that $\C^H_k(\Sigma) = \C_k(\Sigma)$.
The difficulty in proving this is that we don't have a direct criterion for
checking if $d_k(\alpha, \beta) \ge 2$.

\begin{lemma}\label{track index}
Let $\alpha$ and $\beta$ be a pair of curves that fill $\Sigma$. Then there exists a train track $\tau(\alpha, \beta)$ that carries both $\alpha$ and $\beta$ with $\ind(\tau(\alpha,\beta)) = \ind(\alpha,\beta)$.
\end{lemma}

\begin{proof}
Recall that $\alpha$ and $\beta$ are chosen to be in minimal position.
A train track with bigons $\tau'(\alpha,\beta)$ is obtained from $\alpha\cup \beta$ by
replacing each intersection point by a large branch at which $\alpha$ turns
right and $\beta$ turns left.
We will collapse the bigons to form a track $\tau(\alpha,\beta)$. That is, we will
define a quotient map of the surface,
which is a self-homotopy equivalence, so that the image of $\tau'(\alpha,\beta)$ will be the track $\tau(\alpha, \beta)$. Some care needs to be taken to ensure that the quotient space is still the same surface.

We make the collapsing procedure explicit as follows.
The union $\alpha\cup \beta$ is a graph embedded in $\Sigma$. There is thus a dual
graph $(\alpha\cup\beta)^*$.
As a graph, all vertices of $\alpha\cup \beta$ have valence four and therefore all of
the complementary components of $(\alpha\cup\beta)^*$ are quadrilaterals. The graph
$(\alpha \cup \beta)^*$ can be realized in $\Sigma$ with its punctures filled in.
We place one vertex in each
complementary component of $\alpha\cup\beta$ and when the region contains a puncture we
place the vertex of $(\alpha\cup \beta)^*$ at the puncture.
We can then realize each quadrilateral as a unit square.
These squares define a quadratic differential $q$ on $\Sigma$ with two corresponding
transverse singular foliations.
The union of lines in the squares parallel to $\alpha$ form the {\em vertical
foliation} while the lines parallel to $\beta$ form the {\em horizontal foliation}.
These foliations have singularities at those vertices of $(\alpha \cup \beta)^*$
which are incident to more than four squares.
The number of squares that meet at a vertex of $(\alpha\cup \beta)^*$ is exactly the
number of sides of the corresponding complementary component of $\alpha\cup \beta$.
Each square of $q$ contains a single vertex of $\alpha\cup\beta$ and each edge
of a square intersects a single edge of $\alpha\cup\beta$.

Recall the train track with bigons $\tau'(\alpha, \beta)$ defined above. We use
the foliations of $q$ to explicitly construct a train track $\tau(\alpha, \beta)$
without bigons. We first replace
$\alpha\cup\beta$ with a trivalent graph $(\alpha\cup\beta)'$ by inserting an
edge of length $\sqrt{0.02}$ and slope one centered at each vertex of
$\alpha\cup\beta$ (contained inside a square of the quadratic differential $q$)
and then modifying the original edges to line segments that start and end at the vertices of the slope one edges. After doing this, the vertical edges of $\alpha\cup \beta$ will have slope $-9$ and the horizontal edges will have slope $-1/9$.
We can then smooth the vertices of $\alpha\cup \beta$ to define a train track
$\tau'(\alpha, \beta)$ which has the property that the slopes of all its tangent lines
lie outside the range $\left(-9, -1/9\right)$.
See Figure \ref{fig:smoothing}.

\begin{figure}[h]

  \centering

  \def\svgwidth{0.9\textwidth}
  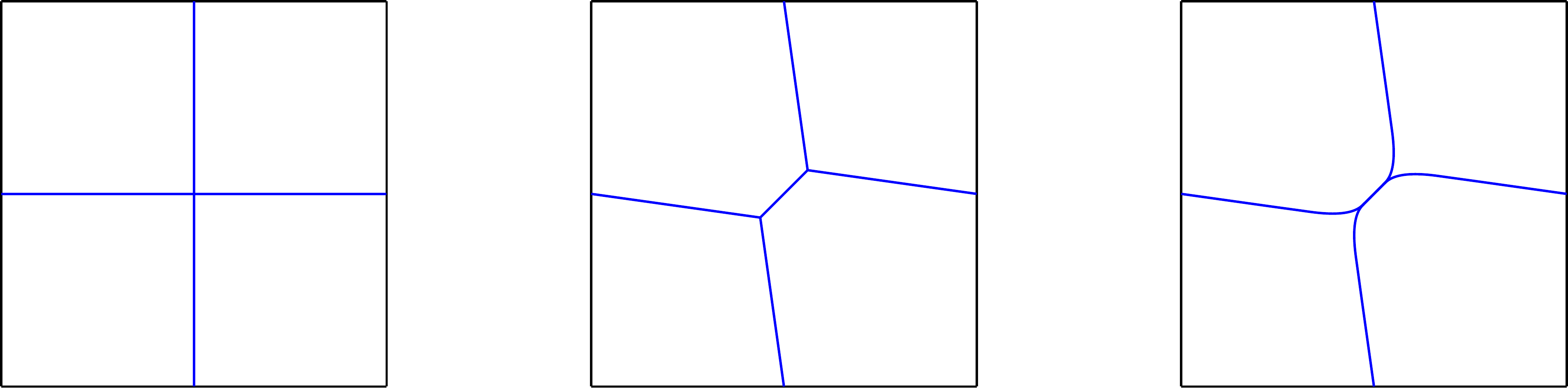

  \caption{Constructing a train track with tangent lines avoiding a fixed slope in the
  range $(-9,-1/9)$.}
  \label{fig:smoothing}
\end{figure}

Now fix an irrational slope $\delta \in (-9, -1/9)$. This slope also determines a foliation $\cF_\delta$ on $\Sigma$ with possible singularities at the vertices of $(\alpha\cup\beta)^*$. Note that this foliation does not have closed leaves.
Furthermore, since non-punctured bigons of $\tau'(\alpha, \beta)$ correspond to quadrilaterals in $\alpha\cup\beta$, they don't contain singularities, and therefore, by the Mean Value Theorem,
any arc of $\cF_\delta\setminus \tau'(\alpha,\beta)$ contained in such a bigon connects
the two sides of the bigon.
We then define a quotient map $\Sigma \to \Sigma'$  by collapsing the closed segments
of $\cF_\delta$ joining the two sides of a non-punctured bigon to points.
Under this map, the image of $\tau'(\alpha,\beta)$ is a train
track $\tau(\alpha, \beta)$ (without bigons). By R.L. Moore's theorem,
$\Sigma'$ is homeomorphic to $\Sigma$, see \cite[Section 25]{daverman}.

There is a bijection between the set of complementary components of $\tau(\alpha,\beta)$ and the set of complementary components of
$\alpha\cup \beta$ that aren't unpunctured quadrilaterals.
This bijection has the property that if a complementary
component of $\alpha\cup\beta$ has $2s$ sides then the corresponding complementary
component of $\tau(\alpha,\beta)$ has $s$ sides. Furthermore, $\tau(\alpha,\beta)$
is orientable if and only if the algebraic intersection of $\alpha$ and $\beta$ is equal to the geometric intersection number. Therefore $\ind(\alpha,\beta) = \ind(\tau(\alpha, \beta))$.
\end{proof}

Note that in the construction above we built two singular foliations associated to a
quadratic differential $q$ that we called vertical and horizontal. The curve $\alpha$
is a leaf of the vertical foliation while $\beta$ is transverse to it. The roles of
$\alpha$ and $\beta$ are reversed for the horizontal foliation. We can use  either of
these foliations to calculate $\ind(\alpha,\beta)$. We say that a foliation $\cF$ is
{\em compatible} with a pair of curves $\alpha,\beta$ if one curve is a leaf of
$\cF$ and the other curve is transverse to it.

\begin{lemma}\label{side count}
Assume that $\alpha$ and $\beta$ fill $\Sigma$ and that $\cF$ is a foliation which is
compatible with $\alpha,\beta$. Let $Q$ be a component of
$\Sigma\smallsetminus (\alpha\cup \beta)$ with $2s$ sides. Suppose $Q$ contains
$\ell$ singularities, which have $i_1, \dots, i_\ell$ prongs, respectively. Then
$$s = 2-2\ell + \sum_{j=1}^\ell i_j.$$
\end{lemma}

\begin{proof}
Collapsing the leaves of $\cF$ in $Q$ to points yields a finite simplicial tree.
This tree contains $\ell$ vertices of valence $i_1, \dots, i_\ell$, corresponding to
the singularities, and an additional $s$ terminal (valence one) vertices,
corresponding to an alternating set of sides. Therefore the tree has
$s+\ell$ vertices and $\frac12(s + \sum i_j)$ edges. Since the Euler characteristic
of a tree is one, we have
$$s + \ell -\frac12\left( s+ \sum i_j\right) = 1$$
which rearranges to give the equality.
\end{proof}

Summing over all complementary regions we have the following corollary:
\begin{cor}\label{one singularity}
Assume that $\alpha$ and $\beta$ fill $\Sigma$ and that $\cF$ is a foliation compatible with
$\alpha,\beta$ which has $\ell$ singularities with $i_1, \dots, i_\ell$ prongs,
respectively.
Then
\begin{itemize}
\item If $\alpha,\beta$ is not an orientable pair then $\ind(\alpha, \beta) \ge 2p + \sum(i_j - 3)$.

\item If $\alpha,\beta$ is an orientable pair then $\ind(\alpha, \beta) \ge -1+ 2p + \sum(i_j - 3)$.

\end{itemize}
In both cases equality holds if and only if there is one singularity of $\cF$ in each large region of $\Sigma \smallsetminus (\alpha\cup\beta)$.
\end{cor}

We also need the following condition for $\cF$ to be orientable.
\begin{lemma}\label{induced orientation}
Assume that $\alpha$ and $\beta$ fill $\Sigma$ and that $\cF$ is a compatible foliation with exactly one singularity in each large component of $\Sigma\smallsetminus(\alpha\cup\beta)$. If $\alpha,\beta$ is an orientable pair then $\cF$ is orientable.
\end{lemma}
\begin{proof}
Assume that $\alpha$ is carried by $\cF$. As each region $Q$ of $\Sigma\smallsetminus (\alpha\cup \beta)$ contains at most one singularity and the pair $\alpha,\beta$ is orientable, an orientation of $\alpha$ induces an orientation of $Q\cap \cF$. These orientations patch together to give a consistent orientation on all of $\mathcal F$.
\end{proof}

A curve $\alpha$ \emph{hits $\tau$ efficiently} if there are no bigons between
$\alpha$ and any train path on $\tau$. If $\alpha$ hits $\tau$ efficently,
we write $\alpha\pitchfork \tau$.

Given a recurrent train track $\tau$, a curve $\alpha$ with $\alpha\pitchfork\tau$, and a
curve $\beta\in\S(\tau)$, there is a standard way to construct a foliation $\cF(\alpha;\beta, \tau)$ that is compatible with $\alpha$ and $\beta$.
First replace $\tau$ by a regular neighborhood. Equip it with a singular foliation such
that the singularities occur exactly at the corners of the regular neighborhood
and the non-singular leaves are all homotopic to $\beta$.
(If $\beta$ isn't fully carried then we need to add curves to get a multi-curve that is fully carried.)
If $P$ is a region in $\Sigma \smallsetminus \tau$, and $\alpha$ is a curve,
then $P\smallsetminus \alpha$ is a union of polygons whose sides are either
arcs of $\tau$, which we call $\tau$-sides, or arcs of $\alpha$, which we
call $\alpha$-sides.
We construct a
singular foliation in $P$ such that the components of $P\cap \alpha$ are leaves of
this foliation and each component of $P\smallsetminus \alpha$ that has at least three $\tau$-sides or has a puncture has exactly one singularity.
If there is a puncture we place the singularity at the puncture. We then collapse the leaves of the foliation in $P$ to points, for every complementary component $P$,
to obtain $\F(\alpha;\beta, \tau)$. Note that the singularities at punctures may have one or two prongs but all other singularities have at least three.

We also observe that the track $\tau$ can be embedded in $\Sigma$ so that the non-singular leaves of $\cF(\alpha;\beta,\tau)$ are homotopic to train paths in $\tau$ in the complement of the singularities.
\begin{lemma}\label{train track side count}
Let $\tau$ be a recurrent train track, $\alpha$ a curve with $\alpha \pitchfork \tau$, and $\beta \in \S(\tau)$. If $P$ is a component of $\Sigma \smallsetminus \tau$ with $s$ sides and $\cF(\alpha;\beta, \tau)$ has $\ell$ singularities in $P$ with $i_1, \dots, i_\ell$ prongs, then
$$s = 2-2\ell + \sum i_j.$$
\end{lemma}

\begin{proof} The proof is very similar to that of Lemma \ref{side count}.
In the construction of $\cF(\alpha; \beta, \tau)$, we constructed a foliation in $P$ that we collapsed to form $\cF(\alpha; \beta, \tau)$.
As in Lemma \ref{side count}, the quotient of $P$ after collapsing the leaves in $P$ is
a tree, but in this
case there is a vertex for each corner of $P$ along with one for each singularity.
That is, the tree has $s + \ell$ vertices and $\frac12\left(s+ \sum i_j\right)$ edges. The formula then follows as before.
\end{proof}

If $Q$ is a component of $\Sigma \smallsetminus \tau$, then a proper arc $\eta$ in
$Q$ \emph{cuts a corner} if the two points of $\partial \eta$ lie on two incident
sides of
$\partial Q$. The curve $\alpha$  {\em only cuts corners} if each arc of
$\alpha \smallsetminus \tau$ cuts a corner. Note that in this case,
$\alpha \pitchfork \tau$.

\begin{lemma}
Let $\tau$ be a recurrent train track and $\alpha$ a curve that only cuts corners of $\tau$. If $\beta$ is carried by $\tau$ then $\ind(\alpha,\beta) \ge \ind(\tau)$.
\end{lemma}
\begin{proof}
Let $\cF = \cF(\alpha; \beta, \tau)$.
Since $\alpha$ only cuts corners, for each $s$-sided region of $\Sigma\smallsetminus\tau$ the foliation $\cF$ will have a singularity with $s$ prongs and there will be no other singularities. In particular, if there are $k$ singularities with $i_1, \dots, i_k$ prongs then, by Proposition \ref{numerology} and Lemma \ref{train track side count},
$$\ind(\tau) = 2p + \sum(i_j - 3)$$
if $\tau$ is non-orientable and
$$\ind(\tau) = -1 +2p + \sum(i_j - 3)$$
if $\tau$ is orientable.

Then the lemma follows from Corollary \ref{one singularity} if $\alpha$ and $\beta$ are not an orientable pair or if there is a component of $\Sigma\smallsetminus (\alpha\cup\beta)$ that contains more than one singularity. If $\alpha$ and $\beta$ are an orientable pair and each component of $\Sigma\smallsetminus (\alpha\cup\beta)$ contains exactly one singularity then by Lemma \ref{induced orientation} we have that $\cF$, and hence $\tau$, is oriented. Then the lemma again follows from Corollary \ref{one singularity}.
\end{proof}

For the next proof, if $\eta$ is an arc on $\Sigma$ and $p,q\in \eta$, then
$\eta|[p,q]$ denotes the sub-arc of $\eta$ bounded by $p$ and $q$.

\begin{lemma}
  \label{cuttingcorners}
  Let $\tau$ be a filling train track. Then there is a curve $\gamma$ which  only cuts
  corners of $\tau$.
\end{lemma}

\begin{proof}
  The curve $\gamma$ may be constructed in the following way. First we construct a
  (simple) arc $\eta$ as follows. Start in a complementary polygon to $\tau$ and then
  cross one of the sides of the polygon transversely to enter another polygon. In
  this new polygon, cross one of the sides incident to the side crossed by $\eta$, so
  that $\eta$ cuts a corner. Continue this process. In each polygon that $\eta$ enters,
  either an adjacent side is available and $\eta$ can be extended so that
  it cuts a corner, or else it enters between two arcs of $\eta$
  previously constructed. In the latter case extend $\eta$ to be parallel
  to one of these arcs.

  After sufficiently many steps of the construction, $\eta$ intersects a side $s$ of
  $\tau$ three times.
  If there are two points
  $p,q \in \eta \cap s$ which are consecutive along $s$ and such that the tangent vectors to $\eta$
  at $p$ and $q$ are parallel, then the curve
  $\eta|[p,q] \cup s|[p,q]$ may be isotoped to be transverse to $s$, and then it cuts
  corners of $\tau$. Otherwise, choose three points $p,q,r$ of $\eta\cap s$ which are
  consecutive along $\eta$. Up to reversing $s$ and/or $\eta$, there are two possibilities
  for the configuration of the points, and surgeries as in Figure \ref{surgery}
  yield the curve $\gamma$ cutting off corners of $\tau$. Note that we may perform
  the surgery inside of a complementary polygon of $\tau$, so that the curve
  $\gamma$ only cuts corners.
\end{proof}

\begin{figure}[h]

  \centering

  \def\svgwidth{0.9\textwidth}
  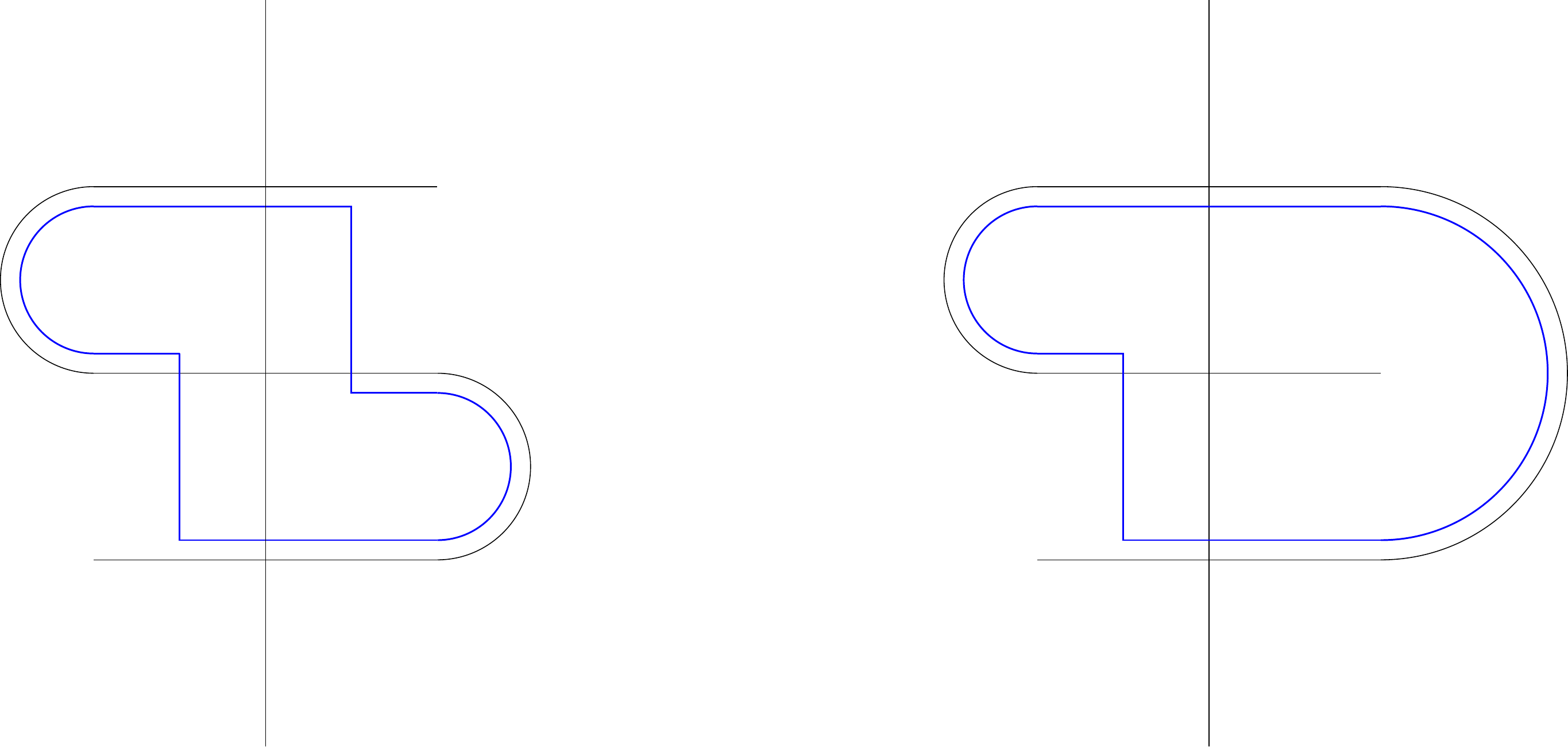

  \caption{Surgeries to obtain a curve cutting off corners. The vertical arc is
    the side $s$ of $\tau$ whereas the horizontal arc is the constructed arc $a$
    cutting off corners.}
  \label{surgery}
\end{figure}

\begin{proof}[Proof of Theorem \ref{ursuladef}]
If $d^H_k(\alpha, \beta) = 1$ then by definition $\ind(\alpha, \beta) > k$ so by Lemma \ref{track index} there is a train track $\tau$ with $\ind(\tau) = \ind(\alpha,\beta) > k$ and $\alpha, \beta \in \S(\tau)$. Therefore $d_k(\alpha,\beta) = 1$.

If $d_k(\alpha, \beta) = 1$ then there is a train track $\tau$ with $\ind(\tau) > k$ and $\alpha, \beta \in \S(\tau)$. By Lemma \ref{cuttingcorners} there is a curve $\gamma$ that only cuts corners of $\tau$. Therefore, by Lemma \ref{lower bound}, $\ind(\alpha, \gamma)$ and $\ind(\beta, \gamma)$ are both $\ge \ind(\tau) > k$. In particular, $d^H_k(\alpha, \gamma) = d^H_k(\beta,\gamma) = 1$ so $d^H_k(\alpha,\beta) \le 2$.
\end{proof}

\section{Elementary properties}

In this section we investigate the basic geometric properties of $\C_k$ and
some of its natural subspaces. We begin with a review of splitting sequences.

Let $\tau$ be a train track. Recall that a \emph{half-branch} is any closed
subsegment of a branch $b$ containing exactly one of the two switches of $b$.
Two half-branches are identified up to the equivalence relation of containing a
common half-branch. A half-branch $h$ of $\tau$ is adjacent to a unique switch $v$.
We say that $h$ is \emph{large} if every train path through $v$ traverses $h$.
Otherwise $h$ is \emph{small}. A branch is called \emph{large}
if both its half-branches are large, \emph{small} if both its half-branches are small,
and \emph{mixed} otherwise.

We say that  a train track $\sigma$ is obtained from $\tau$ by a
\emph{left, central, or right split} (respectively) if $\sigma$ is obtained from
$\tau$ by a move as shown on the left, middle, or right of Figure \ref{split}
(respectively). The split is on the large branch labeled 1 in the figure.
In any of these cases, there is a carrying map $\sigma \to \tau$ which sends switches
to switches.
We call the branch labeled ``1'' on the bottom of Figure \ref{split} the
small branch \emph{created by the split} in case of a left split or right split.
The inverse move to a split is called a \emph{fold} (over the small
branch labeled 1 on the bottom of Figure \ref{split}).
After a left or a right split $\sigma\to \tau$,
there is a natural bijection between the branches of $\sigma$ and those of $\tau$.
The numbers in Figure \ref{split}
indicate that the bijection identifies branches with the same numbers.
The bijection is the ``identity'' on any branch not lying among those five
which are involved in the split. Thus, we will sometimes identify
the branches of $\sigma$ with those of $\tau$ via these bijections.

If $\tau$ is a track with a
large branch $b$ and $\tau_L,\tau_C, \tau_R$ are, respectively, the tracks
obtained by left, central,
and right splits of $\tau$ at $b$, then we have $P(\tau)= P(\tau_L)\cup P(\tau_R)$
and $P(\tau_C)=P(\tau_L)\cap P(\tau_R)$.
By \cite[Lemma 2.1.3]{Penner}, if $\tau$ is recurrent then either $\tau_L,\tau_C,$
and $\tau_R$ are all recurrent or exactly one is. In case exactly one
$\sigma \in \{\tau_L,\tau_C,\tau_R\}$ is recurrent, we have $P(\tau)=P(\sigma)$.
We also observe that if $\tau$ is maximal then either $\tau_L$ or $\tau_R$ will
always be recurrent.
When $\tau$ is recurrent and $\tau_C$ is the only recurrent track in
$\{\tau_L,\tau_C,\tau_R\}$, then it follows from Proposition \ref{dimensions} that
$\tau$ is non-orientable and $\tau_C$ is orientable.

\begin{figure}[h]

  \centering

  \def\svgwidth{0.9\textwidth}
  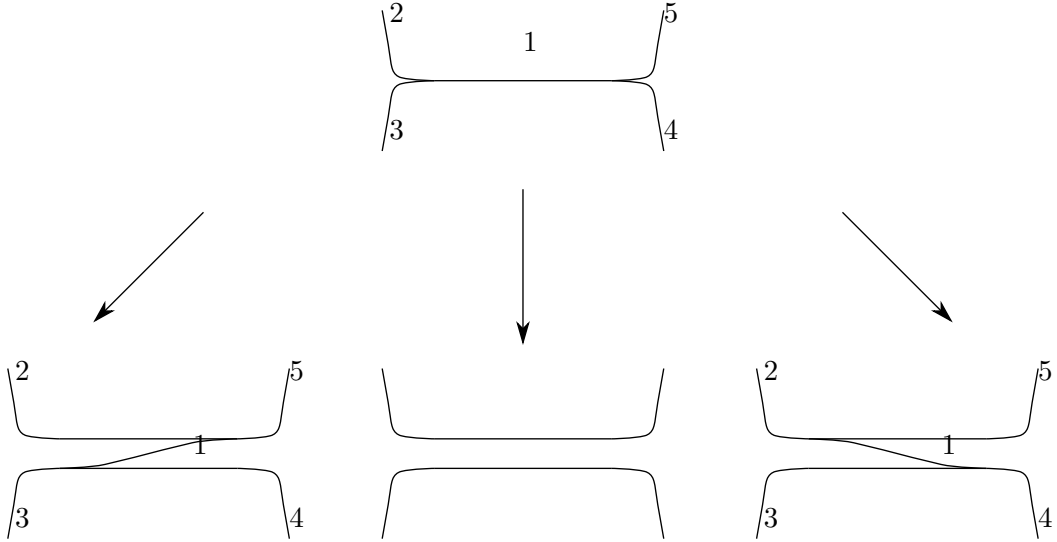

  \caption{The tracks on the bottom are obtained from the track on the top by a left,
    central, and right split, respectively.}
  \label{split}
\end{figure}

A sequence of carrying maps $\tau_0 \to \tau_1 \to \cdots \to \tau_n$ is called a
\emph{splitting sequence} if the carrying map $\tau_i \to \tau_{i+1}$ is induced
by a split on a large branch of $\tau_{i+1}$ for each $i$ and the split is the central split if and only if the central split is the only recurrent split.
If $\sigma \to \tau$ is the carrying map induced by a splitting sequence
(i.e. $\sigma=\tau_0$ and $\tau=\tau_n$ in the above sequence) then
we write $\sigma \ss \tau$. Note that in a splitting sequence at most one split can be a central split.
We will use the following theorem on splitting sequences repeatedly.

\begin{thm}[{\cite[Theorem 1.2]{MMQuasiconv}, \cite[Corollary 2.6]{HamenstadtTT}}]
  \label{quasigeodesics}
  There exists $Q=Q(\Sigma)>0$ and $E =E(\Sigma) >0$ such that if $\cdots \to \tau_2 \to \tau_1 \to \tau_0$ is a
  splitting sequence, then $V(\tau_0),V(\tau_1),V(\tau_2),\ldots$ forms a
  reparametrized $Q$-quasi-geodesic in $\C(\Sigma)$. In particular, $\S(\tau)$
  is $E$-quasi-convex in $\C(\Sigma)$.
\end{thm}

The following proposition of Kapovich-Rafi will play a central role in our work.
\begin{prop}[{Kapovich-Rafi \cite[Proposition 2.5]{KR}}]
\label{KR}
For any positive integers $\delta,L,D$, there exists $\delta'>0$
such that the following holds.
Let $X,Y$ be connected graphs such that $X$ is $\delta$-hyperbolic.
Let $f:X\to Y$ be an $L$-Lipschitz graph map, surjective when restricted
to a map of the vertex sets $\mathcal{V}(X)\to \mathcal{V}(Y)$.
Suppose that the following condition is
satisfied:
for any $x,y\in \mathcal{V}(X)$, if $d(f(x),f(y))\leq 1$ then for any geodesic $[x,y]$
in $X$ we have
\[
\diam_Y(f([x,y])) \leq D.
\]
Then $Y$ is $\delta'$-hyperbolic.
Furthermore
\begin{itemize}
\item The image in $Y$ of a geodesic in $X$ is a reparameterized quasi-geodesic.

\item For any $K$, there is $K' = K'(\delta, L, D, K)$ such that any
$K$-quasi-convex set in $X$ has $K'$-quasi-convex image in $Y$.

\item There exists $E=E(\delta,L,D)$ such that if $x,y \in X$ then
$$f([x,y]) \subset B_E([f(x), f(y)]).$$
\end{itemize}
\end{prop}

In \cite{DT} another formulation of the last bullet is used.
Namely, Dowdall-Taylor define $f\colon X\to Y$ to be {\em alignment-preserving}
if there exists $E$ such that if $a,b,c \in X$ with
$d_X(a,b) + d_X(b,c) = d_X(a,c)$ then
$d_Y(f(a), f(b)) + d_Y(f(b), f(c)) - d_Y(f(a), f(c))\leq E$. If
$X$ and $Y$ are hyperbolic then it is straightforward to check that this
condition is equivalent to the last bullet.

We now see that this proposition applies in our setting.

\begin{prop}\label{hyperbolic}
  There is $\delta>0$ with the following property.
  For every non-sporadic surface $\Sigma$ and every $k\geq 0$, $\mathcal
    C_{k}(\Sigma)$ is $\delta$-hyperbolic. Images of geodesics in $\mathcal
    C(\Sigma)$ are (re-parametrized) quasi-geodesics in $\mathcal
    C_{k}(\Sigma)$ with uniform constants. In particular, all maps
    $\C_{k+1}(\Sigma) \to \C_k(\Sigma)$ are alignment-preserving.
\end{prop}

\begin{proof}
  The natural map $\pi_k:\mathcal C\to \mathcal C_k$ which is the ``identity'' on
  vertices is
  1-Lipschitz and bijective on the vertex sets.
  Moreover, there exists $\delta_0>0$ such that $\mathcal C$ is
  $\delta_0$-hyperbolic, where $\delta_0$ is independent of
  $\Sigma$ (\cite{Aougab}, \cite{Bowditch}, \cite{CRS}, \cite{HPW}).
  So according to
  Proposition \ref{KR}, it suffices to prove that there is $D>0$
  such that
  if $\alpha,\beta$ are adjacent vertices in $\mathcal C_k$ and $p$ is a
  geodesic in $\mathcal C$ from $\alpha$ to $\beta$ then $\diam \pi_k(p)\leq D$.

  Let $\tau$ be a train track of index $>k$ that carries both $\alpha$ and
  $\beta$. By Theorem \ref{quasigeodesics},
  any geodesic from $[\alpha,\beta]$ is contained in a $Q$-neighborhood of
  $\S(\tau)$ in $\C$, so the image of $[\alpha,\beta]$ in $\C_k$ is contained
  in the $Q$-neighborhood of $\S(\tau)$ in $\C_k$. But then $\S(\tau)$ has
  diameter 1 in $\C_k$, so the image of $[\alpha,\beta]$ has diameter
  bounded by $2Q+1$ in $\C_k$.
\end{proof}

Below, we will use the following fact, due to Masur-Minsky, which is a step towards
their Nesting Lemma \cite[Lemma 4.7]{MM}. We will prove this later in the paper.
For the statement, recall that an \emph{extension} of a train track $\tau$
is a train track $\sigma$ with $\tau \subset \sigma$. The extension $\sigma$ is
a \emph{diagonal extension} if every branch of $\sigma \smallsetminus \tau$ joins two switches of
$\tau$. As defined a diagonal extension is not generic as the underlying graphs will not be trivalent where branches have been added.  However, for any non-generic train track $\sigma$ there is a generic train track $\sigma'$ such that $\sigma$ carries $\sigma'$ and $\sigma'$ carries $\sigma$. Furthermore the composed carrying map of $\sigma'$ to itself will be homotopic to a homeomorphism. As we always want our train tracks to be generic we will abuse notation and refer to the track $\sigma'$ as a diagonal extension.

\begin{lemma}[{\cite[Lemma 4.4]{MM}}]\label{MMdiag}
  If $\tau$ is a filling track, $\alpha$ is a simple closed curve fully carried
  by $\tau$, and $\beta$ is a simple closed curve with $i(\alpha,\beta)=0$,
  then $\beta$ is carried by a diagonal extension of $\tau$.
\end{lemma}

We note that in \cite[Lemma 4.4]{MM}, the track $\tau$ is assumed to be birecurrent.
We do not require this property for two reasons. First of all, the fact that $\alpha$
is fully carried by $\tau$ automatically implies that $\tau$ is recurrent. Secondly,
the proof of \cite[Lemma 4.4]{MM} only requires that train paths carried by $\tau$
are quasi-geodesics with constants depending only on $\tau$. Train paths are
quasi-geodesics
whether the track is transversely recurrent or not. See \cite[Lemma 5.8]{casson}
for the case of closed surfaces and \cite{Bonahon},
\cite[Lemma 20]{BZ} for the general case.

\section{The geometry of $\S(\tau)$}

The main result of this section is:

\begin{thm}\label{qtree}
  There exists $Q>0$ with the following property. If $\Sigma$ is any surface
  and $\tau$ is a recurrent track with $ind(\tau)=k$,
  then $S(\tau)$, considered as a subset of $\C_k$, is $Q$-quasi-convex and a
  quasi-tree.
\end{thm}

Throughout the rest of the paper,
for tracks $\tau,\sigma$ on $\Sigma$, we denote by
$d_k(\tau,\sigma)\vcentcolon=\mathrm{diam}_{\C_k}(V(\tau)\cup V(\sigma))$
the \emph{Hausdorff distance} from
$V(\tau)$ to $V(\sigma)$ in $\C_k$. Since the diameter of $V(\tau)$ is uniformly
bounded independently of $\tau$ and $\sigma$, this is coarsely equal to
$d_k(\alpha,\beta)$ for any $\alpha\in V(\tau)$ and $\beta\in V(\sigma)$. Similarly,
if $\alpha$ is a simple closed curve and $\tau$ is a track, we define
$d_k(\alpha,\tau)$ to be $\diam_{\C_k}(\{\alpha\}\cup V(\tau))$.

\begin{proof}[Proof of Theorem \ref{qtree}]
By Theorem \ref{quasigeodesics}, $\S(\tau)$ is quasi-convex in $\C$,
with quasi-convexity constant depending only on $\Sigma$.
It then follows from Propositions \ref{KR} and \ref{hyperbolic} that $\S(\tau)$ is
$Q$-quasi-convex in $\C_k$.

  We will show that $\S(\tau)$ is a quasi-tree using the version of
  Manning's Bottleneck Criterion in Corollary \ref{subspacebottleneck}.
   Let $\alpha,\beta\in \S(\tau)$.

  Split $\tau$ as much as possible while still carrying both
  $\alpha,\beta$; call the resulting track $\tau'$. Split $\tau'$ towards $\alpha$
  and towards $\beta$.
  Concatenating paths of vertex cycles of the split tracks
  after $\tau'$ gives a
  preferred path $p_{\alpha,\beta}$ between $\alpha$ and $\beta$ in $S(\tau)$.
  See Figure \ref{fig:tripod}.

  \begin{figure}[h]

  \centering

  \def\svgwidth{0.25\textwidth}
\begingroup%
  \makeatletter%
  \providecommand\color[2][]{%
    \errmessage{(Inkscape) Color is used for the text in Inkscape, but the package 'color.sty' is not loaded}%
    \renewcommand\color[2][]{}%
  }%
  \providecommand\transparent[1]{%
    \errmessage{(Inkscape) Transparency is used (non-zero) for the text in Inkscape, but the package 'transparent.sty' is not loaded}%
    \renewcommand\transparent[1]{}%
  }%
  \providecommand\rotatebox[2]{#2}%
  \newcommand*\fsize{\dimexpr\f@size pt\relax}%
  \newcommand*\lineheight[1]{\fontsize{\fsize}{#1\fsize}\selectfont}%
  \ifx\svgwidth\undefined%
    \setlength{\unitlength}{689.43058525bp}%
    \ifx\svgscale\undefined%
      \relax%
    \else%
      \setlength{\unitlength}{\unitlength * \real{\svgscale}}%
    \fi%
  \else%
    \setlength{\unitlength}{\svgwidth}%
  \fi%
  \global\let\svgwidth\undefined%
  \global\let\svgscale\undefined%
  \makeatother%
  \begin{picture}(1,1.37509666)%
    \lineheight{1}%
    \setlength\tabcolsep{0pt}%
    \put(0,0){\includegraphics[width=\unitlength,page=1]{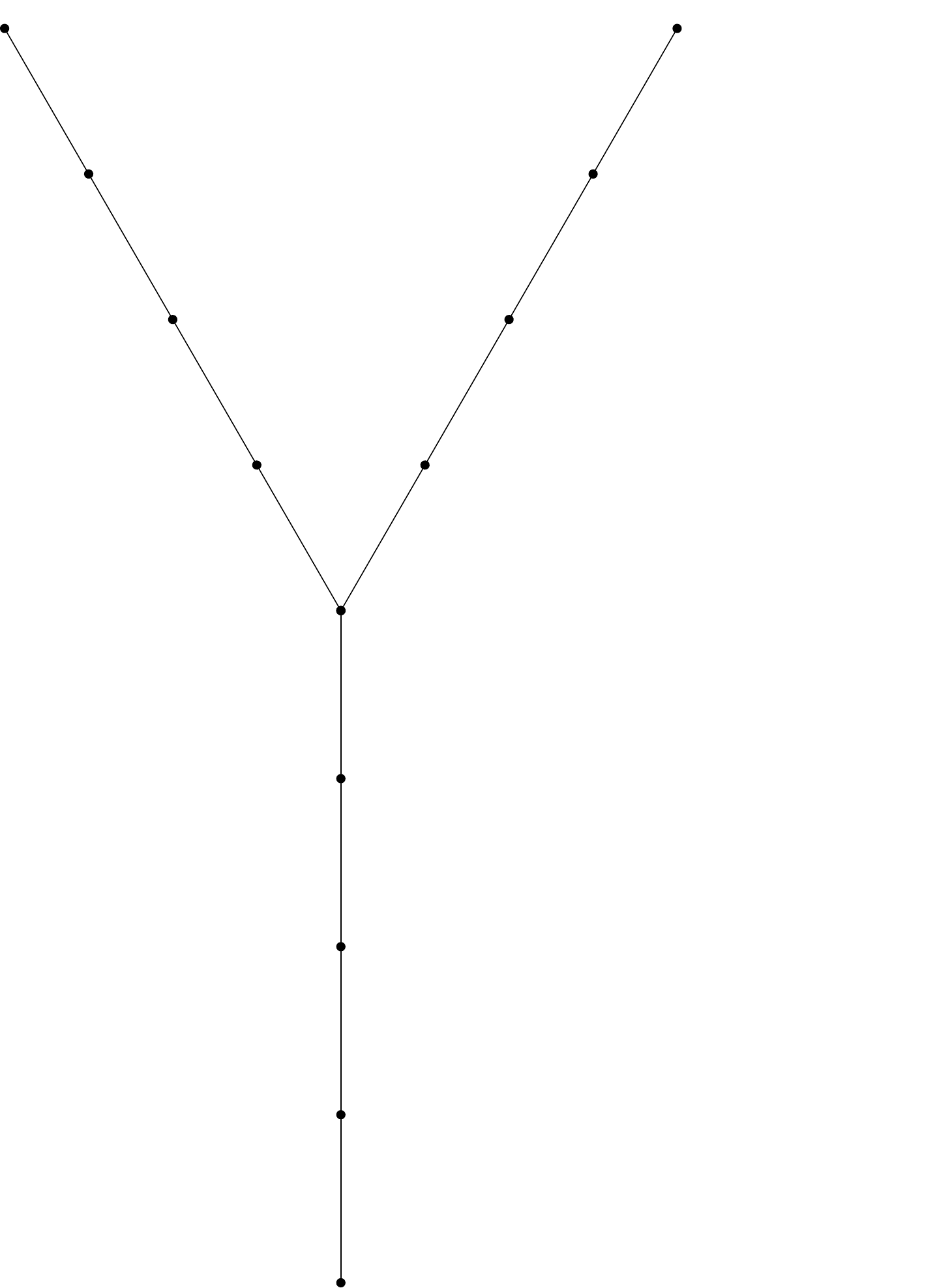}}%
    \put(0.14329351,1.34462826){\color[rgb]{0,0,0}\makebox(0,0)[lt]{\lineheight{1.25}\smash{\begin{tabular}[t]{l}\textit{$\tau''$}\end{tabular}}}}%
    \put(0.86127732,1.34462826){\color[rgb]{0,0,0}\makebox(0,0)[lt]{\lineheight{1.25}\smash{\begin{tabular}[t]{l}\textit{$\tau'''$}\end{tabular}}}}%
    \put(0.50228542,0.00485221){\color[rgb]{0,0,0}\makebox(0,0)[lt]{\lineheight{1.25}\smash{\begin{tabular}[t]{l}\textit{$\tau$}\end{tabular}}}}%
    \put(0.50228542,0.72283606){\color[rgb]{0,0,0}\makebox(0,0)[lt]{\lineheight{1.25}\smash{\begin{tabular}[t]{l}\textit{$\tau'$}\end{tabular}}}}%
    \put(0,0){\includegraphics[width=\unitlength,page=2]{tripod.pdf}}%
  \end{picture}%
\endgroup%

  \caption{Constructing a path of curves from $\alpha$ to $\beta$. The curve $\alpha$
  is a vertex cycle of $\tau''$ while $\beta$ is a vertex cycle of $\tau'''$. Each
  edge of the tripod represents a split and the tracks along the leg from $\tau'$ to
  $\tau''$ carry $\alpha$ but not $\beta$ while the tracks from $\tau'$ to $\tau'''$
  carry $\beta$ but not $\alpha$. The blue path is the path $p_{\alpha,\beta}$
  from $\alpha$ to $\beta$, obtained by taking vertex cycles of tracks along two
  legs of the tripod.}
  \label{fig:tripod}
\end{figure}

  Now let
  $\alpha,\ldots,\gamma,\delta,\ldots,\beta$ be a $(2Q+1)$-connected path in $\mathcal
    S(\tau)$,

  We need to argue that it comes within a uniformly bounded
  distance of any track $\tau_\alpha\neq\tau'$ between $\tau'$ and $\alpha$, and
  similarly for tracks $\tau_\beta$ between $\tau'$ and $\beta$.
  The case for $\tau_\beta$ will follow by symmetry. So
  choose $\gamma$ to be the
  last curve in the sequence carried by $\tau_\alpha$.

   Interpolate a geodesic path
  in $\mathcal C$ between $\gamma$ and $\delta$, say
  $\gamma=\xi_0,\ldots,\xi_m=\delta$. Thus, by Proposition \ref{hyperbolic}, all these curves
  are uniformly close in $\mathcal C_k$.

Since $\delta$ is
carried by $\tau$ but not carried by $\tau_\alpha\to \tau$, by Lemma \ref{diagonal extension} below, $\delta$ is 
not carried by any diagonal extension of $\tau_\alpha$. 
Therefore there is a $\xi_i$, with $i<
m$, such that $\xi_i$ is carried by a diagonal extension
$\tau'_\alpha$ of $\tau_\alpha$ but $\xi_{i+1}$ is not carried by any
diagonal extension of $\tau_\alpha$. We can further assume that $\tau'_\alpha$ is the smallest diagonal extension of $\tau_\alpha$ that carries $\xi_i$.

There are now two cases. First we assume that $\tau'_\alpha = \tau_\alpha$. (This will always happen if $\ind(\tau) =\ind(\tau_\alpha) = 0$.) Then by Lemma \ref{MMdiag}, $\xi_i$ is not fully carried by $\tau_\alpha$. That is, $\xi_i$ is carried on a proper subtrack $\sigma \subset \tau_\alpha$ and therefore $\ind(\sigma) > \ind(\tau_\alpha)=k$. As $\sigma$ carries a vertex cycle of $\tau_\alpha$ and the set of curves carried by $\sigma$ has diameter one in $\C_k$, we have that $d_k(\tau_\alpha, \xi_i)$ is bounded. Since $d_k(\gamma, \xi_i)$ is bounded, we have that $d_k(\gamma, \tau_\alpha)$ is bounded, completing the proof in this case.

Now assume that $\tau_\alpha$ is a proper subtrack of $\tau'_\alpha$.
By Lemma \ref{MMdiag}, $\xi_i$ is fully carried by $\tau'_\alpha$. Since $\xi_i$ traverses every branch in $\tau'_\alpha \smallsetminus \tau_\alpha$ (by the minimality assumption) we have that $\xi_i$ doesn't traverse every branch of $\tau_\alpha$. In particular, if $\sigma$ is the smallest subtrack of $\tau'_\alpha$ that carries $\xi_i$, then $\sigma$ is not a subtrack of $\tau_\alpha$ nor is $\tau_\alpha$ a subtrack of $\sigma$.

Finally apply Corollary \ref{subtrack
  distance} below to the tracks $\tau_\alpha$ and $\sigma$. As the tracks are not nested, there is a
subtrack $\sigma' \subset \tau'_\alpha$ of index $ \ge k+1$ and a curve
$\gamma'$ carried by $\sigma'$ such that the geodesic $\gamma=\xi_0,
\xi_1, \dots, \xi_i$ passes uniformly close to $\gamma'$. That is,
there is a $\xi_j$ such that $d(\gamma',\xi_j)$ (which is $\ge d_k(\gamma',
\xi_j)$) is bounded. 
Then
$$d_k(\gamma, \tau_\alpha) \le d_k(\gamma, \xi_j) + d_k(\xi_j, \gamma') + d_k(\gamma', \sigma') + d_k(\sigma', \tau_\alpha)$$
is bounded since all the terms on the right are bounded (the third
summand is bounded since the set of curves carried by $\sigma$ has diameter one in $\C_k$). This completes the proof in the general case modulo the missing corollary and lemma.
\end{proof}

We now state and prove the missing corollary and lemma.
Corollary \ref{subtrack distance} follows from a distance estimate from \cite{BB} which we state here:
\begin{prop}[{\cite[Corollary 3.24]{BB}}\label{3.24}]
  For every $C>0$ there exists $C'= C'(\Sigma, C)>0$ such that the following holds.
  Let $\tau$ be a filling train track on $\Sigma$ and let $\sigma_1$ and $\sigma_2$ be subtracks. If $\alpha \in \S(\sigma_1)$ and $\beta\in \S(\sigma_2)$ with $d(\alpha, \beta) \le C$ then either
  \begin{enumerate}[(1)]
    \item $d(\alpha, \tau), d(\beta,\tau) \le C'$ or
    \item there exists $\gamma \in \S(\sigma)$, where $\sigma = \sigma_1 \cap \sigma_2$, such that $d(\alpha, \gamma), d(\beta, \gamma) \le C'$.
  \end{enumerate}
\end{prop}

We will need a more refined version of this proposition later.

\begin{cor}\label{subtrack distance}
 For every $C$ there exists $C'= C'(\Sigma, C)$ such that the
 following holds.  Let $\tau$ be a recurrent train track on $\Sigma$ and let
 $\sigma_1$ and $\sigma_2$ be recurrent subtracks with
 $\sigma_1\not\subseteq\sigma_2$ and $\sigma_2\not\subseteq\sigma_1$. If
 $\alpha \in \S(\sigma_1)$ and $\beta\in \S(\sigma_2)$ with $d(\alpha, \beta) \le C$ then every geodesic between $\alpha$ and $\beta$ contains $\gamma'$ such that
 either
 \begin{enumerate}[(1)]
 \item $d(\gamma', \tau) \le C'$ or
 \item there exists
 a subtrack $\sigma \subset \tau$ with $\ind(\sigma) > \max\{\ind(\sigma_1), \ind(\sigma_2)\}$
and a curve $\gamma \in
 \S(\sigma)$ such that $d(\gamma, \gamma') \le C'$.
\end{enumerate}
\end{cor}

\begin{proof}
If $(X,d_X)$ is a $\delta$-hyperbolic metric space and $A, B\subset X$ are $K$-quasi-convex sets with $\inf\{d(x,y):x\in A, y\in B\} \le D$, then for any geodesic in $X$ with one endpoint in $A$ and the other in $B$, there is a point $c$ on the geodesic with $d_X(c, A)$ and $c_X(c, B)$ both bounded by $D' = D'(\delta, K, D)$.
That is, there are points $a \in A$ and $b \in B$ with $d_X(a,c)$ and $d_X(b, c)$ both $\le D'$.

We now apply this to our situation. The sets $\S(\sigma_1)$ and $\S(\sigma_2)$  both intersect $V(\tau)$ so $\inf\{d(x,y): x\in \S(\sigma_1), y\in \S(\sigma_2)\} \le D$ where $D$ is the (uniform) bound on the diameter of $V(\tau)$.
Therefore any geodesic in $\C$ between $\alpha$ and $\beta$ contains a curve $\gamma'$ such that  there exist $\alpha' \in \S(\sigma_1)$ and $\beta' \in \S(\sigma_2)$ with $d(\gamma', \alpha'), d(\gamma', \beta')\le D'$. Then $d(\alpha', \beta') \le 2D'$. Let $D''$ be the $C'$ constant from Proposition \ref{3.24} (for $C = 2D'$).
 Then either $d(\alpha', \tau), d(\beta',\tau) \le D''$
or there is a track $\sigma$ that is a (necessarily proper) subtrack of both $\sigma_1$ and $\sigma_2$ and a curve $\gamma \in \S(\sigma)$ such that $d(\gamma, \alpha'), d(\gamma, \beta') \le D''$. Then $C' = D' + D''$ is our desired constant.
\end{proof}

\begin{lemma}\label{diagonal extension}
  Suppose $\tau\ss\sigma$ and both tracks are recurrent.
  If $\tau'$ is a diagonal
  extension of $\tau$ then there is a diagonal extension $\sigma'$ of
  $\sigma$ with $\tau'\ss\sigma'$.
  Conversely, if $\sigma$ is a subtrack of track $\sigma'$ then $\tau$ is a subtrack of a track $\tau'$ with $\tau'\ss\sigma'$
  Finally, if $\alpha$ is a curve carried by both $\tau'$
  and $\sigma$ then it is carried by $\tau$.
\end{lemma}

\begin{proof}
  For the first part, it suffices to suppose $\tau$ was obtained from
  $\sigma$ by one split. Suppose the split is a left split. Let $b$
  be the small branch of $\tau$ created by the split. We define $\sigma'$ as
  follows. For any branches of $\tau' \smallsetminus \tau$ with endpoints
  which lie on $b$ and are associated to the southwest corner of $b$,
  we isotope them to have endpoints on the branch lying southwest from
  $b$. Similarly, for any branches of $\tau' \smallsetminus \tau$ with
  endpoints which lie on $b$ and are associated to the northeast corner
  of $b$, we isotope them to have endpoints on the branch northeast to
  $b$. If there are $i$ such branches for the southwest corner and $j$
  such branches for the northeast corner, then these isotopies
  correspond to $i + j + i\cdot j$ folds of $\tau'$.
  Call the resulting track $\tau''$. Finally, $b$ is a branch of
  $\tau''$ and we fold it to form $\sigma'$. The track $\sigma'$
  is a diagonal extension of $\sigma$. See Figure \ref{splittingdiagonal}.
\begin{figure}[h]

  \centering

  \def\svgwidth{0.9\textwidth}
  \includegraphics[width=\textwidth]{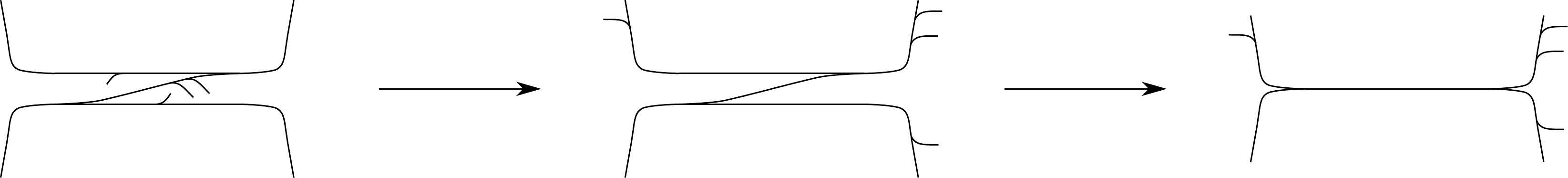}

  \caption{Diagonal extensions push forward under folds. The track on the left is $\tau$ and the track on the right is $\sigma$. The short branches are branches of the diagonal extensions.}
  \label{splittingdiagonal}
\end{figure}

  For the second statement, it again suffices to consider the case when $\tau\ss\sigma$ is a single left split on large branch $b$ in $\sigma$. The branch $b$ of $\sigma$ will be a union of branches of $\sigma'$ with at least one large branch $b'$. We set $c(b;\sigma')$ to be the number of train paths in $\sigma'$ that are the union of a small half branch in $\sigma'\smallsetminus\sigma$ and branches in $b$. We observe that $c(b;\sigma') =0$ if and only if $b= b'$ is a single branch in $\sigma'$. In this case we set $\tau' \ss \sigma'$ to be the left split on $b'$ and observe that $\tau$ is a subtrack of $\tau'$.
  
If $c(b; \sigma') > 0$ then there is a split $\sigma''\ss\sigma'$ such that $\sigma$ is a subtrack of $\sigma''$ and the small branch of $\sigma''$ that is created by the split is contained in $\sigma$. This implies that $\sigma''$ is recurrent. Furthermore we have that $c(b, \sigma'') < c(b, \sigma')$. We repeat this until the complexity is zero and then apply the first case.

  For the last part,
  $P(\tau)$ is a face of $P(\tau')$ and these are codimension 0 subsets of
  $P(\sigma)\subset P(\sigma')$, also a face inclusion. So
  $P(\sigma)\cap P(\tau')=P(\tau)$.
\end{proof}

\section{Progress in $\C_k$}
\label{sec:progress}

As we have seen, vertex cycles of splitting sequences of train tracks form
reparametrized quasi-geodesics in $\C_k$. However, since these
are reparametrized quasi-geodesics, they may not make linear progress in $\C_k$
or could even be bounded.
In this section we describe a criterion that guarantees that a splitting
sequence of train tracks of index $k$ makes progress in $\C_k$,
analogously to the Masur-Minsky Nesting Lemma. See Theorem
\ref{progress2}. 

\begin{definition}
  We write $\tau\twoheadrightarrow\tau'$ if $\tau\ss\tau'$ and
  further, for any branch $b$ of $\tau'$,
  the union of those branches of
  $\tau$ that don't traverse $b$ carries no curves.
\end{definition}

Equivalently, every curve carried by $\tau$ is fully carried by $\tau'$, or $P(\tau)$ is contained in the interior of $P(\tau')$.

\newcommand{\tac}{{(\textasteriskcentered)}}
\begin{lemma}\label{allsplit}
If $\tau\twoheadrightarrow\tau'$ then every large branch of $\tau'$ is split in any splitting sequence between $\tau$ and $\tau'$
\end{lemma}

\begin{proof}
Let $b'$ a large branch of $\tau'$ that is not split in some splitting sequence $\tau\ss\tau'$. 
Then the pre-image of $b'$ and its adjacent half branches in $\tau$ is a large branch $b$ and its adjacent half branches. We will show that $\tau\not\twoheadrightarrow \tau'$.
We first assume that:
\begin{itemize}
\item[\tac] there is a curve $\alpha \in \S(\tau)$ that doesn't traverse all four half branches adjacent to $b$.
\end{itemize}
 Therefore the image of $\alpha$ in $\tau'$ is disjoint from one of the half branches adjacent to $b'$ and $\tau\not\twoheadrightarrow\tau'$.

The proof when $b$ does not satisfy $\tac$ is more involved and we prove it in the next sequence of lemmas.

\begin{lemma}\label{bad_branch}
Let $\tau$ be a recurrent train track and $b$ a large branch that doesn't satisfy $\tac$. 
\begin{enumerate}
\item Let $t$ be a train path that starts and ends at half branches adjacent to $b$ but is disjoint from $b$. Then
\begin{enumerate}
\item the two half branches are distinct but on the same side of $b$;
\item  $t$ is embedded in $\tau$.
\end{enumerate}

\item The track $\tau$ is non-orientable and the only recurrent split along $b$ is the central split which gives an orientable track.

\item If $\tau'\ss\tau$, but $b$ is not split in the splitting sequence, then $\tau'\not\twoheadrightarrow\tau$.
\end{enumerate}
\end{lemma}

\begin{proof}
Let $h_0^+, h_0^-, h_1^+, h_1^-$ be the half branches adjacent to $b$
with the $h_0^\pm$ adjacent to one half branch of $b$ and the
$h_1^\pm$ adjacent to the other half branch of $b$. Then $t$
starts and end at the $h$'s but does not intersect them anywhere
else. If $t$ starts at one of $h_0^\pm$ but ends at one of the
$h_1^\pm$ then we can add $b$ to $t$ to form a closed train path that
doesn't contain all of the $h$'s, contradicting that $b$ doesn't have
$\tac$. If $t$ starts and ends at, say, $h_0^+$ then we can join it
along $b$ to a train path that starts and ends in the $h^\pm_1$ again forming a closed train path not containing all of the $h$'s. This gives (1a)

If $t$ was not embedded than either it would contain a closed train path that didn't contain any of the $h$'s, contradicting that $b$ doesn't have $\tac$, or we could form a train path starting and ending at the the same half branch in the $h$'s, contradicting (1a).

Property (2) follows from (1a).

The pre-image of $b$ in $\tau'$ is a large branch $b'$ that doesn't have $\tac$. In particular, every vertex cycle $\gamma'$ of $\tau'$ is a barbell that is the union of $b'$ and train paths $t'_0$ and $t'_1$ that are disjoint from $b'$. The image of $t'_i$ in $\tau$ is a train path $\t_i$ that is disjoint from $b$ so by (1b) the $t_i$ are embedded. The image $\gamma$ of $\gamma'$ in $\tau$ is the union of $t_0$, $t_1$ and $b$ so $\gamma$ is also a vertex cycle and therefore $\tau'\not\twoheadrightarrow\tau$
\end{proof}

\begin{lemma}\label{bad_branch_stays}
Let $\tau$ be a train track that has a large branch $b$ without property $\tac$. If $\tau'$ is another track and $\tau\ss \tau'$ then $\tau'$ has a large branch $b'$ without property $\tac$ and the pre-image of $b'$ in $\tau$ is $b$.
\end{lemma}

\begin{proof}
It suffices to check this when $\tau\ss\tau'$ is a single split. Note that under a single split the image of a large branch in $\tau$ is a single branch in $\tau'$. Let $b'$ be the image of $b$. If $b'$ is not large then it must be adjacent to a large branch where the splitting occurs and after the splitting $b$ is adjacent to a small branch in $\tau$ that contains a small half branch $h$ that is not adjacent to $b$. Then there are two other half branches adjacent to $h$ - a small half branch $h_0$ and a large half branch $h_1$. Let $t$ be a train path, as in (1) of Lemma \ref{bad_branch}, that contains $h_0$. As $t$ must also contain $h$ we have that $t$ crosses $h_1$ at least twice, contradicting (1b) of Lemma \ref{bad_branch}. Therefore $b'$ is a large branch and the splitting $\tau\to\tau'$ must occur on some other large branch of $\tau'$. 
\end{proof}

We now finish the proof of Lemma \ref{allsplit}. We are left with the case when the branch $b$ doesn't satisfy $\tac$. Then by Lemma \ref{bad_branch_stays} the branch $b'$ also doesn't have property $\tac$ and by (3) of Lemma \ref{bad_branch} we have that $\tau\not\twoheadrightarrow\tau'$.
\end{proof}

\begin{lemma}\label{nodouble}
If $\ind(\tau) = \ind(\tau') = k$, and $\tau\ss\tau'$ but $\tau\not\twoheadrightarrow\tau'$, then there are proper subtracks $\sigma \subset \tau$ and $\sigma' \subset \tau'$ with $\sigma \to \sigma'$. In particular, $d_k(\tau,\tau')$ is bounded.
\end{lemma}
\begin{proof}
Since $\tau\not\twoheadrightarrow\tau'$, there is a vertex cycle $\alpha$ of $\tau$ whose image in $\tau'$ misses a branch $b$. Let $\sigma \subset \tau$ be the smallest subtrack of $\tau$ that carries $\alpha$. (The track $\sigma$ is either an embedded simple closed curve or a barbell.) The image of $\sigma$ in $\tau'$ is a proper subtrack $\sigma'$ as $\alpha$ doesn't traverse $b$. This proves the first statement.

Let $\alpha'$ be a vertex cycle of $\sigma'$ and hence a vertex cycle of $\tau'$. Since both $\alpha$ and $\alpha'$ are carried by $\sigma'$ and $\ind(\sigma') > \ind(\tau') = k$ we have $d_k(\alpha, \alpha') = 1$.
This bounds $d_k(\tau, \tau')$.
\end{proof}

We would like to extend the defining property of $\twoheadrightarrow$
to diagonal extensions. For this, we need long enough compositions.

\begin{lemma}\label{extension carry}
There exists $M = M(\Sigma)>0$ such that the following holds. Suppose
$$\tau_0 \twoheadrightarrow \tau_1 \twoheadrightarrow \cdots \twoheadrightarrow \tau_M$$
and that $\alpha$ is carried by a diagonal extension of $\tau_0$. Then $\alpha$ is fully carried by a diagonal extension of $\tau_M$.
\end{lemma}

\begin{proof}
Assume that $\alpha$ is carried by a diagonal extension $\tau'_0$ of $\tau_0$. Note that if $\tau'_0 = \tau_0$ then $\alpha$ is fully carried by $\tau_1$ since $\tau_0\twoheadrightarrow\tau_1$. So we assume that $\tau_0$ is a proper subtrack of $\tau'_0$.
 Then by Lemma \ref{diagonal extension} there are diagonal extensions $\tau'_i$ of $\tau_i$ such that
$$\tau'_0\ss \tau'_1 \ss \cdots \ss \tau'_M.$$
The curve $\alpha$ determines a trainpath in each $\tau'_i$ and these
paths can be decomposed into subpaths that alternate between paths in
$\tau_i$ and branches in $\tau'_i \smallsetminus \tau_i$. For each $i$
we can choose a maximal subpath $\gamma_i$ that lies in $\tau_i$ and such that the composition of the train path $\gamma_i$ with the carrying map $\tau'_i \to \tau'_{i+1}$ factors through an inclusion of $\gamma_i$ in $\gamma_{i+1}$. We then observe that the number of branches that $\gamma_i$ traverses increases at each stage, since every large branch of $\tau_{i+1}$ is split in $\tau_i$ by Lemma \ref{allsplit}. Therefore, if $M$ is sufficiently large, the image of $\gamma_{M-1}$ in $\tau_{M-1}$ contains a closed curve. Since $\tau_{M-1}\twoheadrightarrow \tau_M$, $\gamma_M$ traverses every branch of $\tau_M$ so $\alpha$ is fully carried by a diagonal extension of $\tau_M$.
\end{proof}

\begin{thm}\label{progress2}
  There is $p=p(\Sigma)$ such that the following holds.
  Suppose
  $$\tau_0\twoheadrightarrow\tau_1\twoheadrightarrow\cdots\twoheadrightarrow\tau_{n\cdot p}$$
  and all these tracks are recurrent, filling, and have index $k$.
  Assume $\beta$ is carried by a diagonal extension of $\tau_0$ while $\alpha$
  is not carried by any diagonal extension of
  $\tau_{p\cdot n}$. Then $d_k^H(\alpha,\beta)> n$.
\end{thm}

If $n=1$ and if we replace $d^H_k$ with the usual curve complex distance then this is, essentially, the Masur-Minsky nesting lemma (Lemma \ref{MMdiag}). In fact, the general case of Theorem \ref{progress2} follows easily from the $n=1$ case. To prove this, we want to homotope a curve so that it intersects a train track {\em efficiently}. For example, if $\alpha$ is carried by $\tau$ then it can be homotoped to a train path while if $\alpha$ hits  $\tau$ efficiently it can be homotoped so that there are no bigons between $\alpha$ and any train path. We want a definition of efficient that includes both of these cases but also includes curves that are neither carried by $\tau$ nor hit $\tau$ efficiently. Note that in both examples a lift of $\alpha$ to the universal cover will intersect any train path in a connected set.
We can then take this as the definition of efficient position. That is, a simple closed curve $\alpha$ is in {\em efficient position} with $\tau$ if any lift $\alpha$ to the universal cover $\tilde\Sigma$ intersects the lift of any train path in $\tau$ in a connected set.

 In \cite{MMS} it is shown that every simple closed curve can be homotoped to one in efficient position.
 While we will not use this result directly, we use many of the ideas behind its proof which we develop now.

Let $\alpha$ be a simple closed curve in $\Sigma$ and
$\tau$ a recurrent, filling, generic train track.
We let $\tilde\tau$ be the preimage of
$\tau$ in the universal cover $\tilde\Sigma$ of $\Sigma$ and let
$\tilde\alpha$ be a lift of $\alpha$ to $\tilde \Sigma$.

Recall that bi-infinite train paths in $\tilde\tau$ are
quasi-geodesics and have well-defined and distinct endpoints on the Gromov boundary
$\partial \tilde \Sigma = S^1_\infty$.
Consider a complementary component of a bi-infinite train
path. The boundary at infinity of this component is an interval and
we suppose that this interval does not
contain either endpoint of $\tilde\alpha$ in its interior. Thus such an interval is
contained in one of the two closed complementary components $A,B$ of
$\partial\tilde\alpha$ in $S^1_\infty$. Let $\tilde{A}$ be the union of
those closed complementary components of train paths
whose boundaries are contained in $A$, and
similarly define $\tilde{B}$. Both $\tilde{A}$ and $\tilde{B}$ are invariant
under the deck transformations that preserve $\tilde\alpha$. The
intersection $\tilde{A}\cap\tilde{B}$ is either empty, or a collection
of finite train paths, or a bi-infinite train path. See Figure
\ref{fig:polygonal_region}.

\begin{figure}[h]
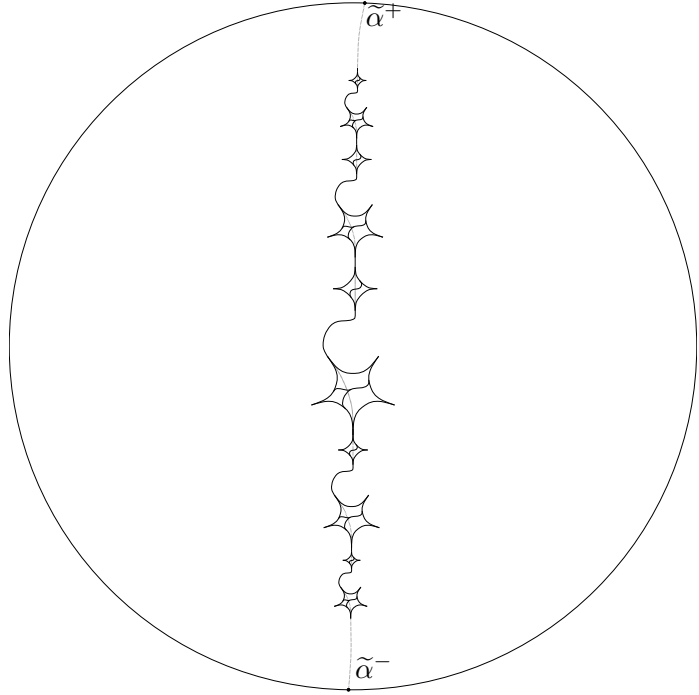


  \centering

  \def\svgwidth{0.6\textwidth}
\begingroup%
  \makeatletter%
  \providecommand\color[2][]{%
    \errmessage{(Inkscape) Color is used for the text in Inkscape, but the package 'color.sty' is not loaded}%
    \renewcommand\color[2][]{}%
  }%
  \providecommand\transparent[1]{%
    \errmessage{(Inkscape) Transparency is used (non-zero) for the text in Inkscape, but the package 'transparent.sty' is not loaded}%
    \renewcommand\transparent[1]{}%
  }%
  \providecommand\rotatebox[2]{#2}%
  \newcommand*\fsize{\dimexpr\f@size pt\relax}%
  \newcommand*\lineheight[1]{\fontsize{\fsize}{#1\fsize}\selectfont}%
  \ifx\svgwidth\undefined%
    \setlength{\unitlength}{6326.64885273bp}%
    \ifx\svgscale\undefined%
      \relax%
    \else%
      \setlength{\unitlength}{\unitlength * \real{\svgscale}}%
    \fi%
  \else%
    \setlength{\unitlength}{\svgwidth}%
  \fi%
  \global\let\svgwidth\undefined%
  \global\let\svgscale\undefined%
  \makeatother%
  \begin{picture}(1,1.00327865)%
    \lineheight{1}%
    \setlength\tabcolsep{0pt}%
    \put(0,0){\includegraphics[width=\unitlength,page=1]{polygonal_region.pdf}}%
    \put(0.51879396,0.96734798){\color[rgb]{0,0,0}\makebox(0,0)[lt]{\lineheight{1.25}\smash{\begin{tabular}[t]{l}\textit{$\widetilde{\alpha}^+$}\end{tabular}}}}%
    \put(0,0){\includegraphics[width=\unitlength,page=2]{polygonal_region.pdf}}%
    \put(0.50308534,0.01927435){\color[rgb]{0,0,0}\makebox(0,0)[lt]{\lineheight{1.25}\smash{\begin{tabular}[t]{l}\textit{$\widetilde{\alpha}^-$}\end{tabular}}}}%
  \end{picture}%
\endgroup%

  \caption{The intersection $\tilde{A} \cap \tilde{B}$ consists of the periodic collection
  of train paths connecting 4-sided and 5-sided polygons in the figure above.
  The branches inside the 4-sided and 5-sided polygons \emph{cross} $\alpha$
  in this case.}
  \label{fig:polygonal_region}
\end{figure}

A branch $\tilde b$ of $\tilde\tau$ {\em carries} $\tilde\alpha$ if it contained in $\tilde{A} \cap \tilde{B}$ and it {\em crosses} $\tilde\alpha$ if it is not in $\tilde{A}\cup \tilde{B}$. Equivalently, a branch $\tilde b$ crosses $\tilde\alpha$ if every bi-infinite train path through $\tilde b$ separates the endpoints of $\tilde \alpha$.
We then define a branch $b$ of $\tau$ to carry $\alpha$ if there is a branch $\tilde b$ in its pre-image that carries some lift of $\alpha$. We similarly define when a branch crosses $\alpha$.
Observe that if $\tilde b$ carries $\tilde\alpha$ but crosses some lift $\tilde\alpha_0$ then $\tilde\alpha$ and $\tilde\alpha_0$ will intersect, contradicting that $\alpha$ is simple.
Therefore a branch $b$ of $\tau$ cannot both carry $\alpha$ and cross it.

A train path is {\em carried} ({\em crossing}) if it is a union of carried (crossing) branches.

\begin{lemma}\label{traverse set}
If the set of branches that carry $\tilde\alpha$ is a bi-infinite train path then $\alpha$ is carried by $\tau$. Any finite maximal train path that carries $\tilde\alpha$ begins and ends at a large half branch. In particular, if the set of branches that carry $\tilde\alpha$ is non-empty then it contains a large branch.
\end{lemma}

\begin{proof}
The set of branches that carry $\tilde\alpha$ is invariant under the stabilizer of $\tilde\alpha$, so if it is a bi-infinite train path, it will descend to a curve homotopic to $\alpha$. Therefore $\alpha$ is carried by $\tau$.

If a small half branch of $\tilde\tau$ carries $\tilde\alpha$ then the adjacent large half branch (and therefore the branch that contains it) also carries $\tilde\alpha$. Therefore any finite maximal transverse train path begin and end at a large half branch. Any train path that begins and ends at a large half branch contains a large branch.
\end{proof}

Let $\Gamma_\alpha \cong \Z$ be the subgroup of the deck group that stabilizes $\tilde\alpha$.

\begin{lemma}\label{connected}
If $t$ is a bi-infinite train path in $\tilde\tau$ then the set of branches in $t$ that cross $\tilde\alpha$ is connected.

The number of $\Gamma_\alpha$-orbits of branches that cross $\tilde\alpha$ is finite.
\end{lemma}

\begin{proof} 
Orient $t$ so that its initial endpoint is in $A$ and terminal endpoint is in $B$.
Assume that the branch $\tilde b$ is in $t \cap\tilde{A}$. Then there is a bi-infinite train path $t'$ that contains $\tilde b$ and has both endpoints in $A$. Then we can make a new bi-infinite train path $t''$ by splicing the first half of $t$ with half of $t'$. Then $t'' \subset \tilde{A}$ so every branch of $t$ that occurs before $\tilde b$ is in $\tilde{A}$. A similar statement holds if $\tilde b \in t \cap \tilde{B}$. This proves connectivity of the set of branches of $t$ that cross $\tilde \alpha$.

Let $\beta \neq \alpha$ be a simple closed curve that is fully carried by $\tau$. The train path in $\tau$ representing $\beta$ lifts to a collection of bi-infinite train paths $\cP$ in $\tilde\tau$. Let $\cP' \subset \cP$ be the subset of these paths that separate the endpoints of $\tilde\alpha$. Since $\beta$ is fully carried by $\tau$, every branch $\tilde b$ in $\tilde\tau$ is contained in some $t \in \cP$ and if $t \in \cP\smallsetminus \cP'$ then $\tilde b \in \tilde{A} \cup\tilde{B}$ and hence $\tilde b$ does not cross $\tilde\alpha$. In particular, if a branch crosses $\tilde\alpha$ then it is contained in a train path in $\cP'$.

Now fix $t \in \cP'$ with initial endpoint in $A$ and terminal endpoint in $B$ and assume that $t$ contains a branch $\tilde b$ that crosses $\tilde\alpha$.
As every bi-infinite train path contains a large branch,
there is a large branch in $t$ and there will be another bi-infinite train path $t_0$ in $\tilde\tau$, whose endpoints are distinct from $t$, such that $t\cap t_0$ contains this large branch.
 There is a $\Z$ subgroup of the deck group that stabilizes $t$ and we let $\dots, t_{-1}, t_0, t_1, \dots$ be the translates of $t_0$ under this group.
As $i\to\infty$ the endpoints of $t_i$ limit to the terminal endpoint of $t$ while as $i\to-\infty$ they limit to the initial endpoint. In particular, there is $t_i \subset \tilde{A}$ with $t \cap t_i$ occurring before $\tilde b$ and a $t_j\subset \tilde{B}$ with $t\cap t_j$ occurring after $\tilde b$. Then connectivity implies that the compact set between these two intersections contains all branches of $t$ that  cross $\tilde\alpha$.

This shows that each $t \in \cP'$ contains at most finitely many branches that cross
$\tilde\alpha$.
The set $\cP'$ is $\Gamma_\alpha$-invariant and the quotient $\cP'/\Gamma_\alpha$ has $k$ elements where $k$ is the geometric intersection number of $\alpha$ and $\beta$. Therefore there are finitely many $\Gamma_\alpha$-orbits of crossing branches.
\end{proof}

Each complementary component $\tilde P$ of $\tilde\tau$ is either a finite-sided polygon, in which case the interior of $\tilde P$ descends homeomorphically to a polygon $P$ in $\Sigma \smallsetminus\tau$, or it is an infinite-sided polygon, stabilized by a parabolic subgroup. In this case, the interior of $\tilde P$ covers a punctured polygon in $\Sigma\smallsetminus\tau$.

Each side of $\tilde P$ is a union of branches in $\tilde\tau$.
These branches on a side of $\tilde P$ may be branches in $\tilde{A}$, branches in $\tilde{B}$, branches in $\tilde{A}\cap\tilde{B}$ (the carried branches) and branches in neither $\tilde{A}$ nor $\tilde{B}$ (the crossed branches).
The following lemma characterizes polygons of $\tilde \Sigma \setminus \tilde \tau$.

\begin{lemma}\label{polygons}
If $\tilde P$ is a polygon of $\tilde\Sigma\smallsetminus\tilde\tau$ then the closure, in $\tilde\Sigma \cup S^1_\infty$, of $\tilde{A} \cap \tilde P$ and  $\tilde{B}\cap \tilde P$  is connected.
A complementary polygon $\tilde P$ satisfies one of the following four mutually
exclusive conditions:
\begin{enumerate}
\item $\partial \tilde P$ contains a branch that carries $\tilde\alpha$. Then there is a single train path in $\partial\tilde P$ that carries $\tilde\alpha$ and the remaining branches are either entirely in $\tilde{A}$ or entirely in $\tilde{B}$.

\item $\partial \tilde P$ has no branches that carry $\tilde\alpha$ and either all
branches of $\tilde P$ lie in $\tilde{A}$ or all branches of $\tilde P$ lie in $\tilde{B}$.
 
\item $\partial \tilde P$ contains a branch that crosses $\tilde\alpha$. Then there are one or two train paths of $\partial \tilde P$ (maximal in $\partial\tilde P$) in three possible configurations.
\begin{enumerate}
\item There is one crossing train path $t$ in $\partial\tilde P$ and it is a maximal crossing train path in $\tilde\tau$. One endpoint of the train path is in $\tilde{A}$ and the other is $\tilde{B}$.

\item There are two crossing train paths in $\partial\tilde P$ that are maximal crossing train paths in $\tilde\tau$, each with one endpoint in $\tilde{A}$ and the other in $\tilde{B}$.

\item There are two crossing train paths in $\partial\tilde P$ that meet at a vertex of $\tilde P$. Then the vertex of $\tilde P$ where these two paths meet is not in $\tilde{A} \cup \tilde{B}$ while the other two endpoints are both in $\tilde{A}$ or both in $\tilde{B}$.
\end{enumerate}

\item No branches of $\partial\tilde P$ carry $\tilde\alpha$ and no branches of $\partial \tilde P$ cross $\tilde\alpha$. Each side of $\partial\tilde P$ is either entirely in $\tilde{A}$ or entirely in $\tilde{B}$ and there are at least two sides of each type.

\end{enumerate}
\end{lemma}

We will say that a polygon has type $(1)$ if it satisfies condition $(1)$ above, and
similarly for the other list items.

\begin{proof}[Proof of Lemma \ref{polygons}]
If $t$ is a bi-infinite train path that contains a branch on a side of $\tilde P$, we let $I_t \subset S^1_\infty$ be the interval such that the disk bounded by $t \cup I_t$ does not contain $\tilde P$. Observe that $t$ intersects the interior of exactly one side $s$ of $\tilde P$.
 We let $\cI_s$ be the set of all such intervals $I_t$ over all train paths $t$ meeting $s$.
 For each side $s$ there is an oriented train path $t_s$ that contains $s$ with $P$ on the right and such that as $t_s$ leaves $s$ going forward it always turns to the left and as it leaves $s$ going backwards it always turns to the right.  Set $I_s = I_{t_s}$ and then for any $I \in \cI_s$ we have $I \subset I_s = I_{t_s}$.
 The union of the intervals $I_s$ is $S^1_\infty$ if $\tilde P$ has finitely many sides.
 If $\tilde P$ has infinitely many sides then it contains all of $S^1_\infty$ except the fixed point of the stabilizer of $\tilde P$.
 In particular, if $\alpha_-$ and $\alpha_+$ are the endpoints of $\tilde \alpha$ then they are each contained in some interval $I_s$. We also note that if $I_s \cap I_{s'} \neq \emptyset$ then $s$ and $s'$ are adjacent sides (for otherwise $\tilde \tau$ would contain a complementary monogon).
 
 Now we enumerate the possibilities for $\tilde P$.

\begin{itemize}
\item If $\alpha_-, \alpha_+ \in I_s$ for some $s$ and there is $I \in \cI_s$ such that $\{\alpha^-,\alpha^+\} \cap I = \emptyset$, then $\tilde P$ has type (1).

\item If $\alpha_-, \alpha_+ \in I_s$ for some $s$ and there is no $I \in \cI_s$ such that $\{\alpha^-,\alpha^+\} \cap I = \emptyset$ then $\tilde P$ has type (2).

\end{itemize}
For the remaining cases we have that $\alpha_- \in I_{s_-}$ and $\alpha_+ \in I_{s_+}$ with $s_- \neq s_+$.
\begin{itemize}
\item There is $s'_- \neq s_-$ with $\alpha_- \in s'_-$ but $\alpha_+$ is contained only in $I_{s_+}$. This is type (3a). Similarly $\tilde P$ is type (3a) if the the plus and minus are swapped.

\item If $\alpha_-$ and $\alpha_+$ are only contained in $I_{s_-}$ and $I_{s_+}$ and $s_-$ and $s_+$ are disjoint then $\tilde P$ is type (3b). If they are adjacent then $\tilde P$ is type (3c).

\item If there are $s'_-\neq s_-$ and $s'_+ \neq s_+$ with $\alpha_- \in I_{s'_-}$ and $\alpha_+ \in I_{s'_+}$ then $\tilde P$ has type (4).
\end{itemize} 
See figures \ref{fig:threea} and \ref{fig:poly_types}. \end{proof}

\begin{figure}[h]

  \centering

  \def\svgwidth{0.6\textwidth}
  \includegraphics[scale=.5]{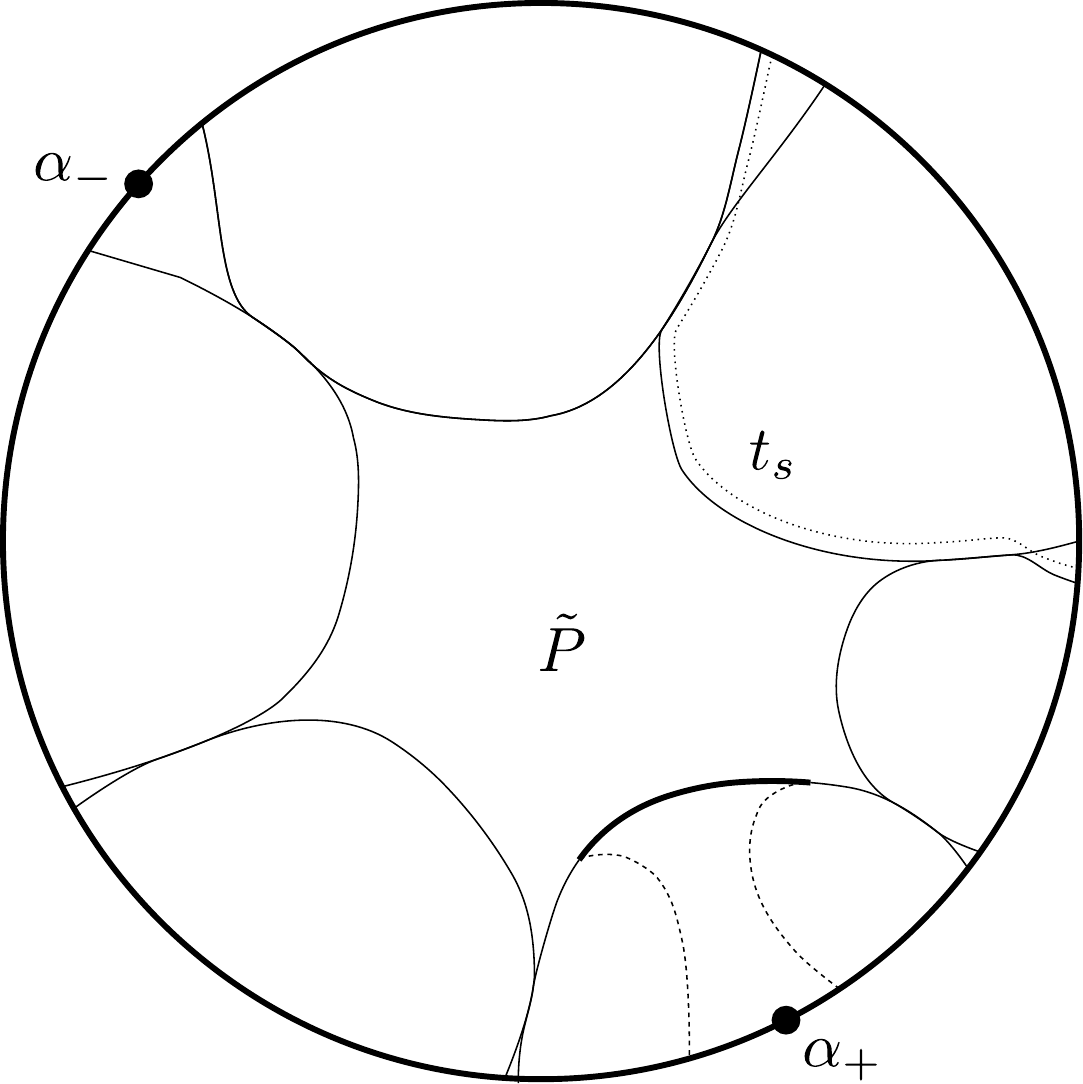}

  \caption{This example is a polygon of type (3a) along with the train paths $t_s$ for each side $s$. The crossing branches are shaded.}
  \label{fig:threea}
\end{figure}

\begin{figure}[h]

  \centering

  \def\svgwidth{0.6\textwidth}
  \includegraphics[scale=.5]{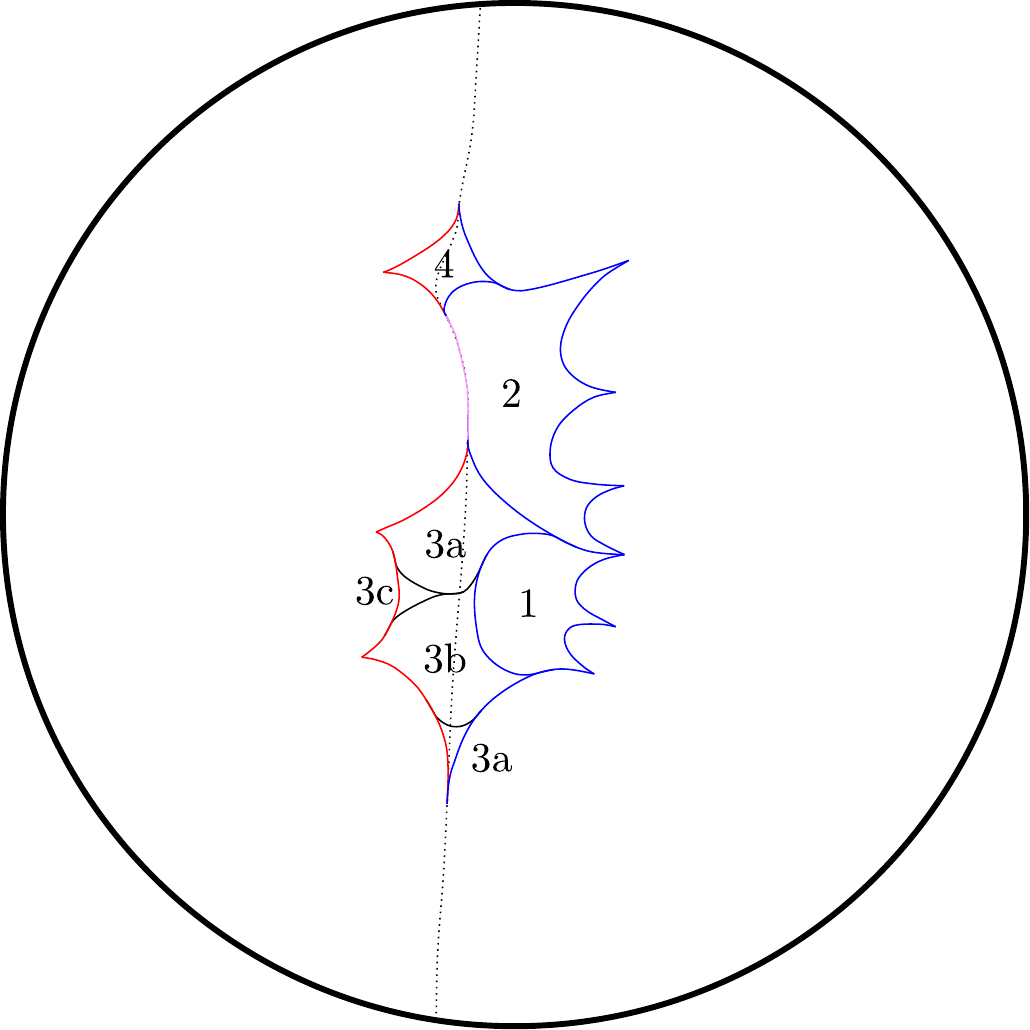}

  \caption{In the figure a portion of the train track $\tilde\tau$ is shown illustrating a polygon of each type. The red arcs are branches in $\tilde{A}$, the blues arcs are branches in $\tilde{B}$, the purple arc is a carried branch and the black arcs are crossed branches. The dashed line is the lift $\tilde\alpha$.}
  \label{fig:poly_types}

\end{figure}

\begin{lemma}\label{transverse extension}
Suppose $\tilde \tau$ has a complementary polygon $\tilde P$ of type (4).
Then there is a diagonal extension $\tau'$ of
$\tau$, consisting of $\tau$ plus a single diagonal branch, such that the unique branch
of $\tau' \smallsetminus \tau$ carries $\alpha$.
\end{lemma}

\begin{proof} 
 A type (4) polygon has distinct vertices $\tilde v_0$ and $\tilde v_0$ both in $\tilde{A}\cap \tilde{B}$. The image of $\tilde P$ in under the covering map $\tilde\Sigma\to\Sigma$ will be a polygon $P$ in $\Sigma\smallsetminus\tau$ and the image of $\tilde v_0$ and $\tilde v_1$ will be vertices $v_0$ and $v_1$ of $P$. If $\tilde P$ is finite sided then the restriction of the covering map to the interior of $\tilde P$ is a homeomorphism to the interior of $P$. In particular,  $v_0\neq v_1$ and they don't bound a side of $P$  so we can add a branch $b$ in $P$ between them to form $\tau'$. The pre-image of $b$ in $\tilde\Sigma$ contains a branch $\tilde b$ from $\tilde v_0$ to $\tilde v_1$ and $\tilde b$ is carried by $\tilde\alpha$. Therefore $b$ is carried by $\alpha$ as claimed.
 
If $\tilde P$ has infinitely many sides then $P$ is a punctured polygon and $\partial\tilde P$ is a topological line. The vertices $\tilde v_0$ and $\tilde v_1$ bound a segment $s$ which is the union of, say, the $A$-sides of $\partial \tilde P$. If $\gamma \in \pi_1(\Sigma)$ is a non-trivial element of the deck group that stabilizes $\tilde P$ and $s$ and $\gamma(s)$ intersect in their interiors then $\tilde\alpha$ and $\gamma(\tilde\alpha)$ intersect, contradicting that $\alpha$ is simple. While we may have that $v_0= v_1$, this implies that the distinct sides of $\tilde P$ in $s$ map to distinct sides of $P$. We can then  choose a diagonal branch $b$ to $\tau$ such that $\tilde b$ lifts to an arc $\tilde b$ with endpoints $\tilde v_0$ and $\tilde v_1$ and such that $\tilde b$ is homotopic to $s$ rel endpoints. See Figure \ref{fig:diagext}. Then $\tilde b$ is carried by $\tilde\alpha$ and $\tau' = \tau \cup b$. 
\end{proof}

\begin{figure}[h]

  \centering

  \def\svgwidth{0.6\textwidth}
  \includegraphics[scale=.7]{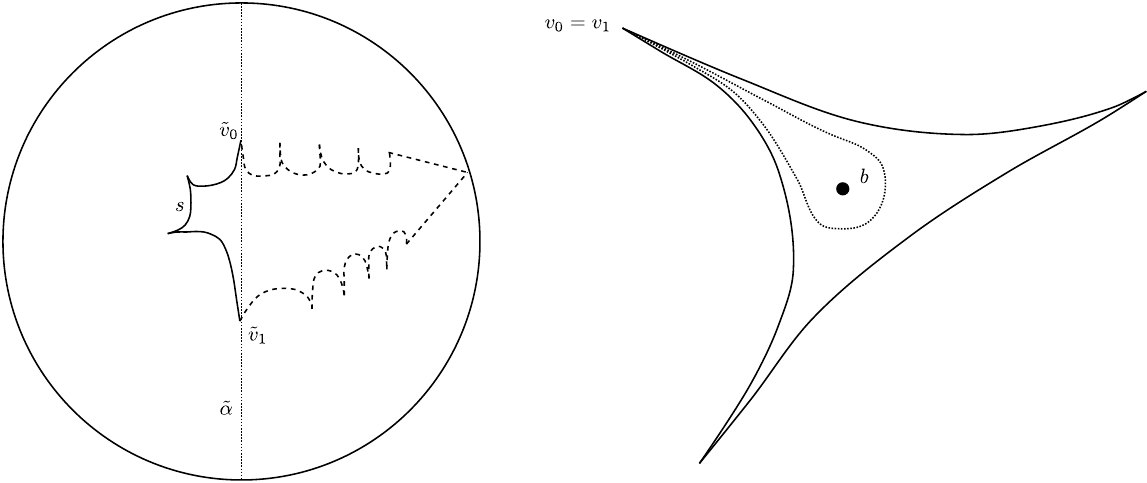}

  \caption{In the figure, the segment $s$ is a union of three $\tilde{A}$-sides of the infinite-sided polygon $\tilde P$. In $P$ we have that $v_0 = v_1$ but the image of $s$ is still three distinct sides.}
  \label{fig:diagext}
\end{figure}

\begin{prop}\label{diagonal carry}
If there are no branches of $\tau$ that cross $\alpha$ then $\alpha$ is carried by a diagonal extension of $\tau$.
\end{prop}

\begin{proof}
If the set of branches of $\tilde\tau$ that are carried by $\tilde\alpha$ is a bi-infinite path then we are done by Lemma \ref{traverse set}. If not, we will show that $\tilde\tau$ contains a polygon of type (4). By Lemma \ref{transverse extension} there is a diagonal extension $\tau'$ so that the single branch in $\tau' \smallsetminus \tau$ is carried by $\alpha$. If $\tau'$ carries $\alpha$ then the proof is complete. If not, then we repeat this process. The process will eventually terminate since a maximal train track cannot have a type (4) polygon. As any diagonal extension of a diagonal extension is a diagonal extension, this will complete the proof.

Now we need to find a type (4) polygon. If there are no branches of $\tilde \tau$ carried by $\tilde \alpha$ (i.e. $\tilde{A} \cap \tilde{B} = \emptyset$) then there are branches that cross $\tilde\alpha$, since $\tilde\tau$ is connected. This is a contradiction, so there
are in fact branches of $\tilde \tau$ that carry $\tilde \alpha$.
Since  the set of branches carried by $\tilde \alpha$ is not a bi-infinite path, there is a large half-branch whose endpoint $v$ is the endpoint of a carried train path. Note that one of the two small half-branches adjacent to $v$ is in $\tilde{A}$ and not in $\tilde{B}$ while the other is in $\tilde{B}$ but not $\tilde{A}$.  If $\tilde P$ is the polygon that has $v$ as a vertex then this implies that it is not type (1) or (2). Since there no crossing branches, $\tilde P$ is not type (3) and is therefore of type (4).
\end{proof}

We also observe that the Masur-Minsky nesting lemma follows easily from this proposition.
\begin{proof}[Proof of Lemma \ref{MMdiag}]
Let $\tilde \beta$ be a lift of $\beta$ to $\tilde\Sigma$. Since $\alpha$ is fully carried, it is represented by a train path that traverses every branch of $\tau$.
Consider a branch $\tilde b$ of $\tilde \tau$. We claim it does not cross $\tilde \beta$.
There is a lift $\tilde\alpha$ which passes through $\tilde\tau$.
Since $d(\alpha, \beta) = 1$, the endpoints of $\tilde\alpha$ do not separate the endpoints of $\tilde\beta$. Therefore $\tilde b$ is not crossed by $\tilde\beta$, as claimed.
The proposition then follows from Proposition \ref{diagonal carry}.
\end{proof}

The following lemma is very close to a version of the Masur-Minsky nesting lemma but for the spaces $\C^H_k$ instead of the curve graph. However, there are extra technical assumptions. Instead of just assuming that $\alpha$ is not carried by a diagonal extension of $\tau$ we need $\alpha$ to hit $\tau$ efficiently and to intersect every large branch. These technical assumptions account for much of the extra work needed in the proof of Theorem \ref{progress2}.

\begin{lemma}\label{lower bound}
Let $\alpha$ and $\beta$ be simple closed curves and let $\tau$ be a recurrent train track such that $\alpha \pitchfork \tau$, $\alpha$ intersects every large branch of $\tau$ and $\beta$ is fully carried by a diagonal extension of $\tau$.
Then $\ind(\alpha, \beta) \le \ind(\tau)$. In particular, if $k = \ind(\tau)$ then $d^H_k(\alpha, \beta) \ge 2$.
\end{lemma}

\begin{proof}
If $\alpha$ hits $\tau$ efficiently then the same is true for any diagonal extension of $\tau$, so it suffices to assume that $\beta$ is fully carried by $\tau$.  Let $\cF = \cF(\alpha; \beta, \tau)$. By summing the expression in Lemma \ref{train track side count}
over each component of $\Sigma\smallsetminus \tau$ and applying Proposition
\ref{numerology}, we have
$$\ind(\tau) \ge 2p + \sum(i_j - 3)$$
if $\tau$ is non-orientable and
$$\ind(\tau) \ge -1 +2p + \sum(i_j - 3)$$ 
if $\tau$ is orientable. Here the inequality is strict if and only if there is a component of $\Sigma\smallsetminus \tau$ that contains more than one singularity
of $\cF$.

Now fix $\alpha$ and $\beta$ to be representatives of their isotopy classes such that $\alpha$ is transverse to $\cF$ and $\beta$ is a leaf of $\cF$. We can also assume that $\beta$ is homotopic into a train path in $\tau$ in the complement of the singularities of $\cF$ and there is at most one singularity in any complementary component of $\tau\cup\alpha$. After arranging $\alpha$ and $\beta$ in this way we claim that
each large component of $\Sigma\smallsetminus(\alpha\cup\beta)$ contains exactly one singularity of $\cF$.

Assuming this claim (which we prove below) we complete the proof.
We now note that, by Lemma \ref{induced orientation}, if $\alpha$ and $\beta$ are an oriented pair then $\cF$, and hence $\tau$, is oriented, so the second inequality above holds. If $\alpha$ and $\beta$ are not an oriented pair, we have the first inequality. In both cases, Corollary \ref{one singularity} implies that the right hand side of the inequality is $\ind(\alpha, \beta)$.
\end{proof}

\begin{proof} [Proof of Claim]
Let $\eta$ be an embedded arc between two distinct singularities $p$ and $q$. We assume that $\eta$ is disjoint from $\beta$ and then show it must intersect $\alpha$.

By construction, any punctures of $\Sigma$ are singularities of $\cF$ so if one or both of $p$ and $q$ are singularities then $\eta$ is asymptotic to this puncture. If not, puncture $\Sigma$ at the endpoint to obtain a new surface $\Sigma_0$ such that, in $\Sigma_0$, $\eta$ is asymptotic to punctures at both ends. As our choice of $\beta$ is homotopic to a train path in $\tau$ in the complement of the singularities of $\cF$, we have that $\beta$ is carried (and hence fully carried) by $\tau$ as a curve in $\Sigma_0$.

 Let $\tilde\Sigma_0$ be the universal cover of $\Sigma_0$ and $\tilde\tau_0$ the lift of $\tau$.  The arc $\eta$ lifts to a proper line $\tilde\eta$ that has two endpoints in $S^1_\infty$ and we can separate $S^1_\infty$ into two components $A$ and $B$ as before. We define branches of $\tilde\tau_0$ to be carried by $\tilde \eta$ or to cross $\tilde\eta$ analogously to the definitions given for lifts of simple closed curves.
 
 With the additional punctures it may be that a component of $\Sigma_0\smallsetminus \tau$ may have more than one puncture, in which case $\tilde\tau_0$ will be disconnected. However, since each complementary component $\tau \cup \alpha$ contains at most one singularity either $\eta$ intersects $\alpha$, and we are done, or $\eta$ intersects $\tau$.

In this latter case we proceed as in the proof of Proposition \ref{MMdiag}. Since $\eta$ intersects $\tau$ we have that $\tilde\eta$ intersects $\tilde\tau_0$ so there must either be a branch of $\tilde\tau_0$ that carries $\tilde\eta$ or a branch that  crosses $\tilde\eta$. However, since $\eta$ is disjoint from the fully carried curve $\beta$, we have that no branches of $\tilde\tau_0$ cross $\tilde\eta$. Therefore there is a branch of $\tilde\tau_0$ that carries $\tilde\eta$, and hence there is a large branch $\tilde b$ that carries $\tilde\eta$.

Let $\tilde\alpha$ be a lift of $\alpha$ that intersects $\tilde b$. Since $\alpha \pitchfork \tau$, $\tilde\alpha$ intersects some train path in $\tilde\tau_0$ at most once. As the branch $\tilde b$ is carried by $\tilde\eta$, there are bi-infinite train paths $t$ and $t'$, both containing $\tilde b$, with the endpoints of $t$ in $A$ and the endpoints of $t'$ in $B$. Together this implies that one endpoint of $\tilde\alpha$ is in $A$ and the other is in $B$. In particular, $\tilde\alpha$ separates the endpoints of $\tilde\eta$ so $\alpha$ must intersect $\eta$.
\end{proof}

Next we give a criterion for a curve $\alpha$ to be homotopic to a curve that hits $\tau$ efficiently and that intersects every large branch. At this stage it is convenient to work in the annular cover $\Sigma_\alpha\to \Sigma$ where $\Sigma_\alpha$ contains a homeomorphic lift $\hat\alpha$ of $\alpha$. We let $\hat\tau$ be the pre-image of $\tau$ in $\Sigma_\alpha$. The ideal boundary of $\Sigma_\alpha$ is two circles which we label $\partial_A$ and $\partial_B$ where $\partial_A$ ($\partial_B$) is the image of $A$ ($B$) under the (extension of the) covering map $\tilde\Sigma\to\Sigma_\alpha$. We then let $\hat A$ be the set of branches $b$ of $\hat\tau$ where there is a bi-infinite train path $t$, containing $b$, with both endpoints of $t$ in $\partial_A$. We similarly define $\hat B$. Equivalently, $\hat A$ and $\hat B$ are the images of $\tilde A$ and $\tilde B$ under the covering map $\tilde\Sigma\to \Sigma_\alpha$. A branch of $\hat\tau$ is {\em crosses} $\hat\alpha$ is it is in $\hat\tau\smallsetminus \hat A \cup \hat B$ and it {\em carries} $\hat\alpha$ if it is contained in $\hat A \cap \hat B$. Again, the crossing and carried branches in $\hat\tau$ are the images of the crossing and carried branches in $\tilde\tau$.

\begin{lemma}\label{transverse function}
Let $\alpha$ be a simple closed curve and assume that no branches of $\tau$ carry $\alpha$.
Then there is a continuous function $\phi\colon \hat\tau \to [0,1]$ such that $\phi(\hat{A}) =\{0\}$, $\phi(\hat{B}) = \{1\}$, $\phi$ is monotonic on any crossing train paths and $\phi^{-1}(1/2)$ is discrete.

Furthermore if $\tau \to \tau_0$ and $b_0$ is a branch of $\tau_0$ that is crossed by $\alpha$ then $\phi$ can be chosen such that for every branch $b$ of $\tau$ whose in $\tau_0$ contains $b_0$ we have that $1/2 \in \phi(b)$.
\end{lemma}

\begin{proof}
We first set $\phi$ to be zero on $\hat{A}$ and one on $\hat{B}$.
Let $N$ be the maximal number of vertices in any crossing train path and let $v \not\in \hat{A}\cup\hat{B}$ be a switch in the union of crossing branches. Then there is a crossing path that starts in $\hat A$ and ends at $v$. Let $\sigma(v)$ be the maximal number of switches in any such path. If $t$ is a maximal train path with initial train path in $\hat A$ and terminal train path $\hat B$ and $v$ and $v'$ are switches in $t$ with $v$ occurring before $v'$ then $\sigma(v) < \sigma(v')$. 
 Next we choose real numbers $x_i$ such that
$$0< x_1< \cdots < x_N < 1.$$
We the define $\phi(v) = x_{\sigma(v)}$ and extend $\phi$ linearly along each crossing branch. This gives the first claim.

Now assume that $\tau\to \tau_0$ and $\phi_0 \colon \hat\tau_0 \to [0,1]$ has been chosen as above. We can also choose a branch $\hat b_0$ in $\hat\tau_0$, that is crossed by $\hat\alpha$, such that $\hat b_0$ maps to $b_0$ under the covering map $\hat\tau_0 \to \tau_0$ and we can also assume that $1/2 \in \phi_0(\hat b_0)$. The carrying map $\tau\to \tau_0$ lifts to a carrying map $\hat \tau \to \hat \tau_0$ that extends to the identity on the ideal boundary of $\Sigma_\alpha$ and we define $\phi$ to be the composition of $\phi_0$ with the carrying map.
\end{proof}

Observe that the deck action of $\Gamma_\alpha$ on $\tilde\Sigma$ will not fix any polygon $\tilde P$ in $\tilde\Sigma\smallsetminus\tilde\tau$ and it will preserve polygon types.

\begin{prop}\label{transverse curve}
Let $\tau$, $\tau_0$ be train tracks and $\alpha$ a curve such that
\begin{itemize}
\item $\tau\to\tau_0$;
\item no branch of $\tau_0$ is carried by $\alpha$;
\item the branch $b_0$ of $\tau_0$ is crossed by $\alpha$.
\end{itemize}
Then a representative of $\alpha$ can be chosen such $\alpha\pitchfork \tau$ and $\alpha$ intersects every branch of $\tau$ that contains the pre-image of $b_0$ under the carrying map $\tau\to\tau_0$.
\end{prop}

\begin{proof} 
As in Lemma \ref{transverse function} we will work in the cover $\Sigma_\alpha\to \Sigma$. We let $\phi\colon\hat\tau\to [0,1]$ be the function given by Lemma \ref{transverse function}. In particular, for every branch $b$ of $\tau$ whose image in $\tau_0$ contains $b_0$ we have a branch $\hat b$ of $\hat\tau$ such that $\hat b$ maps to $b$ under the covering map $\hat\tau\to\tau$ and $1/2 \in \phi(\hat b)$.

Assume that the polygon $\hat P$ in $\Sigma_\alpha\smallsetminus \hat\tau$ contains a branch that is crossed $\hat\alpha$. Therefore $\hat P$ has type (3). Note that it can't have type (3a) for if it did then the intersection $\hat{A} \cap \hat{B}$ would be non-empty. Therefore it is of type (3b) or (3c) and $\hat P$ contains exactly two train paths $t$ and $t'$ with either both maximal crossing paths or $t$ and $t'$ meet at a switch while their other endpoints are both in $\hat{A}$ or both in $\hat{B}$. In both case the $\phi$-image of $t$ and $t'$ are equal. In particular if there is a $p \in t$ with $\phi(p) = 1/2$ there is a $p' \in t'$ with $\phi(p') =1/2$ and we can fix an embedded arc in $\hat P$ with endpoints $p$ and $p'$. The union of these arcs is a collection of properly embedded lines and simple closed curves that we label $\hat \beta$. We will show that $\hat\beta$ is a single, essential simple closed curve in $\Sigma_\alpha$.

First we note that by Lemma \ref{connected} the number of crossing branches in $\hat\tau$ is finite. Therefore $\hat\beta$ is a union of finitely many arcs and cannot contain a line. Next we observe that any train path $t$ contains a most one point $p \in t$ with $\phi(p) = 1/2$ so $\hat\beta$ intersects $t$ at most once. However, if $\hat\beta$ contains a non-essential closed curve it will intersect every bi-infinite train path an even number of times, a contradiction. Therefore $\hat\beta$ is a union of essential curves. Every essential curve will intersect any bi-infinite train path with endpoints in distinct components of the ideal boundary of $\Sigma_\alpha$. Therefore $\hat\beta$ must be a single essential simple closed curve.
However, the image in $\Sigma$ may not be a simple curve. We resolve this last issue in  Lemma \ref{homotope to simple} that follows.
\end{proof}

\begin{lemma}\label{homotope to simple}
Let $\alpha$ be an immersed curve that hits $\tau$ efficiently. If $\alpha$ is homotopic to a simple closed curve then it is homotopic to a simple closed curve $\alpha'$ with $\alpha'\pitchfork \tau$ and $\alpha \cap \tau \subset \alpha' \cap \tau$.
\end{lemma}

\begin{proof}
After performing a small homotopy we can assume that all of the self intersections of $\alpha$ are disjoint from $\tau$ and that the new curve has larger intersection with $\tau$. 
Let $x$ and $x'$ be arcs in $\alpha$ that form an innermost bigon. We swap $x$ and $x'$ and perform a small homotopy to eliminate the two points of intersection. This gives a new curve $\alpha'$ homotopic to $\alpha$  with $\alpha \cap \tau = \alpha' \cap \tau$ but the self intersections have been reduced by two. 
 Let $\tilde\alpha$ be a lift of $\alpha$ and label the pre-images of $x$ in $\tilde\alpha$ as $\dots, x_{-1}, x_0, x_1, \dots$. Each $x_i$ forms a bigon with an arc $x'_i$ that maps to $x'$ under the covering map. We obtain a lift of $\tilde\alpha'$ of $\alpha'$ by swapping the $x_i$ with the $x'_i$ (and equivariantly performing the small homotopy). 

Any train path can intersect each $x_i$ and $x'_i$ at most once so if a train path intersects $x'_i$ it will also intersect $x_i$. Let $t$ be a train path that intersects $\tilde\alpha'$. If $t$ doesn't intersect any the $x'_i$ then $t\cap \tilde\alpha' \subset \tilde\alpha'\cap \tilde\alpha$ so there is only one point of intersection. If $t$ intersects $x'_i$ then $t$ also intersects $x_i$. This also can happen only once since otherwise $t$ would intersect $\tilde\alpha$ two or more times. Therefore $t\cap\tilde\alpha'$ is a single point.
 Any train path that intersects $x_i$ will also intersect $x'_i$. Since any train path $t$ can intersect $\tilde\alpha$ at most once this implies that $t$ can intersect $\tilde\alpha'$ at most once and therefore $\alpha' \pitchfork \tau$. 
 We repeat this process until we have a simple curve.
\end{proof}

\begin{lemma}\label{commute}
Assume that $\tau'\ss\tau$ and that $b$ is a large branch in $\tau$ that is split in $\tau'$. Then the splitting sequence can be arranged such that the first split in the sequence occurs on $b$.
\end{lemma}

\begin{proof}
Assume that $\tau' =\tau_0 \to \tau_1 \to \tau_2 = \tau$ is a sequence of single splits. If $\tau_1\to\tau_2$ is a split along $b$ we are done. If not there is another large branch $b'$ in $\tau_2$ where $\tau_1 \to \tau_2$ is a split along $b'$. Then the pre-image of $b$ in $\tau_1$ is a large branch and the split $\tau_0\to\tau_1$ occurs on this branch. One can then check directly that if we first split on $b$ and then on $b'$ the final track will still be $\tau_0$. The general case follows by induction.
\end{proof}

\begin{lemma}\label{traverse}
There exists $N>0$ such that if
$$\tau_0\twoheadrightarrow \cdots \twoheadrightarrow\tau_N$$
and $b$ is branch in $\tau_N$
then there is a track $\tau$ with $\tau_0 \ss \tau\ss \tau_N$ and every large branch of $\tau$ intersects the pre-image of $b$ under the carrying map $\tau\to \tau_N$.
\end{lemma}

\begin{proof}
Let
$$\sigma_0 \ss\sigma_1 \ss\cdots \ss\sigma_n = \tau_{N-1}$$
be a sequence of single splittings and let $\sigma^b_i$ be the branches of $\sigma_i$ that don't intersect the pre-image of $b$ under the carrying map $\sigma_i \to \tau_N$.  Since $\tau_{N-1} \twoheadrightarrow \tau_N$ we have that $\sigma^b_n$ doesn't carry any closed curves.

Assume that each split $\sigma_i \ss\sigma_{i+1}$ is along a large branch in $\sigma^b_{i+1}$ and that $\sigma_0^b$ contains a large branch.  We will obtain a bound on $n$. A {\em $b$-train path} in $\sigma_i$ is a a train path whose initial and terminal half branches are in $\sigma_i \smallsetminus \sigma^b_i$ while the remainder of the train path is in $\sigma^b_i$. Let $t_i$ be the number of $b$-train paths in $\sigma_i$
 Every large branch in $\sigma^b_i$ is traversed by at least four distinct $b$-train paths so $t_i \ge 4$ for $i = 0, \dots, n$. Furthermore, the image in $\sigma_{i+1}$, of a $b$-train path in $\sigma_i$ is also a $b$-train path and there is at least one $b$-train path in $\sigma_{i+1}$ that is not the image of a $b$-train path in $\sigma_i$ so $t_i \le t_{i+1} -1$.
As $\sigma^b_n$ doesn't carry any closed curves, none of these train paths can contain a closed curve and, therefore, none of these paths train paths traverse any branch more than twice. As the number of branches of any train track on $\Sigma$ is uniformly bounded  there is a constant $N' = N'(\Sigma)$ such that $t_i \le N'+5$. Since $t_i < t_{i+1}$ and $t_i \ge 4$ this implies that $n < N'$.

Let $N = N'+1$ and assume that the lemma fails: for any train track $\tau$ with $\tau_0 \ss \tau \ss \tau_N$ there is a large branch of $\tau$ that is disjoint from the pre-image of $b$. 
We will show, via induction, that for all $m$ with $0\le m \le N'$ there is a sequence of single splits
$$\sigma_{N'-m} \ss \cdots \ss \sigma_{N'} =\tau_{N-1}$$
where all the splits occur on large branches in $\sigma^b_i$
and $\tau_{N'-(m+1)} \ss \sigma_{N'-m}$.
If $m=0$ this is clear. For the induction step we assume such a sequence exists for $m$ and prove there is a sequence for $m+1$.
 As $\tau_{N'-(m+1)} \twoheadrightarrow \tau_{N'-m}$ we have $\tau_{N'-(m+1)} \twoheadrightarrow \sigma_{N'+1-m}$.
  By our assumption $\sigma^b_{N'-m}$ contains a large branch and by Lemma \ref{allsplit} this branch is split in any splitting sequence $\tau_{N'-(m+1)} \twoheadrightarrow \sigma_{N'-m}$.
  By Lemma \ref{commute} we can arrange this splitting sequence so that it happens first.  This produces the new sequence for $m+1$. As $\tau_{N'-(m+2)} \twoheadrightarrow \tau_{N'-(m+1)}$ we have $\tau_{N'-(m+2)} \twoheadrightarrow \sigma_{N'-(m+1)}$.
  This completes the induction step.

When $m =N'$ we have sequence  
$$\sigma_0 \ss\sigma_1 \ss\cdots \ss\sigma_{N'} = \tau_{N-1}$$
where all the splits occur in $\sigma^b_i$ and $\sigma_0^b$ contains a large branch. This contradicts the choice of $N'$ finishing the proof.
\end{proof}

\begin{proof}[Proof of Theorem \ref{progress2}]
First assume $n=1$.
Let $p = N+M$ where $N$ is the constant from Lemma \ref{traverse} and $M$ is the constant from Lemma \ref{extension carry}. Recall that $\alpha$ isn't carried by any diagonal extension of $\tau_p$ so by Proposition \ref{diagonal carry} there is branch $b$ of $\tau_p$ that $\alpha$ crosses.
  By Lemma \ref{traverse} there is a track $\tau$ with $\tau_M \to \tau\to \tau_p$ such that every large branch of $\tau$ is in the pre-image of $b$ under the carrying map $\tau\to\tau_p$. Then Proposition \ref{transverse curve} implies that  there is a representative of $\alpha$ with
$\alpha \pitchfork \tau$ and such that $\alpha$ crosses every large branch of $\tau$.
By Lemma \ref{extension carry}, since $\beta$ is carried by a diagonal extension of $\tau_0$ it is fully carried by a diagonal extension of $\tau_M$  and hence is fully carried by a diagonal extension of $\tau$. By Lemma \ref{lower bound}, $d^H_k(\alpha, \beta) \ge 2$.

For the general case we let $\beta=\gamma_0, \gamma_1, \dots$ be an arbitrary path in $\C_k(\Sigma)$. By the previous case we inductively see that for $i\le n$ the curve $\gamma_i$ is carried by a diagonal extension $\tau_{i\cdot p}$. In particular, if $\alpha =\gamma_j$ then $j>n$ since $\alpha$ is not carried by a diagonal extension of $\tau_{n\cdot p}$ and, hence, not carried by a diagonal extension of any $\tau_{i\cdot p}$.  Therefore $d^H_k(\alpha, \beta)>n$.
\end{proof}

\begin{lemma}\label{no extension carry}
Let $\tau_0 \ss \tau_1$. If $\alpha$ is carried by $\tau_1$ but not by $\tau_0$ then $\alpha$ is not carried by any diagonal extension of $\tau_0$. In particular, if $\tau_0 \twoheadrightarrow \tau_1$ then any vertex cycle of $\tau_1$ is not carried by any diagonal extension of $\tau_0$.
\end{lemma}

\begin{proof}
By assumption there is a single split $\eta_0\ss \eta_1$ in the sequence $\tau_0 \ss\tau_1$ where $\eta_1$ carries $\alpha$ but $\eta_0$ does not.

Lift the carrying map $\eta_0 \to \eta_1$ to a carrying map $\tilde\eta_0 \to \tilde\eta_1$ in the universal cover that is the identity at infinity. Let $b_0$ be the small branch in $\eta_0$ that is created in the splitting and let $\tilde b_0$ be a lift of $b_0$ to $\tilde\eta_0$. We'll show that $b_0$ crosses $\alpha$. Let $t_0$ be a bi-infinite train path in $\tilde\eta_0$ that contains $\tilde b_0$ and let $t_1$ be the image of $t_0$ in $\tilde\eta_1$. The image of $\tilde b_0$ in $\tilde\eta_1$ is a large branch $\tilde b_1$ and $t_1$ traverses it - that is, $t_1$ enters and leaves on opposite sides. Since $\alpha$ is carried by $\eta_1$ but not $\eta_0$, there is a lift $\tilde\alpha$ that contains $\tilde b_1$ but enters and leaves $\tilde b_1$ at the small half branches that $t_1$ misses. In particular, $t_1$ separates the endpoints of $\tilde\alpha$. Since the carrying map is the identity at infinity the train path $t_0$ also separates the endpoints of $\tilde\alpha$. Hence, $\tilde b_0$ crosses $\tilde\alpha$ and $b_0$ crosses $\alpha$.

Since $\eta_0$ has a branch that crosses $\alpha$, by Proposition \ref{diagonal carry} no diagonal extension of $\eta_0$ carries $\alpha$. If a diagonal extension of $\tau_0$ carried $\alpha$ then, by Lemma \ref{diagonal extension}, a diagonal extension of $\eta_0$ would carry $\alpha$. Hence, no diagonal extension of $\tau_0$ carries $\alpha$.
\end{proof}

\begin{cor}\label{distance_lower}
There exist constants $A_0$ and $A_1$ such that following holds.
Let
$$\tau_0 \twoheadrightarrow \cdots \twoheadrightarrow \tau_n$$
be a sequence of index $k$ train tracks. If $\alpha$ is carried by a diagonal extension of $\tau_0$ and $\beta$ is not carried by a diagonal extension of $\tau_n$ then
$$d_k(\alpha, \beta) \ge A_0 n - A_1.$$
Furthermore
$$d_k(\tau_0, \tau_n) \ge A_0 n - A_1.$$
\end{cor}

\begin{proof} 
By Theorem \ref{progress2} we have
$$d^H_k(\alpha, \beta) \ge \left \lfloor \frac{n}p \right\rfloor \ge \frac1p\cdot n -1.$$
By Lemma \ref{no extension carry} the vertex cycles of $\tau_0$ are not carried by a diagonal extension of $\tau_1$ so this gives
$$d^H_k(\tau_0, \tau_n) \ge \frac1p\cdot(n-1) - 1 = \frac1p\cdot n -\left(\frac1p +1\right).$$
From Theorem \ref{ursuladef} we have $d^H_k \le 2d_k$ so letting $A_0 = \frac{1}{2p}$ and $A_1 = \frac12\cdot\left(\frac{1}p + 1\right)$ gives the inequalities.
\end{proof}

We define the \emph{index} of a
lamination $\Lambda$ by
\[
  \ind(\Lambda) = \sup \{\ind(\tau) : \tau \mbox{ carries }\Lambda\}.
\]

\begin{cor}\label{converge to boundary}
Assume that $\tau$ carries $\Lambda$ with $\ind(\Lambda) = \ind(\tau) =k$.
Then there exists an infinite sequence of double arrows
$$\cdots \twoheadrightarrow \tau_2 \twoheadrightarrow \tau_1 \twoheadrightarrow \tau_0 = \tau.$$
Furthermore, the vertex cycles of the tracks $\tau_i$ form a uniform quality quasi-geodesic in $\C_k$  and they converge to $\Lambda$ in the Hausdorff topology on $\Sigma$.
\end{cor}

\begin{proof}
Let 
$$\cdots \to \tau'_2 \to \tau'_1 \to \tau'_0 = \tau$$
be a sequence of splittings that carries $\Lambda$. As $\ind(\Lambda) = \ind(\tau) = \ind(\tau'_i)$ we have that $\Lambda$ is fully carried by all $\tau_i$. In particular, $\Lambda$ is in the interior of $P(\tau'_i)$. Furthermore for any $j\ge 0$ if $\Lambda' = \displaystyle\bigcap_{i \ge j} P(\tau'_i)$ then $\Lambda$ and $\Lambda'$ are equal as geodesic laminations so the intersection is also contained in the interior of each $P(\lambda_i)$. (See Proposition 3.22 in \cite{BB}, for example.) Therefore for each $i$ there is a $j> i$ such that $P(\tau_j)$ is contained in the interior of $P(\tau'_i)$ or, equivalently, $\tau'_j \twoheadrightarrow \tau'_i$. We then inductively define the sequence of $\tau_i$ by choosing $\tau_{i+1}$ to be the first track in the sequence with $\tau_{i+1}\twoheadrightarrow \tau_i$.

By Lemma \ref{nodouble} the distance between $\tau_i$ and the track that appears right before $\tau_{i+1}$ (in the $\tau'_j$ sequence) is uniformly bounded so $d_k(\tau_i, \tau_{i+1})$ is uniformly bounded and $d_k(\tau_i, \tau_j)$ is bounded above by a linear function of $|i-j|$.
By Corollary \ref{distance_lower}, we have
$$d_k(\tau_i, \tau_j) \ge A_0 |i-j| - A_1.$$
Together these bounds imply that the vertex cycles of the $\tau_i$ are a quasi-geodesic in $\C_k$ and hence in $\C_\ell$ if $\ell \ge k$. 
\end{proof}

\section{Fibers}\label{sec:7}
A carrying map  $\tau\to \tau'$ between maximal tracks has {\em index $k$} if there is a vertex cycle $\alpha$ of $\tau$ whose image in $\tau'$ is carried by an index $k$ subtrack.
A sequence
$$\tau_0 \to \tau_1 \to \ldots \to \tau_n$$
of maximal train tracks has index $k$ if each individual carrying map $\tau_i\to \tau_{i+1}$ has index $k$. This can also be phrased in terms of the polyhedra. The polyhedron $P(\tau)$ is a subset of the polyhedron $P(\tau')$. Then $\tau\to \tau'$ has index $k$ if a vertex of $P(\tau)$ is contained in a co-dimension $k$ face of $P(\tau')$.

\begin{lemma}\label{upper bound}
  Let
  $$\tau_0 \to \tau_1 \to \cdots \to \tau_n$$
  be an index $k+1$ sequence. Then
  $$d_{k}(\tau_0, \tau_n) \le Cn$$
  for some constant $C$ depending on $\Sigma$.
\end{lemma}

\begin{proof}
We just need to show that $d_{k}(\tau_i, \tau_{i+1}) \le C$. Let $\alpha$ be a vertex cycle of $\tau_i$ that is carried by an index $k+1$ subtrack $\sigma \subset \tau_{i+1}$ and let $\alpha'$ be a vertex cycle of $\sigma$. Then $d_k(\alpha, \alpha')\leq 1$. The diameter in $\C_k$ of the set of vertex cycles of any train track is uniformly bounded. We then choose $C$ to be twice this bound plus one to get $d_k(\tau_i, \tau_{i+1}) \le C$.
\end{proof}

\begin{lemma}\label{no_double_index}
Assume that $\tau\to \tau'$ are maximal tracks and $\sigma\subset \tau$ and $\sigma' \subset \tau'$ are subtracks of index $k$ with $\sigma \to \sigma'$ but $\sigma\not\twoheadrightarrow\sigma'$. Then $\tau\to \tau'$ has index $k+1$.
\end{lemma}

\begin{proof}
Since $\sigma\not\twoheadrightarrow\sigma'$, there is a vertex cycle $\alpha$ of $\sigma$ that is not fully carried by $\sigma'$. In particular, $\alpha$ is carried by a subtrack of $\tau'$ with index $> \ind(\sigma') = k$. Since $\alpha$ is also a vertex cycle of $\tau$, this implies that $\tau\to\tau'$ has index $k+1$.
\end{proof}

\begin{lemma}\label{full index}
Assume that 
$$\tau\ss \tau'$$
is an index $k$ carrying map but not index $k+1$. Then there is a subsequence
$$\tau\to \eta \to \tau'$$
such that $\tau\to\eta$ is the carrying map induced by
an index $k+1$ sequence of length two and $\eta$ contains an index $k$ subtrack whose image in $\tau'$ has index $k$.
\end{lemma}

\begin{proof}
Since $\tau\to\tau'$ has index $k$ and not index $k+1$, there is an index $k$ subtrack $\sigma' \subset \tau'$ that fully carries a vertex cycle $\alpha$ of $\tau$.
Let
$$\tau= \tau_0 \to \tau_1 \to \cdots \to \tau_n = \tau'$$
be the splitting sequence and let $\sigma_i \subset \tau_i$ be the subtrack that fully carries $\alpha$. Note that $\sigma_i \to \sigma_{i+1}$ and that $\ind(\sigma_i)$ is non-increasing. Let $\eta = \tau_j$ be the first index with $\ind(\sigma_j) = k$. Then $\tau_0 \to \tau_{j-1}$ has index $k+1$ since $\ind(\sigma_{j-1}) \ge k+1$ and $\tau_{j-1}\to \tau_j$ has index $k+1$ since it is a single split. Therefore $\eta$ has the  the desired property.
\end{proof} 
Note that we allow the possibility that $\eta = \tau'$ in the above lemma.

\begin{lemma}\label{induction step}
There exist constants $A_0$ and $A_1$ such that if
$$\tau_0 \to \cdots\to \tau_n$$
is an index $k$ sequence then there is an index $k+1$ sequence
$$\tau_0= \tau'_0 \to \cdots\to \tau'_m =\tau_n$$
with 
$$m \le n\left(A_0d_{k+1}(\tau_0, \tau_n) + A_1\right).$$
\end{lemma}

\begin{proof}
For $i=0, \dots, n-1$ we will show that $\tau_i \to \tau_{i+1}$ can be written as an index $(k+1)$-sequence of length $\le A_0d_{k+1}(\tau_0, \tau_n) + A_1$. Combining these sequences will give the lemma.

Let $\eta$ be the track given by applying Lemma \ref{full index} to the index $k$ carrying map $\tau_i \to \tau_{i+1}$. In particular there is a subtrack $\sigma \subset \eta$ whose image in $\tau_{i+1}$ is a subtrack $\sigma'$ with $\ind(\sigma) = \ind(\sigma') = k$.
The sequence $\eta \ss\tau_{i+1}$ induces a splitting sequence $\sigma \ss \sigma'$. 
We can construct a subsequence 
$$\sigma= \sigma_0 \to \sigma_1 \to \cdots \sigma_{2\ell} \to \sigma_{2\ell+1} = \sigma'$$
with $\sigma_{2j} \twoheadrightarrow \sigma_{2(j+1)}$ but $\sigma_j \not\twoheadrightarrow \sigma_{j+1}$ as follows. Assume $\sigma_j$ has been chosen and is even. Then let $\sigma_{j+1}$ be the last track in the sequence $\sigma_j\ss\sigma'$ such that  $\sigma_j \not\twoheadrightarrow\sigma_{j+1}$.\footnote{If $\sigma_j = \sigma'$ then $\sigma_{j+1} = \sigma'$ and $\sigma_j\to\sigma_{j+1}$ is the identity map.} If $\sigma_{j+1} = \sigma'$ we stop. If not we let $\sigma_{j+2}$ be the next track in the sequence. That is $\sigma_{j+1}\to \sigma_{j+2}$ is a single split.  By Corollary \ref{distance_lower}, the constant $\ell$ is bounded above by a linear function of $d_{k+1}(\sigma_0, \sigma_{2\ell})$. Theorem \ref{quasigeodesics} and Proposition \ref{hyperbolic} imply that the vertex cycles of $\tau_0, \sigma_0, \sigma_{2\ell}$ and $\tau_n$ lie on a uniform quality quasigeodesic in $\C_{k+1}$ (in that order). Therefore $d_{k+1}(\sigma_0, \sigma_{2\ell}) - d_{k+1}(\tau_0,\tau_n)$ is bounded above uniformly so $\ell$ is bounded above by a linear function of $d_{k+1}(\tau_0, \tau_n)$.

We then have a subsequence
$$\eta = \tau'_0 \to \cdots \tau'_{2\ell+1} = \tau_{i+1}$$
of $\eta\ss\tau_{i+1}$ with $\sigma_j$ a subtrack of $\tau'_j$. By Lemma \ref{no_double_index} this is an index $k+1$ sequence. By Lemma \ref{full index} we can write $\tau_i \to \eta$ as an index $k+1$ sequence of length two so $\tau_i \to \tau_{i+1}$ can be written as an index $k+1$ sequence of length $2\ell+3$. As $\ell$ is bounded by a linear function of $d_{k+1}(\tau_0, \tau_n)$ so is $2\ell+3$ and the lemma follows. \end{proof}

\begin{thm}\label{block bound}
There exist constants $A_0$ and $A_1$ such that the following holds.
  Assume that $\tau \ss \tau'$ are maximal tracks. Then there exists an index $(k+1)$-subsequence
  $$\tau=\tau_0 \to \tau_1 \to \ldots \to \tau_n=\tau'$$
with 
$$\frac1{A_0} \cdot d_k(\tau, \tau') \le n \le \left(A_0 d_{k}(\tau, \tau') + A_1\right)^{k+1}.$$
\end{thm}

\begin{proof} We induct on $k$. Note that $\tau\to\tau'$ is an index $0$ sequence so if $k=0$ the result follows from Lemma \ref{induction step}. This is the base case.

Now assume that statement holds for $k-1$. In particular $\tau\to\tau'$ can be written as an index $k$ sequence of length $\le (A_0 d_k(\tau,\tau') +A_1)^{k-1}$. Then by Lemma \ref{induction step} we can write $\tau\to \tau'$ as a sequence of length
$$\le (A_0 d_k(\tau,\tau') +A_1)^{k-1} \cdot (A_0 d_{k+1}(\tau, \tau') + A_1).$$
As $d_k(\tau, \tau') \le d_{k+1}(\tau, \tau')$ this implies the theorem.

The left inequality is Lemma \ref{upper bound}.
\end{proof}

\section{Quasi-trees from train tracks}
Here we construct certain sets $Q^k_i(\tau)\subset \S(\tau)$ of curves
carried by a birecurrent
train track $\tau$. Viewed as subsets of $\C_{k-1}$, these sets will be
bounded and will give an exhaustion of $\C_{k-1}$, in the sense that
any bounded subset of $\C_{k-1}$ will be contained in some
$Q^k_i(\tau)$ for a suitable $\tau$ and $i$. As subsets of $\C_k$,
these sets will be quasi-convex and (with the subspace metric from
$\C_k$) they will be quasi-isometric to trees.

In the following theorem, $B^k_D(\tau)$ denotes the $D$-neighborhood of $V(\tau)$
in $\C_k$.

\begin{thm}\label{exhaustion}
Let $\tau$ be a maximal, birecurrent train track. Then there are sets of curves
$$Q^k_1(\tau) \subset Q^k_2(\tau) \subset \cdots \subset \S(\tau)$$
such that each $Q^k_i(\tau)\subset \C_k$ is a quasi-tree with constants depending only on $i$ and for all $D>0$ there exists $i = i(D)$ such that
$$B^k_D(\tau) \cap \S(\tau) \subset Q^{k+1}_i(\tau).$$
\end{thm}

Given the maximal train track $\tau$, we fix a directed rooted tree $\T$. The vertices of $\T$ are pairs $(\tau_0, b_0)$ where $\tau_0$ is a recurrent train track and $b_0$ is a large branch in $\tau_0$. The directed edges of $\T$ are carrying maps $(\tau_0, b_0) \longrightarrow (\tau_1, b_1)$ where $\tau_0$ is obtained by splitting $\tau_1$ along $b_1$. 
The basepoint of $\T$ is the initial track $\tau$. We assume that $\T$ has been chosen and that it is maximal. That is, for every for vertex $(\tau_1, b_1)$ there are two predecessors if both splittings along $b_1$ are recurrent or one predecessor if only one of the splittings is recurrent. We also assume that every large branch that occurs in some train track in $\T$ is eventually split. 

\begin{lemma}\label{split to vertex cycle}
Let
$$ \cdots \to \tau_i \to \cdots \to \tau_0 = \tau$$
be an infinite path in $\T$ such that $\gamma \in \S(\tau_i)$ for all $i$. Then for $i$ sufficiently large, $\gamma$ is a vertex cycle of $\tau_i$.
\end{lemma}

\begin{proof}
Let $\sigma_i \subset \tau_i$ be the smallest subtrack that carries $\gamma$. Note that the fact that the tracks $\sigma_i$ are the subtracks that minimally carry $\gamma$ implies that $\sigma_{i+1}\to\sigma_i$ and furthermore each carrying map is either homotopic to a homeomorphism or is a single splitting along a large branch. In this last case it is possible that the split is a central split. By the argument at the end of the proof of Lemma \ref{diagonal extension} every large branch that appears in any $\sigma_i$ is split.

Let $w_i$ be the maximal weight of $\gamma$ on $\sigma_i$ and $k_i$ the number of branches where the maximum is realized. Any branch realizing the maximal weight is a large branch of $\sigma_i$ and, as noted above, this branch will eventually be split.  When this happens the complexity of the pair $(w_i, k_i)$ (with the lexographic order) will strictly decrease. Therefore we must eventually have that $\sigma_i$ doesn't contain any large branches which implies that $\sigma_i$ is a curve and that $\gamma$ is a vertex cycle in $\tau_i$.
\end{proof}

 Next we define a collection $\T_k$ of index $k$ subtracks of tracks in $\T$.  Let $\tau_0 \to \tau_1$ be a directed edge in $\T$ and $\sigma_0 \subset \tau_0$ an index $k$ subtrack. Then $\sigma_0 \in \T_k$ if the image of $\sigma_0$ in $\tau_1$ has index $< k$. In particular, if $\sigma_0 \in \T_k$ and $\tau' \neq \tau_0$ is a vertex in $\T$ with $\tau_0\to \tau'$ then the image of $\sigma_0$ in $\tau'$ has index $< k$. Note that it is possible that $\sigma_0 \to \sigma_1$ with both tracks in $\T_k$ but in this case there are tracks $\tau_0$ and $\tau_1$ in $\T$ with $\sigma_i \subset \tau_i$ but $\tau_0 \not\to\tau_1$. On the other hand, if $\sigma \not\in\T_k$ is an index $k$ subtrack of some track in $\T$ then there is $\sigma_0 \in \T_k$ with $\sigma \to \sigma_0$.
  Furthermore, by Lemma \ref{split to vertex cycle}, for every $\gamma \in \S(\tau)$ there is $\sigma \in \T_k$ with $\gamma \in \S(\sigma)$.

\begin{prop}\label{intersection_small}
There exists $N = N(\Sigma)$ such that the following holds. Assume that $\tau_1, \dots, \tau_m \in \T$ with $\tau_i \not\to \tau_j$ for all $i\neq j$. If $m > N$ then
$$\bS = \S(\tau_1) \cap \cdots \cap\S(\tau_m)$$
does not fill $\Sigma$.
\end{prop}

\begin{proof}
 We'll assume that $\bS$ fills and obtain a bound on $m$.
 
Let $\T_0$ be the minimal subtree of the directed tree $\T$ that contains all the $\tau_i$. Let $\tau_0$ be the terminal vertex in $\T_0$ so that for each $\tau_i$ there is a unique directed path in $\T_0$ from $\tau_i$ to $\tau_0$. The assumption that $\tau_i \not\to \tau_j$ implies that $\tau_j$ does not appear in the path from $\tau_i$ to $\tau_0$. We will bound the total number of $\tau_i$ by showing that there is $N'>0$ such that the number of vertices in this path that have two forward entering edges is $\le N'$. Then we will set $N = 2^{N'}$.

   For each vertex $\tau' \in \T_0$ we let $\sigma'$ be the smallest subtrack of $\tau'$ such that $\bS\subset \S(\sigma')$. We will define a complexity for the pair $(\tau', \sigma')$ and show that it decreases when we pass through a vertex with two entering edges.

    The complexity for $(\tau', \sigma')$ will be a pair $(b(\tau', \sigma'),p(\tau', \sigma')) = (b', p')$ where $b'$ is the number of branches of $\tau'$
  that belong to $\sigma'$, and $p'$ is the number of
  train paths in $\tau'$ that do not traverse any branches of
  $\sigma'$. Note that branches of $\sigma'$ may be unions of branches in $\tau'$  
  and $b'$ counts all the branches of $\tau'$ contained in $\sigma'$. Since the total number of branches of $\tau'$ is bounded (by a constant depending on $\Sigma$) $b'$ is bounded. Any legal train path in $\tau'\smallsetminus \sigma'$ traverses a branch at most twice, for if it traversed a branch three times $\tau'\smallsetminus \sigma'$ would contain an essential curve, contradicting that $\sigma'$ is filling. Therefore there is also a uniform bound on $p'$, again depending only on $\Sigma$. The total number of possible complexities is therefore bounded and this bound will be our $N'$. We order the pairs $(b',p')$ lexicographically.

  Assume that $\tau'$ has two entering edges in $\T_0$. Then there is a split of $\tau'$ along a large branch yielding two tracks $\tau'_1$ and $\tau'_2$ that are both vertices in $\T_0$. That is $\bS\subset \S(\tau'_1) \cap \S(\tau'_2)$ which implies that $\bS \subset \S(\hat\tau)$ where  $\hat\tau$ is the central split. In particular, the new small branch created by the split in both $\tau'_1$ and $\tau'_2$ is not be contained in $\sigma'_1$ or $\sigma'_2$.
We claim the complexity $(b', p')$ is strictly greater than both $(b'_1,p'_1)$ and $(b'_2, p'_2)$. Note that the image of $\sigma'_i$ in $\tau'$ under the carrying map $\tau'_i \to \tau'$ is $\sigma'$ but the pre-image of $\sigma'$ in the $\tau'_i$ may be strictly smaller than the $\sigma'_i$.

There are two cases:
  
{\bf Case 1:} The large branch of $\tau'$ being split is contained in $\sigma'$. Then either $\sigma'_1$ and $\sigma'_2$ are a central split of $\sigma'$ and $b'_i = b'-2$ or $\sigma'_i \to \sigma'$ is a homeomorphism and $b'_i = b' - 1$.

{\bf Case 2:}  The branch of $\tau'$ being split is not in $\sigma$. Then the carrying map restricts to a (branch preserving) homeomorphism from $\sigma'_i$ to $\sigma'$ so $b_i' = b'$. The carrying map also takes $\tau'_i \smallsetminus\sigma'_i$ to $\tau'\smallsetminus\sigma'$ and distinct legal paths in $\tau'_i \smallsetminus\sigma'_i$ map to distinct legal paths in $\tau'\smallsetminus\sigma'$. On the other hand, there are legal paths in $\tau'\smallsetminus \sigma'$ that are not the image of legal paths in $\tau'_i\smallsetminus\sigma'_i$. This implies that $p'_i < p'$.
\end{proof}

As a corollary of this proposition and Proposition \ref{k alternative} in the Appendix we have:
\begin{cor}\label{qtree_ii}
There exists $N = N(\Sigma)$ such that if $\sigma_1, \dots, \sigma_m \in \T_k$ and $m> N$ then the diameter of
$$I_C(\S(\sigma_1), \dots, \S(\sigma_m))$$
is uniformly bounded.
\end{cor}

\begin{proof}
If $\S(\sigma_1) \cap \cdots \cap \S(\sigma_m) = \emptyset$ then the corollary is (1) of Proposition \ref{k alternative}.
If $\S(\sigma_1) \cap \cdots \cap \S(\sigma_m) \neq \emptyset$ then by (2) of Proposition \ref{k alternative} we have 
$$I_{C}(\S(\sigma_1), \dots, \S(\sigma_m)) \subset B_{C'}(\S(\tau_1) \cap\cdots \cap\S(\tau_m))$$
for some $C'$ depending on $C$. To finish the proof we need to bound the diameter of the intersection on the right.

For each $\sigma_i$ we have a $\tau_i \in \T$ with $\sigma_i$ a
subtrack. If $\tau_i \to \tau_j$ for $i \neq j$ then, by the construction of $\T_k$,  $\sigma_i
\not\to\sigma_j$.  Since $\S(\sigma_i) \cap \S(\sigma_j) \neq \emptyset$ there is a subtrack $\sigma'_i \subset \sigma_i$ with $\S(\sigma'_i) = \S(\sigma_i) \cap \S(\sigma_j)$. As $\sigma_i \not\to\sigma_j$ the track $\sigma'_i$ must be a proper subtrack and hence $\ind(\sigma'_i) > k$. Therefore
$$B_{C'}(\S(\sigma_1)\cap \cdots \cap \S(\sigma_m)) \subset B_{C'}(\sigma'_i)$$
and as $\ind(\sigma'_i) > k$ the diameter of this set in $\C_k$ is bounded by $2C'+1$.

Now we can assume that $\tau_i \not\to \tau_j$ for all $i\neq j$. Then by Proposition \ref{intersection_small} if $m>N$ we have that the curves in 
$$\bS = \S(\tau_1) \cap\cdots \cap\S(\tau_m)$$
don't fill $\Sigma$ so $\S$ has diameter 1 in $\C_k$. 
\end{proof}

Define $\psi_k\colon \T \to \N$ such that $\psi_k(\tau') \le m$ if there is an index $k$ subsequence
$$\tau'=\tau_0 \to \tau_1 \to \cdots \to\tau_m = \tau$$
in $\T$ and $\psi_k$ is the largest function with this property. We extend the domain
of $\psi_k$ to subtracks of tracks in $\T$ as follows.
If $\sigma$ is a subtrack of $\tau'$ then we set $\psi_k(\sigma) = \psi_k(\tau')$.

\begin{lemma}\label{large distance}
Given $\alpha\in \S(\tau)$ there exists $\sigma \in \T_{k+1}$ with $\alpha \in \S(\sigma)$ and $\psi_{k+1}(\sigma) \le (A_0 d_k(\alpha, \tau) + A_1)^{k+1}$.
\end{lemma}

\begin{proof}
By Lemma \ref{split to vertex cycle} the curve $\alpha$ is a vertex cycle for some track in $\T$.
In particular, $\alpha$ is carried by an index $k+1$ subtrack of a track in $\T$ and therefore $\alpha$ is carried by a track $\sigma \in \T_{k+1}$ which is a subtrack of a track $\tau_0 \in \T$.
By Proposition \ref{hyperbolic} the vertex cycles of $\tau_0$ lie on a uniform quality quasigeodesic from $V(\tau)$ to $\alpha$ and therefore $d_k(\tau, \tau_0) - d_k(\tau, \alpha)$ is uniformly bounded. Applying this and Theorem \ref{block bound} gives the result.
\end{proof}

Set $Q^k_i(\tau)$ to be the union of the sets  $\S(\sigma)$ such that $\sigma \in \T_k$ has the property that $\psi_{k+1}(\sigma) \le i$.
Note that $Q^k_0(\tau)$ corresponds to the index $k+1$ subtracks of $\tau$ and this is a finite set of bounded size. However, when $i \ge 1$ there are infinitely many train tracks in $Q^k_i(\tau)$ and in fact $Q^k_i(\tau) \smallsetminus Q^{k}_{i-1}(\tau)$ has infinitely many tracks.

\begin{proof}[Proof of Theorem \ref{exhaustion}]
We show that the sets $Q^k_i(\tau)$ are quasi-trees by applying Theorem \ref{quasitreelem} inductively. As noted above $Q^{k}_0(\tau)$ is the index $k+1$ subtracks of $\tau$ and has uniformly bounded size and hence is a quasi-tree. This is the base case.

For the induction step we assume that $T_0 = Q_{i-1}^{k}(\tau)$ is a
quasi-tree. For $j>0$ we let $T_j$ be the collection of $\S(\sigma)$
with $\sigma \in \T_k$ and $\psi_{k+1}(\sigma) = i$. In particular, $\sigma$ is a subtrack of $\tau_0 \in \T$ with $\psi_k(\tau_0)  = i$ so there is $\tau_1 \in \T$ with $\tau_0 \to \tau_1$ an index $k+1$ carrying map and $\psi_k(\tau_1) = i-1$. Therefore all the curves carried in index $k+1$ subtracks of $\tau_1$ are in $T_0$ and $d_k(\sigma, T_0)$ is uniformly bounded verifying (i) of Theorem \ref{quasitreelem}. Property (ii) of Theorem \ref{quasitreelem} follows from Corollary \ref{qtree_ii}.

 The last sentence follows from Theorem \ref{large distance}.
\end{proof}

We need one more lemma in the proof of Theorem \ref{mainthm}.

\begin{lemma}\label{tau}
  Let $B$ be a set in $\C_k$ of finite diameter. Then there is a
  maximal, birecurrent train track $\tau$ such that every curve in $B$
  is carried by $\tau$. Furthermore, $d_k(\tau, B)$ is bounded by a
  constant only depending on $\Sigma$ and  $\diam B$.
\end{lemma}

\begin{proof}
As there are only finitely many curves up to the action of the mapping class group, we can fix a curve $\gamma$ such that $\phi(B)$ is contained in the ball of radius $R$ around $\gamma$ where $\phi$ is element of the mapping class group and $R- \diam B$ is uniformly bounded. In particular, the theorem will follow if we can show that the radius $R$ neighborhood $\gamma$ is carried by some maximal train track $\tau$.

  Fix a maximal train track $\tau$. As $P(\tau)$ contains an open set in $\PML$ and index zero laminations are dense in $\PML$ we have that $\tau$ carries an index zero lamination. Then by Corollary \ref{converge to boundary} there is a sequence
  $$\cdots \twoheadrightarrow \tau_2 \twoheadrightarrow \tau_1 \twoheadrightarrow \tau_0 = \tau.$$
  Then, by Corollary \ref{distance_lower}, if $\alpha \in \S(\tau_i)$ and $\beta\not\in\S(\tau)$ we have
  $$d_k(\alpha, \beta) \ge d_0(\alpha, \beta)\ge A_0\cdot i - A_1.$$
  Again applying that there are only finite many curves up to the action of the mapping class group we can assume that $d_k(\alpha, \gamma)\le D$ where $D$ only depends on $\Sigma$ and $\alpha$ is a vertex cycle of $\tau_i$. We then choose $i$ to be the first integer larger than $\frac1{A_0}\left(R+A_1 +D\right)$ and we have that $B \subset \S(\tau)$.
\end{proof}  

\begin{proof}[Proof of Theorem \ref{mainthm}]
  Consider the projection $\C_{k+1}(\Sigma)\to\C_{k}(\Sigma)$ and a
  bounded set $B\subset \C_{k}(\Sigma)$. By Lemma \ref{tau} there is
  a birecurrent maximal track $\tau$ that carries every curve in
  $B$. By Theorem \ref{exhaustion} there is some $i$ so that $B\subset
  Q^{k+1}_i(\tau)$ and $Q^{k+1}_i(\tau)$ is bounded in $\C_k$, but its
  preimage in $\C_{k+1}$ (which equals $Q^{k+1}_i(\tau)$) is a
  quasi-tree.
\end{proof}

\subsection{Bounds on asymptotic dimension}

\begin{thm}\label{maincor}
  For every $k$ we have $\asdim \C_{k+1}\leq\asdim\C_k +1$
  and $\asdim \C_k\leq k+1$. In
  particular, $\asdim\C=\asdim\C_{m(\Sigma)}\leq m(\Sigma)+1$.
\end{thm}

To compute bounds on the asymptotic dimensions of the graphs $\C_k(\Sigma)$,
we apply the following Hurewicz-type theorem for asymptotic dimension due to
Bell-Dranishnikov. We refer the reader to \cite{Hurewicz} for the definition of
a family of subsets $\{X_\alpha\}$ of a metric space $X$ satisfying the
inequality $\asdim X_\alpha \leq n$ \emph{uniformly}.

\begin{thm}[\cite{Hurewicz}]
  Let $f:X\to Y$ be a Lipschitz map of geodesic metric spaces. Suppose
  that for every $R>0$, the family of sets $\{f^{-1}(B_R(y))\}_{y\in Y}$
  satisfies the inequality $\asdim f^{-1}(B_R(y))\leq n$ uniformly.
  Then $\asdim X\leq \asdim Y + n$.
\end{thm}

\begin{proof}[Proof of Theorem \ref{maincor}]
  Let $\pi:\C_{k+1}\to \C_k$ be the natural map. Then for each
  $y\in \C_k$, $\pi^{-1}(B_R(y))$ is contained in a quasi-tree, hence it
  has asymptotic dimension $\leq 1$, as a subspace of a space with asymptotic
  dimension $\leq 1$. On the other hand, there are finitely many sets $B_R(y)$
  up to the action of $\Mod(\Sigma)$, and $\Mod(\Sigma)$ also permutes the
  sets $\pi^{-1}(B_R(y))$ by isometries. This shows that
  $\asdim \pi^{-1}(B_R(y))\leq 1$ \emph{uniformly} in $y$. Thus
  $\asdim \C_{k+1} \leq \asdim C_k + 1$ for each $k$ and the theorem
  follows by induction on $k$.
\end{proof}

\begin{remark} With the value of $m(\Sigma)$ from Proposition \ref{ms}
  this recovers the bound on $\asdim\C$ from \cite{BB}.
\end{remark}

It would be interesting to know whether $\asdim\C_k=k+1$, i.e. whether
equality holds in $\asdim \C_{k+1} \leq \asdim \C_k + 1$.

\section{Boundary}

In this section we characterize the boundaries and loxodromics of the graphs
$\mathcal C_k(\Sigma)$. The main tools are Corollary \ref{converge to boundary} and the following theorem of Dowdall-Taylor.
The setting is that $X,Y$ are hyperbolic metric spaces and $p:X\to Y$ is
alignment-preserving. Then the \emph{$Y$-subboundary of $X$}, denoted $\partial_Y X$, is defined to be the
set of equivalence classes of quasi-geodesic rays $\gamma:\R_+\to X$ with
$\diam_Y p(\gamma) = \infty$.

\begin{thm}[{\cite[Theorem 3.2]{DT}}]
  \label{DT}
  Suppose that $p:X\to Y$ is a coarsely surjective, alignment-preserving map
  between hyperbolic spaces. Then $p$ extends to a homeomorphism
  $\partial_p : \partial_Y X \to \partial Y$.
\end{thm}

Denote by $\mathcal{EL}(\Sigma)$ the space of ending laminations on $\Sigma$
with the \emph{coarse Hausdorff topology}. This is the quotient topology
induced by considering the subspace of $\mathcal{PML}(\Sigma)$ consisting
of projectivized measured ending laminations, and identifying two measured
laminations with the same underlying lamination. 
We consider the subspace $\mathcal{EL}_k(\Sigma)$ consisting of ending laminations
of index $\leq k$.

\begin{thm}
  \label{boundary}
  The Gromov boundary of $\C_k$ is equivariantly homeomorphic to
  $\mathcal{EL}_k$.
\end{thm}

\begin{proof}
  Consider the natural map
  $p:\mathcal C \to \mathcal C_k$. By \cite{HamenstadtTT},
  the boundary of
  $\C$ is identified with $\mathcal{EL}$ where a lamination $\Lambda \in \mathcal{EL}$ is identified with a quasi-geodesic $c_0, c_1, \dots$ in $\C$ that converges to $\Lambda$ in the coarse Hausdorff topology. By Corollary \ref{converge to boundary} we can choose the $c_i$ to lie in $\S(\tau)$ where $\ind(\tau) = \ind(\Lambda)$. If $\ind(\Lambda) = \ind(\tau) \ge k$ then Corollary \ref{converge to boundary} implies that the $c_i$ are a quasi-geodesic in $\C_k$. If $\ind(\Lambda) < k$ then the $d_k(c_i, c_j) = 1$  and the quasi-geodesic has finite diameter in $\C_k$. The result then follows from Theorem \ref{DT}.
  \end{proof}

\section{Acylindricity}

Our aim in this section is to prove the following theorem:
\begin{thm}
  \label{allacyl}
  The action of $\MCG(\Sigma)$ on $\C_k(\Sigma)$ is acylindrical for every $k$.
\end{thm}

Recall that for an element $g$ of a group $G$ acting on the hyperbolic space
$X$, the \emph{asymptotic translation length}
\[
  \ATL(g)\vcentcolon=\lim_{n\to \infty} \frac{d_X(g^nx, x)}{n}
\]
is well-defined, independently of the chosen basepoint $x$.

We start by proving that there is a positive lower bound on asymptotic
translation lengths of loxodromic elements in $\C_k$, which is a necessary condition for acylindricity.

\begin{thm}
  \label{loxclassification}
  The mapping class $\phi \in \MCG(\Sigma)$ is loxodromic in the action on $\C_k(\Sigma)$
  if and only if $\phi$ is pseudo-Anosov with fixed laminations lying in
  the Gromov boundary $\partial \C_k$. Equivalently, $\phi$ is loxodromic if and
  only if all its invariant train tracks have index $\leq k$.
  Moreover, there is a constant $E=E(\Sigma)$ depending
  only on $\Sigma$ such that whenever $\phi$ is loxodromic on $\C_k$, we have $\ATL(\phi)\geq E$ on $\C_k$.
\end{thm}

\begin{proof}
  If $\phi$ is not pseudo-Anosov then it acts elliptically on $\C$ and hence
  also on $\C_k$.
If $\phi$ is pseudo-Anosov then for any $\alpha$, the sequence $\{\alpha, \phi(\alpha), \phi^2(\alpha), \dots\}$ is a quasi-geodesic in $\C$ that converges to the stable lamination $\Lambda^+$.
Then $\{\alpha, \phi(\alpha), \phi^2(\alpha), \dots\}$ is a quasi-geodesic in $\C_k$ if $\ind(\Lambda) \le k$ and has bounded diameter if $\ind(\Lambda) > k$. In particular, $\phi$ is loxodromic if $\ind(\Lambda) \le k$ and is elliptic if $\ind(\Lambda) > k$. Note that $\phi$ is elliptic if and only $\phi^{-1}$ is elliptic so this argument implies that $\Lambda^-$, the unstable lamination, has the same index as $\Lambda^+$.

When $\phi$ is loxodromic we need a more subtle argument to get a uniform lower bound on $\ATL(\phi)$.
Let $\tau'$ be a train track carrying $\Lambda$ with $\ind(\tau') =\ind(\Lambda)$. Then by
  \cite[Section 3]{Agol} we may split $\tau'$ to a track $\tau$
  such that $\phi(\tau) \ss \tau$. As the tracks $\tau_n=\phi^n(\tau)$ all have index equal to $\ind(\Lambda^+)$, we have that $\Lambda^+$ is fully carried by every $\tau_n$. Therefore there is an $m>0$ such that $P(\tau_m)$ is contained in the interior of $P(\tau)$.
   
The number of proper faces of $P(\tau)$ is $\le N$ where $N$ only depends on $\Sigma$. We claim that if $\gamma$ is a vertex cycle of $\tau$ then $\phi^N(\gamma)$ is fully carried by $\tau$. If not there are $0\le i<j\le N$ such that $\phi^i(\gamma)$ and $\phi^j(\gamma)$ are both contained in the interior of the same proper face $F$ of $P(\tau)$. This implies that $\phi^{j-i}(F) \subset F$, i.e. $\phi^{n(j-i)}(F) \subset F$ for all $n\ge 0$. On the other hand since $P(\phi^m(\tau))$ is contained in the interior of $P(\tau)$ if $n(j-1) \ge m$ then $P(\phi^{n(j-i)}(\tau))$ is contained in the interior of $P(\tau)$, a contradiction. Therefore $\phi^N(\tau)$ is fully carried by $\tau$ or, equivalently, $\tau_N \twoheadrightarrow \tau$. Periodicity implies that
$$\cdots \twoheadrightarrow \tau_{2N} \twoheadrightarrow \tau_N \twoheadrightarrow \tau_0 = \tau$$
and by Corollary \ref{distance_lower} we have $d_k(\gamma, \phi^{nN}(\gamma))\ge d_k(\tau, \tau_{nN}) \ge A_0 n - A_1$ whenever $k \ge \ind(\Lambda^+)$ and the constants $A_0$ and $A_1$ only depend on $\Sigma$. Therefore $\ATL(\phi) \ge A_0/N$ is uniformly bounded below. 
\end{proof}

Theorem \ref{allacyl} now follows immediately from the theorem of
Bowditch \cite{bowditch:tight} that the action on $\C$ is acylindrical and the following more general theorem.
\begin{thm}
  \label{generalacyl}
  Let $X$ and $Y$ be hyperbolic graphs with actions of the virtually torsion-free group $G$.
  Let $f:X\to Y$ be a Lipschitz alignment-preserving $G$-equivariant
  map which is surjective on vertices.
  Finally, suppose there exists $\omega>0$ such that for any
  $g\in G$ which acts loxodromically on $Y$, its asymptotic
  translation length $\ATL_Y(g)$ on
  $Y$ is at least $\omega$. If the action of $G$ on $X$ is acylindrical, so is the action
  on $Y$.
\end{thm}

We use the following general facts about a $\delta$-hyperbolic space $Z$.
First of all, for any elliptic isometry $g$ of $Z$,
there is a point $z\in Z$ with $d(z,g^nz)\leq 4\delta + 1$
for all $n\in \Z$ (this follows from \cite[Lemma 7.3]{Hull}).
Denote by
$\mathcal{F}_D(g)$ the set of points whose orbit under $g$ has diameter
bounded by $D$. Thus $\mathcal{F}_{4\delta+1}(g)\neq \emptyset$. Moreover,
any geodesic $[y,gy]$ passes within $4\delta$ of $\mathcal{F}_{4\delta+1}(g)$.
Consequently if $y$ is at distance $C$ from $\mathcal{F}_{4\delta+1}(g)$ then
$d(y, gy)\geq 2C-8\delta$.

If $g$ is a loxodromic isometry of $Z$ with $\ATL(g)>2\delta$ then $g$
has a quasi-geodesic axis $\ell$ with uniform constants. Simply take a vertex
$v$ that minimizes $d(v,gv)$ and let $\ell=\cup_{n\in \Z}
g^n([v,gv])$. Without the assumption on $\ATL$, in general
quasi-isometric axes with uniform constants do not exist.
In addition, if $\ATL(g)>2\delta$ then for any $x\in X$ a geodesic
segment $[x,gx]$ comes within a uniform distance of $\ell$.

If the action of $G$ on $Z$ is acylindrical,
then for any $D>0$ there is a constant $S=S(G,D)>0$ such that if
$g\in G$ with
$\diam(\mathcal{F}_D(g))>S$ then $g$ has finite order. Finally, if the
action of $G$ on $Z$ is acylindrical, there is a uniform positive
lower bound to $\ATL(g)$ for all loxodromic $g$.

\begin{proof}[Proof of Theorem \ref{generalacyl}]
  Let $\delta$ be a constant of hyperbolicity for both $X$ and $Y$. First observe that
  if a finite-index subgroup of $G$ acts acylindrically on $Y$,
  then so does $G$ itself. Hence we may assume without loss of generality that $G$
  itself is torsion-free and for any $g\in G$ acting elliptically on $X$, the set
  $\mathcal{F}_D(g)$ has diameter $\leq S(D)$.

  Having assumed $G$ to be torsion-free, for any
  $\epsilon>0$, we will determine a constant $R$ depending on $\epsilon$ such that
  if $x,y\in Y$ with $d(x,y)>R$ then there is a bounded number of elements of $G$ moving
  both $x$ and $y$ a distance $\leq \epsilon$. The constant $R$ will be determined during the proof.

  Suppose that $d(x,y)>R$ and that $g\in G$ satisfies $d(x,gx)\leq
  \epsilon$ and $d(y,gy)\leq
  \epsilon$. We consider
  three cases depending on the actions of $g$ on $X$ and $Y$ and show
  that the number of such $g$ is bounded if $R$ is large.

  First suppose that
  $g\neq 1$ acts elliptically on $X$, and hence on $Y$.
  Choose two points $\widetilde x, \widetilde y \in X$
  with $f(\widetilde x)=x$ and $f(\widetilde y)=y$. Then
  $\mathcal F = \mathcal F_{4\delta+1}(g)\subseteq X$ is non-empty and has diameter bounded
  by $S=S(G,4\delta+1)$. Moreover, $[\widetilde x,g\widetilde x]$ and
  $[\widetilde y,g \widetilde y]$ both pass within $4\delta$ of $\mathcal F$ and
  hence they pass within $S+8\delta$ of each other.
  Hence, $[x,gx]$ and $[y,gy]$
  pass within $D$ of each other where $D$ depends only on $\delta$ and
  the Lipschitz and 
  alignment-preserving constants for $f$ and the acylindricity constants for the action of $G$ on $X$. If $R$ is sufficiently large compared
  to $\epsilon$, this is impossible, so there are no such elements
  $g\neq 1$.

  Now suppose that $g$ acts loxodromically on $X$ but elliptically on
  $Y$. Since the action of $G$ on $X$ is assumed to be acylindrical,
  there is a uniform positive lower bound to the asymptotic
  translation length of $g$ on $X$. Thus, there is a uniformly bounded power
  $g^n$ of $g$ whose asymptotic translation length is
  $>2\delta$. Note that since $g^n$ is loxodromic, any of its roots
  are loxodromic with the same fixed points in the Gromov boundary, hence lying in the same
  cyclic group, and this implies that $g$ is the only 
 $n^{\text{th}}$ root for $g^n$. 
  Thus, replacing $\epsilon$ by $n\epsilon$, we may assume that
  $g$ itself has asymptotic translation length $>2\delta$ and it
  suffices to bound the number of such $g$'s.
  
  Choose $\ell$ to be an invariant quasi-geodesic axis for $g$ in $X$
  with uniform quasi-geodesic constants. Since $\ATL(g)>2\delta$ the geodesics
  $[\widetilde x,g\widetilde x]$ and $[\widetilde y,g\widetilde y]$
  pass uniformly closely to
  $\ell$.
  It then follows from the alignment-preserving
  properties that geodesics between points in $X$ and their images
  under $g$ pass uniformly close to $\mathcal F_{4\delta+1}(g)$ that
  $f(\ell)$ has uniformly bounded diameter.
  By the Lipschitz property, $[x,gx]$ and $[y,gy]$ pass
  uniformly close to $f(\ell)$ and thus uniformly close to each other.
  Again, if $R$ is sufficiently
  large compared to $\epsilon$, then this is impossible.

  Note that if $g$ acts loxodromically on $X$ then it cannot act parabolically
  on $Y$ by the alignment-preserving assumption.
  So the final case is that $g$ is loxodromic in both $X$ and $Y$. Again,
  thanks to our assumption that $\ATL(g)\geq\omega>0$, by passing to
  a bounded power we may assume that $g$ has asymptotic translation
  length $>2\delta$ in both $X$ and $Y$. Let $\ell$ be a
  quasi-geodesic axis of $g$ in $X$ with uniform constants. Then
  $f(\ell)$ is a quasi-geodesic axis of $g$ in $Y$; it is
  reparametrized but with uniform constants, because $f$ is alignment-preserving.
  The distance between $x$ and $f(\ell)$ is uniformly
  bounded in terms of $\epsilon$ and constants coming from $X,Y,$ and
  $f$, because a geodesic $[x,g(x)]$ comes uniformly close to
  $f(\ell)$ by our assumption that $\ATL(g)>2\delta$. For the same
  reason, $y$ is at a uniformly bounded distance from $f(\ell)$.

  Since $\ATL(g)\leq\epsilon$,
  by making $R$ large compared to $\epsilon$
  we can ensure that $[x,y]$ fellow travels
  $f(\ell)$ along a segment that contains as many fundamental domains
  of $g$ acting on $f(\ell)$ as we like. In particular, $[x,y]$ fellow
  travels $f(\ell)$ along a segment which is large compared to the
  asymptotic translation length of $g$.
  Lift $x,y$ to $\widetilde
  x,\widetilde y\in X$ and note that a geodesic $[\widetilde
    x,\widetilde y]$ also must fellow travel a segment in $\ell$ that
  contains as many fundamental domains of $g$ acting on $\ell$. In
  fact, $[\widetilde
    x,\widetilde y]$ can be written as a concatenation of three
  segments where the middle segment fellow travels $\ell$ and the
  other two map to bounded segments joining $x$ and $y$ to points uniformly near $f(\ell)$.

  To
  summarize, given any $g$ as above, its axis in $X$ fellow
  travels a long segment of $[\widetilde x,\widetilde y]$ and
  involves many fundamental domains. If $g'$ is another such element,
  the two long segments of $[\widetilde x,\widetilde y]$ fellow
  traveled by the axes of $g$ and $g'$ have to overlap in a
  segment that fellow travels both axes along many fundamental domains,
  by the alignment-preserving property of $f$ and the fact that the
  images coarsely contain $[x,y]$ except for bounded segments at each
  end.

  Acylindricity of the action of $G$ on $X$ then forces the
  set of such $g$ to be contained in a virtually cyclic group,
  which is cyclic since $G$ is assumed torsion-free
  (see e.g. \cite[Proposition 6]{BF}). Now
  it follows that the set of such $g$ is
  uniformly bounded by $2\left(\frac\epsilon{2\delta}\right)+1$.
\end{proof}

\setcounter{section}{0}
\setcounter{thm}{0}
\renewcommand{\thesection}{\Alph{section}}

\section{Appendix: Distance bounds}
\label{sec:appendix}
Let $\tau_1$ and $\tau_2$ be obtained from $\tau$ by splitting and let $\sigma_1\subset \tau_1$ and $\sigma_2\subset \tau_2$ be subtracks. We want to show that the coarse intersection of $\S(\sigma_1)$ and $\S(\sigma_2)$ is contained in a uniform neighborhood of $\S(\sigma_1) \cap \S(\sigma_2)$ (if this intersection is non-empty) or is uniformly bounded or empty (if the intersection is empty). When $\tau_1 = \tau_2 = \tau$ this is essentially Proposition \ref{3.24}.

The proof of the more general statement will be an inductive argument. The inductive step will come from the case when $\tau_1$ and $\tau_2$ are the two splittings of $\tau$ along some large branch.
In particular, the following was part of Corollary 3.24 in \cite{BB}:
\begin{prop}\label{3.24b}
Given $C>0$ there exists $C'>0$ such that the following holds.
Let $\tau_1$ and $\tau_2$ be the two splittings of a train track $\tau$ along some large branch. Assume that $\alpha_1 \in \S(\tau_1)$ and $\alpha_2 \in \S(\tau_2)$ with $d(\alpha_1, \alpha_2) \le C$. Then either
\begin{enumerate}[(1)]
\item $d(\alpha_1, \tau), d(\alpha_2, \tau) \le C'$, or

\item there exists $\gamma \in \S(\tau_1) \cap \S(\tau_2)$ with $d(\gamma, \alpha_1),  d(\gamma, \alpha_2) \le C'$.
\end{enumerate}
\end{prop}

If $\alpha \in \S(\tau)$ we let $\tau(\alpha) \subset \tau$ be the smallest subtrack that carries $\alpha$. One very useful observation about this definition is that if $\tau'\to\tau$ then for any curve $\alpha \in \S(\tau')$, $\tau'(\alpha)\to\tau(\alpha)$ and thus $\S(\tau'(\alpha)) \subset \S(\tau(\alpha))$.

 The following lemma is derived from Propositions \ref{3.24} and \ref{3.24b}.

\begin{lemma}\label{nested_induction}
For all $C>0$ there exists $C'>0$ such that the following holds. Let $\tau$ be a train track and $\tau'\ss\tau$ a single split. Assume that $\alpha \in \S(\tau)$ and $\alpha' \in \S(\tau')$ with $d(\alpha, \alpha') \le C$. Then either
\begin{enumerate}[(1)]
\item $d(\alpha, \tau), d(\alpha', \tau) \le C'$ or;
\item $\alpha \in \S(\tau'(\alpha'))$ or;
\item there exists $\gamma \in \S(\tau(\alpha)) \cap \S(\tau'(\alpha'))$ with $d(\alpha, \gamma), d(\alpha',\gamma) \le C'$ and $\ind{\tau'}(\gamma)> \ind{\tau}(\alpha)$.
\end{enumerate}
\end{lemma}

\begin{proof}
If $\alpha \in \S(\tau')$ the lemma follows from Proposition \ref{3.24} with the additional observation that if (1) and (2) don't hold then $\tau'(\alpha)$ is a carried by proper subtrack of $\tau(\alpha)$ so $\ind \tau'(\alpha) > \ind \tau(\alpha)$.

Now assume that $\alpha \not\in \S(\tau')$. We can also assume that $\tau'$ is the right split and let $\tau''$ be left split. Then $\alpha \in \S(\tau'')$.
 We then apply Proposition \ref{3.24b} and see that we are in case (1) or there is $\gamma'$ carried by the central split with both $d(\alpha, \gamma')$ and $d(\alpha', \gamma')$ bounded. Next we apply Proposition \ref{3.24} to $\alpha$, $\gamma'$ and $\tau''$. Again we are either in case (1) or there is a $\gamma'' \in \S(\tau''(\alpha)) \cap \S(\tau''(\gamma'))$ with $d(\alpha, \gamma'')$ and $d(\gamma', \gamma'')$ bounded. We also note that $\ind\tau''(\gamma'') > \ind \tau''(\alpha)$ for otherwise $\alpha$ would be carried by the central splits, and hence by $\tau'$, a contradiction.  As $\gamma''$ is carried by the central split  we have $\gamma'' \in \S(\tau')$. The proof is then completed with another application of Proposition \ref{3.24}.
\end{proof}

The following lemma follows quickly from Theorem \ref{quasigeodesics}, by extending
a splitting sequence $\tau_2 \ss \tau_1$:

\begin{lemma}\label{split_to_quasi}
There exists $C = C(\Sigma)>0$ such that the following holds. Let $\tau_1$ and $\tau_2$ be train tracks with $\tau_2 \ss \tau_1$. If $\alpha \in \S(\tau_2)$ then $d(\alpha, \tau_2) \le d(\alpha, \tau_1) + C$. In particular if $\tau_3$ is another train track with $\tau_3 \ss\tau_2$ then $d(\tau_2, \tau_3) \le d(\tau_1,\tau_3) + C$.
\end{lemma}

The following proposition is a combination of Proposition 3.19 in \cite{BB} and Lemma \ref{diagonal extension}.
\begin{prop}\label{intersection}
Assume that $\tau_1 \ss \tau$ and $\tau_2 \ss \tau$ and let $\sigma_1 \subset \tau_1$ and $\sigma_2\subset \tau_2$ be subtracks with $\S(\sigma_1)\cap \S(\sigma_2) \neq \emptyset$. Then there exist tracks $\eta_1$ and $\eta_2$ with a common subtrack $\eta$ such that $\eta_1 \ss \tau_1$, $\eta_2\ss\tau_2$, and $\S(\eta) = \S(\sigma_1) \cap \S(\sigma_2)$. Furthermore, $\min\{d(\tau_1, \eta), d(\tau_2, \eta)\}$
is uniformly bounded.
\end{prop}

Before proving our main estimate we need the following special case:
\begin{prop}\label{nested}
Given $C>0$ there exists $C'>0$ such that the following holds. Assume that $\tau' \ss\tau$ and let $\sigma \subset \tau$ and $\sigma' \subset \tau'$ be subtracks. If $\alpha \in \S(\sigma)$ and $\alpha'\in\S(\sigma')$ with $d(\alpha, \alpha') \le C$ then either
\begin{enumerate}[(1)]
\item $d(\alpha, \sigma'), d(\alpha', \sigma') \le C'$, or
\item there exists $\gamma \in \S(\sigma) \cap \S(\sigma')$ with $d(\alpha, \gamma), d(\alpha', \gamma) \le C'$.
\end{enumerate}
\end{prop}

\begin{proof}

Let
$$\tau' = \tau_m \to \tau_{m-1} \to \cdots \to \tau_0 = \tau$$
be the splitting sequence. 

Let $C_0 = C$ and inductively define $C_{i+1}$ to be the constant $C'$ yielded by Lemma \ref{nested_induction} for $C_i$. Assume we have constructed a sequence of curves $\alpha_0= \alpha, \alpha_1, \dots, \alpha_i$ with the following properties:
\begin{itemize}
\item $\alpha_{j} \in \S(\tau_{j-1}(\alpha_{j-1})) \cap \S(\tau_{j}(\alpha'))$;

\item if $\alpha_{j-1} \neq \alpha_{j}$ then  $\ind{\tau_j}(\alpha_j) > \ind{\tau_{j-1}}(\alpha_{j-1})$;

\item $d(\alpha_j, \alpha') \le C_{\#(j)}$ where $\#(j)$ is the number of times that $\alpha_{\ell-1} \ne \alpha_\ell$ for $1\le \ell \le j$.
\end{itemize}
We then apply Lemma \ref{nested_induction} to the curves $\alpha_i$ and $\alpha'$ and the tracks $\tau_i$ and $\tau_{i+1}$. We see that either $d(\alpha', \tau_i) \le C_{\#(i+1)}$ or there exists a curve $\alpha_{i+1}$ satisfying all the above bullets, so that the sequence may be extended.

Let $d$ be the maximal index of a train track on $\Sigma$. Then $\#(i) \le d$. If the sequence terminates at $i$ we have that $d(\alpha', \tau_i) \le C_{\#(i)} \le C_d$ and Lemma \ref{split_to_quasi} implies that $d(\alpha', \tau')$ is bounded since $\tau' \ss \tau_i$ and $\alpha' \in \S(\tau')$. This also yields a bound on $d(\alpha,\tau')$ since $d(\alpha, \tau') \le d(\alpha, \alpha') + d(\alpha', \tau')$.

Now assume that the sequence continues until $\alpha_m$. Since $\alpha_i \in \S(\tau_{i-1}(\alpha_{i-1}))$ we have that $\tau_i(\alpha_i) \to \tau_{i-1}(\alpha_{i-1})$ giving a sequence
$$\tau_m(\alpha_m) \to \tau_{m-1}(\alpha_{m-1}) \to \cdots \to \tau_0(\alpha_0) = \tau(\alpha)\subset \sigma.$$
In particular $\alpha_m \in \S(\sigma)$. We also have that $\tau'(\alpha') \subset \sigma'$ so $\alpha_m \in \S(\sigma')$. We then let $\gamma = \alpha_m$ and observe that we have shown that $\gamma \in \S(\sigma) \cap \S(\sigma')$ with $d(\gamma, \alpha') \le C_d$ and $d(\gamma, \alpha) \le C + C_d$ so that we may choose $C'$ to be the maximum of $C_d + C$ and the bound on $d(\alpha,\tau'),d(\alpha',\tau')$ from the previous paragraph.
\end{proof}

\begin{lemma}\label{general_induction}
Given $C>0$ there exists $C'>0$ such that the following holds. Let $\tau, \tau'$, and $\tilde\tau$ be train tracks with $\tau'$ obtained by a single split from $\tau$ and $\tilde\tau\ss \tau$. Assume that $\alpha' \in \S(\tau')$ and $\tilde\alpha \in \S(\tilde\tau)$ with $d(\alpha', \tilde\alpha) \le C$. Then either
\begin{enumerate}[(1)]
\item $d(\alpha', \tilde\tau), d(\tilde\alpha, \tilde\tau) \le C'$, or

\item there exists $\gamma \in \S(\tau'(\alpha')) \cap \S(\tilde\tau(\tilde\alpha))$ with $d(\alpha', \gamma), d(\tilde\alpha, \gamma) \le C'$ and if $\gamma \neq \tilde\alpha$ then $\ind \tau'(\gamma) > \ind{\tau}(\tilde\alpha)$.
\end{enumerate}
\end{lemma} 
Note that if $\tilde\tau= \tau$ then this lemma reduces to Lemma \ref{nested_induction}. However the proof is considerably more involved and we will use Proposition \ref{nested} (and hence Lemma \ref{nested_induction}) to prove Lemma \ref{general_induction}.
\begin{proof} We first prove the lemma ignoring the statement about index in (2). We will address this at the end of the proof.

First assume that $\tilde\alpha \in \S(\tau')$. By Proposition \ref{intersection} there is a train track $\eta$ with $\eta \ss \tau'$ and subtrack $\eta' \subset \eta$ such that $\S(\eta') = \S(\tau') \cap \S(\tilde\tau)$ and at least one of $d(\eta, \tau')$ and $d(\eta, \tilde\tau)$ is bounded.
Then, by Proposition \ref{nested}, either $d(\alpha', \eta)$ and $d(\tilde\alpha, \eta)$ are bounded or there exists $\gamma \in \S(\tau'(\alpha'))\cap \S(\eta'(\tilde\alpha))$ with $d(\alpha', \gamma)$ and $d(\tilde\alpha, \gamma)$ bounded. In the latter case we note that $\S(\eta'(\tilde\alpha)) \subset \S(\tilde\tau(\tilde\alpha))$ so we are in case (2) and we are done.

In the former case if $d(\eta, \tilde\tau)$ is bounded we are in case (1) and also done. If not $d(\eta, \tau')$ is bounded and, since $\tau'$ is a single splitting of $\tau$, $d(\eta,\tau)$ is also bounded. As there is a splitting sequence from $\eta'$ to a subtrack of $\tilde\tau$ and that subtrack has a splitting sequence to a subtrack of $\tau$ we can then apply Lemma \ref{split_to_quasi} to get a bound on $d(\alpha', \tilde\tau)$ and $d(\tilde\alpha, \tilde\tau)$.

Now assume $\tilde\alpha \not\in \S(\tau')$. Then $\tilde\alpha \in \S(\tau'')$ where $\tau''$ is the other splitting of $\tau$ along the large branch that we split to form $\tau'$. Then by Lemma \ref{nested_induction} either  $d(\tilde\alpha, \tau)$ is bounded (and hence by Lemma \ref{split_to_quasi} $d(\tilde\alpha,\tilde\tau)$ is bounded) or there is a $\gamma' \in \S(\tau'(\alpha')) \cap \S(\tau''(\tilde\alpha))$ with $d(\gamma', \alpha')$ and $d(\gamma', \tilde\alpha)$ bounded. In the former case, (1) holds.

If the latter holds, we are done if $\gamma' \in \S(\tilde\tau(\tilde\alpha))$. If this doesn't hold, Proposition \ref{intersection} yields a track $\eta \ss \tau''$ with a subtrack $\eta' \subset \eta$ such that $\S(\eta') = \S(\tilde\tau) \cap \S(\tau'')$. Proposition \ref{nested} then shows that either $d(\tilde\alpha, \tilde\tau)$ is bounded or there is $\gamma \in \S(\tau''(\gamma')) \cap \S(\eta'(\tilde\alpha))$ with $d(\gamma, \gamma')$ and $d(\gamma, \tilde\alpha)$ bounded.
As $\gamma'$ is carried by both $\tau'$ and $\tau''$ we have that $\tau'(\gamma') = \tau''(\gamma')$. As $\tilde\tau$ carries $\eta'$ we also have $\S(\eta'(\tilde\alpha)) \subset \S(\tilde\tau(\tilde\alpha))$ and therefore $\gamma \in \S(\tau'(\gamma')) \subset \S(\tau'(\alpha'))$ and case (2) holds.

We now show that if were are in case (2) and $\gamma \neq \tilde\alpha$ then $\ind \tau'(\gamma) > \ind{\tau}(\tilde\alpha)$. There are two cases. We first assume that $\tilde\alpha \in \S(\tau')$. If $\S(\tau'(\tilde\alpha)) \subset \S(\tau'(\alpha'))$ then we can choose $\gamma = \tilde\alpha$ and there is nothing to show. If not, $\tau'(\gamma)$ is a proper subtrack of $\tau'(\tilde\alpha)$ so $\ind \tau'(\gamma) > \ind \tau'(\tilde\alpha) \ge \ind\tau(\tilde\alpha)$. If $\tilde\alpha \not\in \S(\tau')$ then $\tilde\alpha \in \S(\tau'')$ with $\tau''(\tilde\alpha)$ not a subtrack of $\tau'$. Since $\tau''(\gamma) = \tau'(\gamma)$ is a subtrack of $\tau'$ this implies that $\tau'(\gamma)$ is a proper subtrack of $\tau''(\tilde\alpha)$ and $\ind \tau'(\gamma) > \ind \tau''(\tilde\alpha) \ge \ind\tau(\tilde\alpha)$.
\end{proof} 

\begin{prop}\label{general}
Given $C>0$ there exists $C'>0$ such that the following holds. Assume that $\tau^1 \ss \tau$ and $\tau^2 \ss \tau$ and let $\sigma^1 \subset \tau^1$ and $\sigma^2\subset \tau^2$ be subtracks. If $\alpha^1 \in \S(\sigma^1)$ and $\alpha^2\in \S(\sigma^2)$ with $d(\alpha^1, \alpha^2) \le C$ then either
\begin{enumerate}[(1)]
\item $\min\{d(\alpha^1, \sigma^1), d(\alpha^2, \sigma^2)\} \le C'$, or
\item there exists $\gamma \in \S(\sigma^1) \cap \S(\sigma^2)$ with $d(\alpha^1, \gamma), d(\alpha^2, \gamma) \le C'$.
\end{enumerate}
\end{prop}

\begin{proof} By Proposition \ref{intersection} there are tracks $\eta^1$ and $\eta^2$ with $\eta^i \ss \tau^i$ and the $\eta^i$ contain a common subtrack $\eta$ with $\S(\eta) = \S(\tau^1) \cap \S(\tau^2)$. Furthermore, we can assume that $d(\eta, \tau^2)$ is uniformly bounded.

The proof is similar to the proof of Proposition \ref{nested}.
Let
$$\tau^1 = \tau^1_m \to \cdots \to \tau^1_0 = \tau$$
be a splitting sequence for $\tau^1 \ss \tau$. This will play the role of the splitting sequence from $\tau'$ to $\tau$ in Proposition \ref{nested}. We also choose, using Proposition \ref{intersection}, train tracks $\tau^2_i \ss \tau^1_i$ such that there is a subtrack $\eta_i \subset \tau^2_i$ with $\S(\eta_i) = \S(\tau^1_i) \cap \S(\tau^2)\supset \S(\eta)$.

As before set $C_0 = C$ and let $C_{i+1}$ be the constant $C'$ yielded by Lemma \ref{general_induction} for $C_i$. Now assume we have constructed a sequence of curves $\alpha^2_0 = \alpha^2, \alpha^2_1, \dots, \alpha_i^2$ with the following properties:
\begin{itemize}
\item $\alpha^2_j \in \S(\tau^2_{j-1}(\alpha^2_{j-1})) \cap \S(\tau^1_j(\alpha^1))$;

\item if $\alpha^2_{j-1} \neq \alpha^2_j$ then $\ind \tau^1_j(\alpha^2_j) > \ind \tau^1_{j-1}(\alpha^2_{j-1})$;

\item $d(\alpha^1, \alpha^2_j) \le C_{\#(j)}$ where, as before, $\#(j)$ is the number of times that $\alpha^2_{j-1} \neq \alpha^2_j$.

\end{itemize}
Applying Lemma \ref{general_induction} to $\alpha^1$ and $\alpha^2_i$ and the tracks $\tau^1_i$, $\tau^1_{i+1}$ and $\tau^2_i$, we can either continue the sequence to $\alpha^2_{i+1}$ or the sequence terminates and $d(\alpha^1, \tau^2_i) \le C_{\#(i+1)}$.

First assume the sequence ends at $i$. Then $d(\alpha^1, \tau^2_i)\le C_d$ where $d$ is the maximal index. 
By Proposition \ref{intersection} there is a track $\eta'$ with $\S(\eta') = \S(\eta) = \S(\tau_i^2)$ and $\eta' \to \eta$. As $\S(\eta) \subset \S(\tau^2_i)$ we have $\S(\eta') = \S(\eta)$. In particular, both tracks have the same set of vertex cycles.
A subtrack of $\tau^2_i$ is carried by $\tau^2$ so the vertex cycles of $\eta'$, $\tau^2_i$, and $\tau^2$ lie on a quasi-geodesic, in that order. In particular the bound on $d(\eta, \tau^2) = d(\eta', \tau^2)$ implies that $d(\tau^2_i, \tau^2)$ is bounded and hence $d(\alpha^1, \tau^2)$ is bounded and we are in case (1).

If the sequence continues to $\alpha^2_m$ we let $\gamma = \alpha^2_m$ and then, similarly as in the proof of Proposition \ref{nested}, $\gamma \in \S(\sigma^1) \cap \S(\sigma^2)$ with $d(\gamma, \alpha^1)$ and $d(\gamma,\alpha^2)$ uniformly bounded.
\end{proof}

\begin{prop}\label{curve version}
Given $C>0$ there exists $C' = C'(\Sigma, m,C)>0$ such that the following holds. Let $\tau, \tau_1, \dots, \tau_m$ be train tracks with $\tau_i \ss \tau$ and let $\sigma_1 \subset \tau_1, \dots, \sigma_m \subset \tau_m$ be subtracks. Assume that $\alpha_1 \in \S(\sigma_1), \dots, \alpha_m \in \S(\sigma_m)$ with $d(\alpha_i, \alpha_j) \le C$ for all $i, j$. Then either
\begin{enumerate}[(1)]
\item $\min\{d(\alpha_1, \sigma_1), \dots, d(\alpha_m, \sigma_m)\} \le C'$, or
\item there exists $\tilde\alpha \in \S(\sigma_1) \cap \cdots \cap \S(\sigma_m)$ with $d(\alpha_i, \tilde\alpha) \le C'$ for all $i$.
\end{enumerate} 
\end{prop}

\begin{proof}
If for any $i \neq j$ we have $\S(\sigma_i) \cap \S(\sigma_j) = \emptyset$ then by Proposition \ref{general} we have that
$$\min\{d(\alpha_1, \sigma_1), \dots, d(\alpha_m, \sigma_m)\} \le \min\{d(\alpha_i, \sigma_i), d(\alpha_j, \sigma_j)\} \le C'$$
where $C'$ is the constant from Proposition \ref{general}.

Therefore we can assume that $\S(\sigma_1) \cap \S(\sigma_2) \neq \emptyset$. Then, by Proposition \ref{general}, either (1) holds or there is an $\alpha' \in \S(\sigma_1) \cap \S(\sigma_2)$ with both $d(\alpha_1, \alpha')$ and $d(\alpha_2, \alpha')$, and hence $d(\alpha', \alpha_i)$, bounded. 
By Lemma \ref{intersection} there exists $\tau' \ss \tau$ and a subtrack $\sigma' \subset \tau'$ with $\S(\sigma') = \S(\sigma_1) \cap \S(\sigma_2)$
and $\min\{d(\sigma_1,\sigma'),d(\sigma_2,\sigma')\}$ bounded. By induction we can assume that proposition holds for $\alpha' \in \S(\sigma'), \alpha_3, \in \S(\sigma_3), \dots, \alpha_m \in \S(\sigma_m)$. If (1) holds than one of $d(\alpha', \sigma'), d(\alpha_3, \sigma_3), \dots, d(\alpha_m, \sigma_m)$ is uniformly bounded. In the case when only $d(\alpha', \sigma')$ is bounded, as $\min\{d(\sigma_1, \sigma'), d(\sigma_2, \sigma')\}$ is bounded, we can assume that, say, $d(\sigma_1, \sigma')$ is bounded. Since $d(\alpha_1, \alpha')$ is bounded we have that $d(\alpha_1, \sigma_1)$ is bounded and we are in case (1).  If not there is a
$$\tilde\alpha \in \S(\sigma') \cap \S(\sigma_3) \cdots \S(\sigma_m) = \S(\sigma_1) \cap \S(\sigma_2) \cap \cdots \S(\sigma_m)$$
with 
$$d(\tilde\alpha, \alpha'), d(\tilde\alpha,\alpha_3), \dots, d(\tilde\alpha,\alpha_m)$$
all bounded. Since $d(\alpha_1,\alpha')$ and $d(\alpha_2, \alpha')$ are bounded this implies that $d(\tilde\alpha, \alpha_1)$ and $d(\tilde\alpha, \alpha_2)$ are bounded. Therefore (2) holds.
\end{proof}

\begin{prop}\label{alternative}
  Given $C>0$ there exists $C' = C'(\Sigma, m, C)>0$
  such that the following holds. Let $\tau, \tau_1, \dots, \tau_m$ be train tracks with $\tau_i \ss \tau$ and let $\sigma_1 \subset \tau_1, \dots, \sigma_m \subset \tau_m$ be subtracks. Then either
\begin{enumerate}[(1)]
\item the diameter of $I_C(\S(\sigma_1), \dots, \S(\sigma_m))$ is $\le C'$, or
\item $I_C(\S(\sigma_1), \dots, \S(\sigma_m)) \subset B_{C'}(\S(\sigma_1) \cap \cdots \cap\S(\sigma_m))$.
\end{enumerate}
\end{prop}

\begin{proof}
Let ${\bf I} = I_C(\S(\sigma_1), \dots, \S(\sigma_m))$, ${\bf \tilde I} =  \S(\sigma_1) \cap \cdots \cap \S(\sigma_m)$ and $V_i = V(\sigma_i)$.

  We claim there is $C_1=C_1(\Sigma,m,C)$ so
  that $${\bf I} \subset B_{C_1}\left({\bf \tilde I}\right) \cup \bigcup_i B_{C_1}(V_i).$$

  {\it Proof of Claim.} Let $\beta\in {\bf I}$.
Thus there are $\alpha_i\in S(\sigma_i)$ so that
  $d(\alpha_i,\beta)\leq C$. Let $C'$ be the constant from Proposition
  \ref{curve version} with $2C$ as $C$. Since $d(\alpha_i,\alpha_j)\leq
  2C$ it then follows that either
  $d(\alpha_i,\sigma_i)\leq C'$ for some $i$
  or there is $\tilde\alpha\in
  \S(\sigma_1) \cap \cdots \cap\S(\sigma_m)$ with
  $d(\alpha_i,\tilde\alpha)\leq C'$ for all $i$.
  In the first case it follows that
  $d(\beta,\sigma_i)\leq C+C'$ and in the second it follows that
  $\beta\in B_{C+C'}\left({\bf \tilde I}\right)$.
  Either way, the claim is proved with $C_1=C+C'$.

By Theorem \ref{quasigeodesics} there is a $Q>0$ such that the sets $S(\sigma_i)$ are $Q$-quasi-convex. It suffices to assume that $C \ge Q$ and then by \eqref{intersection quasiconvex} of  Lemma \ref{fellow travel} we have that ${\bf I}$ is $C+2\delta$-quasi-convex and hence $2C+4\delta + 1$-connected. By the claim ${\bf I}$ is contained in the union of the $m$ sets $B_{C_1}(V_i)$, each of which has uniformly bounded diameter, and the set  $B_{C_1}({\bf \tilde I})$. If ${\bf \tilde I} = \emptyset$ then this implies that the diameter of ${\bf I}$ is uniformly bounded while if ${\bf \tilde I} \neq \emptyset$ then ${\bf I}$ is contained in a uniform neighborhood of ${\bf \tilde I}$. In the former case (1) holds while in the latter (2) holds.
\end{proof}

\begin{lemma}
Fix constants $Q$ and $C' \ge Q+2\delta$ and let $X$ be a $\delta$-hyperbolic graph. Let $\A$ be a collection of $Q$-quasi-convex subsets of $X$ that is closed under finite intersections. Assume that there exists $D_X>0$ such that for any $A_1, \dots, A_m \in \A$ either
\begin{enumerate}[(1)]
\item $\diam_X I_{C'}(A_1,\dots, A_m) \le D_X$ or;
\item $I_{C'}(A_1, \dots, A_m) \subset B_{D_X}(A_1\cap \cdots \cap  A_m)$.
\end{enumerate}
Let $Y$ be a $\delta$-hyperbolic graph and $f\colon X\to Y$ an injective, $L_0$-Lipschitz map. Also assume that there is $L_1 >0$ such that  for all $x_1, x_2 \in X$ we have
$$f([x_1,x_2]) \subset B_{L_1}([f(x_1), f(x_2)]).$$
Then for any $C>0$ there is $D_Y = D_Y(D_X,C, L_0, L_1,m)$  such that:
\begin{enumerate}[(A)]
\item $\diam_Y I_C(f(A_1),  \dots, f(A_m)) \le D_Y$ or;

\item  \begin{eqnarray*}
I_C(f(A_1), \dots, f(A_m)) & \subset& B_{D_Y}(f(A_1\cap \cdots \cap A_m)) \\ &\subset& B_{D_Y}(f(A_1)\cap \cdots \cap  f(A_m)).
\end{eqnarray*}
\end{enumerate}
\end{lemma}

\begin{proof}
We first assume that there are only two sets $A_1$ and $A_2$. The general case will follow by induction.

Let $a_1, b_1 \in A_1$ and $a_2, b_2 \in A_2$ and fix $c$ and $d$ in the geodesics $[a_1, a_2]$ and $[b_1,b_2]$ such that $d_X(c,d)$ minimizes the distance between the two geodesics. In particular, this implies that $B_{2\delta}([a_1, c]) \cap [c,d] \subset B_{2\delta}(\{c\})$ with a similar statement for $b_1$ and $d$. By $\delta$-hyperbolicity we have 
$$[c,d] \subset B_{2\delta}([a_1, c] \cup [a_1, b_1] \cup [b_1, d])$$
so if $d_X(c,d) > 4\delta$ we have $[c,d] \subset B_{4\delta}([a_1,b_1])$ and therefore $[c,d] \subset B_{Q+4\delta}(A_1)$. A similar statement holds for $A_2$ so if $d_X(c,d) > 4\delta$ we have $[c,d] \subset I_{Q+4\delta}(A_1, A_2)$. 

Assume we are in case (1) with $x,y \in I_C(f(A_1), f(A_2))$. Then there are $a_1, b_1 \in A_1$ and $a_2, b_2 \in A_2$ with $d_Y(f(a_1), x), d_Y(f(a_2), x), d_Y(f(b_1), y), d_Y(f(b_2), y) < C$. We choose $c$ and $d$ as above
and observe that
$$d_Y(x, f(c)), d_Y(y, f(d)) \le 3C+L_1$$
and therefore
\begin{eqnarray*}
d_Y(x,y) &\le &  d_Y(f(c), f(d)) + d_Y(x, f(c)) + d_Y(y, f(d))\\ & \le & L_0 D_X + 6C + 2L_1 = D_Y.
\end{eqnarray*}
This proves that (A) holds.

Now assume we are in case (2) with $x \in I_C(f(A_1), f(A_2))$ and $a_1,a_2$ and $c$ as before. Again we have $d_Y(x, f(c)) \le 3C+L_1$. Then fix $d \in A_1\cap A_2$ and set $b_1 = b_2 = d$. Therefore if $d_X(c,d) > 4\delta$ then $c \in I_{Q+ 4\delta}(A_1, A_2)$ and, as (2) holds, we can further assume that $d$ has been chosen such that $d_X(c,d) \le D_X$. We then have $d_X(c,d) \le \min\{ 4\delta, D_X\} = D_X$ and it follows that
\begin{eqnarray*}
d_Y(x, f(d)) & \le & d_Y(x, f(c)) + d_Y(f(c), f(d)) \\ & \le & 3C+ L_1 + L_0 D_X = D_Y.
\end{eqnarray*}
Therefore (B) holds.

Now assume that the lemma holds for any collection of $m-1$ sets in $\A$ and let $A_1, \dots, A_m$ be sets in $\A$.  If $A_1 \cap A_2 = \emptyset$ then we are in case (1) for the sets $A_1$ and $A_2$ so
$$I_C(f(A_1), \cdots, f(A_2)) \subset I_C(f(A_1), f(A_2))$$
has diameter $\le D^2_Y= D_Y(D_X, C, C', L_0, L_1, 2)$. Therefore we can assume that $A_1 \cap A_2 \neq \emptyset$. Since $f$ is injective we have $f(A_1 \cap A_2) = f(A_1) \cap f(A_2)$ and
\begin{eqnarray*}
I_C(f(A_1), \dots, f(A_m)) &=& I_C(f(A_1), f(A_2)) \cap I_C(f(A_3), \dots, f(A_m)) \\
& \subset & B_{D^2_Y}(f(A_1) \cap f(A_2)) \cap I_{D^2_Y}(f(A_3), \dots, f(A_m)) \\
& \subset & B_{D^2_Y}(f(A_1 \cap A_2)) \cap I_{D^2_Y}(f(A_3), \dots, f(A_m)) \\
& \subset & I_{D^2_Y}(f(A_1 \cap A_2),f(A_3), \dots, f(A_m)).
\end{eqnarray*}
Therefore $D_Y(D_X, C,C', L_0, L_1, m) = D_Y(D_X, D^2_Y,C', L_0, L_1, m-1)$ is the desired constant.
\end{proof}

Combining this lemma with Propositions \ref{KR} and \ref{hyperbolic} we have:

\begin{prop}\label{k alternative}
  Proposition \ref{alternative} holds as stated for the distance $d_k$
  in $\C_k$.
\end{prop}

\bibliographystyle{plain}
\bibliography{disintegrating}

\end{document}